\documentclass{article}[12pt]
\usepackage{amsmath,amsxtra,amssymb,latexsym, amscd,amsthm,enumitem}
\usepackage{graphicx}
\usepackage{pgf,tikz}
\usetikzlibrary{arrows}
\usepackage{hyperref}
\usepackage{pgf,tikz}
\usetikzlibrary{arrows}
\usepackage[mathscr]{eucal}
\usepackage[all]{xy}
\usepackage[portrait, top=2cm, bottom=2cm, left=2cm, right=2cm] {geometry}
\setenumerate[1]{label={\rm(\arabic*)}}

\numberwithin{equation}{section}
\numberwithin{figure}{section}

\theoremstyle{plain}

\newtheorem{thm}{Theorem}[section]
\newtheorem{lem}[thm]{Lemma}
\newtheorem{prop}[thm]{Proposition}
\newtheorem{cor}[thm]{Corollary}

\newtheorem{Asse}{Assertion}[section]
\newtheorem*{Asse0}{Claim}

\newtheorem{Clm}{Claim}

\theoremstyle{definition}
\newtheorem{deff}[thm]{Definition}
\newtheorem{rem}[thm]{Remark}

\newcommand{\hh}{\mathbb H}
\newcommand{\dhr}{\partial}

\newcommand{\Pcal}{\mathcal{P}}
\newcommand{\Acal}{\mathcal{A}}

\newcommand{\Vcal}{\mathcal{V}}

\newcommand{\vh}[1]{\left\langle #1\right\rangle}
\newcommand{\td}[1]{\left\| #1\right\|}

\newcommand{\Hcal}{\mathcal{H}}

\newcommand{\Nbb}{\mathbb{N}}

\DeclareMathOperator{\tr}{tr}

\newcommand{\Index}{\operatorname{Index}}

\newcommand{\PSLhR}{\widetilde{{\rm PSL}}_2(\Rbb,\tau)}

\newcommand{\Ebb}{\mathbb{E}}

\newcommand{\Rbb}{\mathbb{R}}

\newcommand{\Hbb}{\mathbb{H}}
\newcommand{\Hbbh}{\mathbb{H}^2}

\newcommand{\HbbhR}{\mathbb{H}^2\times\Rbb}

\newcommand{\Sbb}{\mathbb{S}}

\newcommand{\Mbb}{\mathbb{M}}

\newcommand{\Nsf}{\mathsf{N}}
\newcommand{\Tsf}{\mathsf{T}}

\newcommand{\Nil}{{\rm Nil}}
\newcommand{\Sol}{{\rm Sol}}

\newcommand{\Ccal}{\mathcal{C}}

\newcommand{\Ebbbkt}{\Ebb^3(\kappa,\tau)}

\newcommand{\chuan}[1]{\left\| #1 \right\|}

\newcommand{\nablabar}{\overline{\nabla}}

\newcommand{\dhrinfty}{\partial_\infty}

\newcommand{\dhriHh}{\partial_\infty\mathbb{H}^2}

\newcommand{\Div}{\operatorname{div}}

\newcommand{\dist}{\operatorname{dist}}
\newcommand{\Gr}{\operatorname{Gr}}
\newcommand{\Ric}{\operatorname{Ric}}
\newcommand{\Area}{\operatorname{Area}}

\newcommand{\dhrou}[2]{\frac{\partial #1}{\partial #2}}
\newcommand{\dhro}[1]{\frac{\partial}{\partial #1}}

\newcommand{\Sigmas}{\Sigma^{\mathrm{s}}}
\newcommand{\Sigmau}{\Sigma^{\mathrm{u}}}

\newcommand{\Lrm}{\operatorname{L}}
\newcommand{\Fcal}{\mathcal{F}}
\newcommand{\Sigmatilde}{\widetilde{\Sigma}}

\newcommand{\Sigmacheck}{\check{\Sigma}}
\newcommand{\Acalcheck}{\check{\mathcal{A}}}

\newcommand{\ellH}[1]{\ell_{\mathbb{H}^2}\left(#1\right)}
\newcommand{\Scal}{\mathcal{S}}

\newcommand{\Fcalh}{\mathcal{F}^{h}}

\newcommand{\vc}{\infty}

\newcommand{\Sigmacheckvc}{\check{\Sigma}_\infty}
\newcommand{\Sigmachecks}{\check{\Sigma}^\mathrm{s}}
\newcommand{\Sigmachecku}{\check{\Sigma}^\mathrm{u}}
\newcommand{\Gammacheck}{\check{\Gamma}}
\newcommand{\pcheck}{\check{p}}
\newcommand{\pcheckvc}{\check{p}_\infty}
\newcommand{\vvvc}{w_\infty}
\newcommand{\Gammacheckvc}{\check{\Gamma}_\infty}

\newcommand{\Dist}{\operatorname{dist}}

\newcommand{\Sigmatildevc}{\widetilde{\Sigma}_\infty}
\newcommand{\mutilde}{\widetilde{\mu}}
\newcommand{\ella}[2]{\ell_{#1}\left(#2\right)}
\newcommand{\Distaa}[2]{\operatorname{dist}_{#1}\left(#2\right)}
\newcommand{\Scals}{\mathcal{S}^\mathrm{s}}

\newcommand{\gammahat}{\widehat{\gamma}}
\newcommand{\Omegabar}{\overline{\Omega}}
\newcommand{\Rest}[2]{\left.#1\right|_{#2}}

\newcommand{\zsf}{\mathsf{z}}

\newcommand{\Omegahat}{\widehat{\Omega}}
\newcommand{\etahat}{\widehat{\eta}}

\begin{document}

\title{\bfseries \LARGE Construction of minimal annuli in $\PSLhR$ via a variational method}
\author{Pascal COLLIN, Laurent HAUSWIRTH, Minh Hoang NGUYEN}
\date{}
\maketitle

\begin{abstract}
We construct complete, embedded minimal annuli asymptotic to vertical planes in the Riemannian $3$-manifold $\PSLhR$. The boundary of these annuli consists of $4$ vertical lines at infinity. 
They are constructed by taking the limit of a sequence of compact minimal annuli. The compactness is obtained from an estimate of curvature which uses foliations by minimal surfaces. This estimate is  independent of the index of the surface. We also prove the existence of a one-periodic family of Riemann's type examples. The difficulty of the construction comes from the lack of symmetry of the ambient space $\PSLhR$.
\end{abstract}


\section{Introduction}
Recently, much attention was drawn to surfaces in simply connected homogeneous 3-maniflods. Simply connected homogenous manifolds with a 4-dimensional isometry group is a two-parameter family denoted by 
$\Ebb(\kappa, \tau)$. When $\tau=0$ and $\kappa =-1$, we  get the product space $\Hbbh \times \Rbb$, where  $\Hbbh$ is the hyperbolic plane. When $\tau= 1/2$ and $\kappa = 0$, $\Ebb(0, 1/2)$ is the 3-dimensional Heisenberg group $\Nil_3$ endowed with a left-invariant metric; otherwise for $\kappa =-1$ and  $\tau \neq 0$, we obtain the universal cover of the Lie group ${\rm PSL}_2(\Rbb,\tau)$ endowed with a left-invariant metric; we will denote it by  $\PSLhR$. One of the main property of these spaces is the existence of a projection $\pi :\Ebb(\kappa, \tau) \to \Mbb  (\kappa)$ which is a Riemannian submersion over a simply connected 2-dimensional surface with constant curvature $\kappa$. Each fiber of the submersion $\pi$ is a vertical geodesic of $\Ebb(\kappa, \tau)$. If $\gammahat$ denote a complete geodesic of $\Mbb  (\kappa)$, the lift $\pi^{-1} (\gammahat)$ is a minimal vertical plane that we will denote by $\gammahat \times \Rbb$ in the following.

In \cite{Pyo11} and \cite{MR12}, the authors constructed complete horizontal minimal annuli of finite total curvature in the Riemannian product manifold  $\HbbhR$. These annuli are asymptotic to vertical planes $\gammahat_1 \times \Rbb$ and $\gammahat_2 \times \Rbb$ where $ \gammahat_1$ and $\gammahat_2 $ are close enough ultraparallel geodesic of $\Hbbh$ (without common point at infinity).  These constructions use minimal graphs  in $ \HbbhR$ over unbounded polygonal domains (see \cite{CR10}), the existence of a conjugate minimal surface to these graphes (see \cite{HST}) and a Schwarz reflection principle. Higher topological examples of finite total curvature in $ \HbbhR$ has been constructed by Martin, Mazzeo and Rodriguez in \cite{MMR} by attaching handles to a finite number of vertical planes using gluing method. All these methods use, at some point, symmetries of the product ambient space $\HbbhR$ and Alexandrov moving plane technique to reduce the difficulty of the problem.

The isometry group of $\PSLhR$ is 4-dimensional and has two connected components: the one of isometries that preserve the orientation of the fibers and of the base of the fibration, and the one of isometries that reverse both orientations. 
Comparatively to the case of $\HbbhR$, isometry groups of $\Nil_3$ and $\PSLhR$ have empty negative component i.e. there is no isometry that reverses the orientation of the fiber and preserve the orientation of the base
or preserve the orientation of the fiber, and reverse the orientation of the base. This lack of symmetry  involves new arguments in view to construct horizontal annuli asymptotic to vertical planes at infinity. In $\Nil_3$, Daniel and Hauswirth(\cite{DH09}) constructed  horizontal embedded annuli using  harmonic Gauss maps, Weierstrass formula and by solving a period problem. However, this method is specific to the case of minimal surface in $\Nil_3$. 

In general manifolds, variational methods or degree theory is more adapted to construct complete examples. In this article, we show the existence of properly embedded horizontal minimal annuli asymptotic to vertical planes in $\PSLhR$. Moreover, we prove the existence of a family of one-periodic properly embedded minimal surface of genus zero with an infinite number of ends asymptotic to vertical planes. These examples can be view as a staircase Riemann's example in $\PSLhR$. Additionally, we construct an embedded genus zero minimal surface with three flat ends asymptotic to vertical planes which is not isotopic to a trinoid.

To construct these examples, we consider a sequence of compact annuli $\Sigma_n$ bounded by compact boundary  $\Gamma_n$. For a subsequence of boundary  $\Gamma_{n_k}$ going to infinity, we can find a subsequence of surface $\Sigma_{n_k}$ converging to a complete proper example. We need to prove that the limit is an annulus. The topology may vanish in two ways. A first one is when the curvature blows up. To overcome this difficulty and get an a priori estimate of the curvature on $\Sigma_{n_k}$, we provide a general geometric argument independent of the index of the surface using a blow up of the surface $\Sigma_{n_k}$ accompanied by three minimal foliations transverse to each other.
The second way to loose the topology is when the "neck" of the annulus disappears at infinity. We use a sequence of compact annuli of index one and a careful study of its limit to retain the neck in a compact domain. 

For fixed number $\tau \in \Rbb$ , the space $\PSLhR$ can be viewed as the smooth 3-manifold

\begin{equation*}
\PSLhR=\left\{(x,y,z)\in\Rbb^3 ; x^2+y^2<4\right\}
\end{equation*}
endowed with the Riemannian metric
\begin{equation}\label{equ4.2.1}
g_{\PSLhR}=\lambda^2\left(dx^2+dy^2\right)+\left(\tau\lambda(ydx-xdy)+dz\right)^2
\end{equation}  
with $\lambda=\frac{4}{4-(x^2+y^2)}$, see \cite[Subsection 2.4]{Daniel07} and \cite[\S 4]{Scott83}.
This model of $\PSLhR$ is called \emph{cylinder model} and is related to $\Hbbh \times \Rbb$ with
$$\Hbbh=(\{ (x,y) \in \Rbb^2 ; x^2+y^2 <4 \} , g_{\Hbbh} =  \lambda ^2 (d x^2 + d y^2))$$
by  Riemannian submersions
\[
\begin{array}{l}
\pi : \PSLhR \to \Hbbh, (x,y,z) \mapsto (x,y) \\
\zsf:  \PSLhR \to \Rbb, (x,y,z) \mapsto z.
\end{array}
\]
Furthermore, $\partial _z$ is a unit Killing vector field and simirlarly to the space $\HbbhR$ we can define points at infinity of $\PSLhR$:
\begin{deff}
The set $\partial_\infty \PSLhR : = \{x^2+y^2 = 4 \} \times \Rbb$ is called boundary points at infinity of the cylinder model. For a geodesic $\gammahat$, the set $\gammahat \times \Rbb:=\pi^{-1}(\gammahat)$ is a vertical plane and $\partial_{\infty}( \gammahat \times \Rbb) = \{\widehat{p},\widehat{q}\}\times \Rbb$ is two vertical lines at infinity where $\widehat{p}$ and $\widehat{q}$ are the endpoints of $\gammahat$ at infinity of $\Hbbh$. In the following we will use the cylinder model to define the product $\gammahat \times [-n,n] := (\gammahat \times \Rbb)\cap \{-n \leq z \leq n  \}$.
\end{deff}

Let $\gammahat_1, \gammahat_2$ be two ultraparallel complete geodesics of $\Hbbh$ and $\Omegahat \subset \Hbbh$ be the ideal quadrilateral domain whose ideal boundary $\partial \Omegahat$ is composed by the four complete geodesics $\gammahat_1, \etahat_1, \gammahat_2, \etahat_2$ in this order together with four distincts endpoints $\widehat{p}_1,\widehat{q}_1,\widehat{p}_2,\widehat{q}_2$ in this order at infinity of the cylinder model (see Figure \ref{fig:def1}). The main result of this article is
	\begin{figure}[h!]
		\centering
	\includegraphics[width=0.2\linewidth]{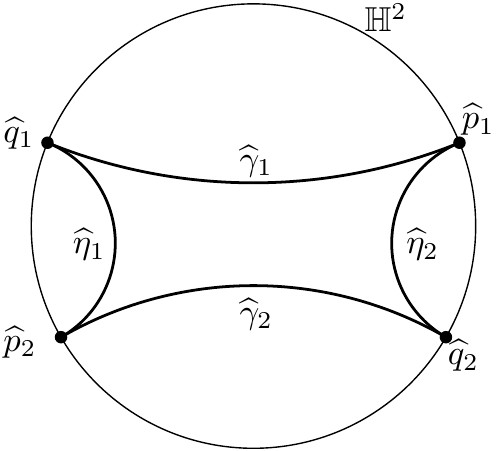}
		\qquad\qquad
		\caption{Ideal domain $\Omegahat$.}
		\label{fig:def1}
	\end{figure}
\begin{thm}
\label{thm1}
If $0<\dist_{\Hbbh} ( \gammahat_1, \gammahat_2) < 2 \ln (\sqrt{2}+1)$, then there exists in $\PSLhR$
\begin{description}
\item{ a)} a complete embedded minimal annulus $\Sigma_1$  whose boundary at infinity is the four vertical lines passing through the end-points of $\gammahat_1, \gammahat_2$,

\item{b)} a complete embedded minimal annulus $\Sigma_2$ whose boundary $\partial \Sigma_2 = \{p_1\} \times \Rbb$
with $p_1$ a point of $\gammahat_1$  and the boundary at infinity is three vertical lines passing through  the end-points $\widehat{q}_1$ of $\gammahat_1$ and $\widehat{p}_2,\widehat{q}_2$ of $\gammahat_2$, 
\item{c)} a complete embedded minimal annulus $\Sigma_3$  whose boundary $\partial \Sigma_3 = \{p_1,p_2\} \times \Rbb$
with $p_1$ (resp. $p_2$) is a point of $\gammahat_1$ (resp. $\gammahat_2$) and the boundary at infinity is two vertical lines passing through  the end-points $\widehat{q}_1$  of $\gammahat_1$ and $\widehat{q}_2$ of  $\gammahat_2$ (or alternatively the point at infinity can be
the end-points $\widehat{q}_1$  of $\gammahat_1$ and $\widehat{p}_2$ of  $\gammahat_2$),
\end{description}

\noindent
such that $\Sigma_1, \Sigma_2, \Sigma_3$ are contained in $\Omegahat \times \Rbb$ and for each geodesic $\gammahat$ that is ultraparallel to both $\gammahat_1$ and 
$\gammahat_2$, the intersection of the annuli $\Sigma_1, \Sigma_2, \Sigma_3$ with $\gammahat \times \Rbb$ is compact. Outside a compact set, these minimal annuli converge uniformly as normal graph to $(\gammahat_1 \times \Rbb) \cup (\gammahat_2 \times \Rbb)$.

In case b), the surface $\Sigma_2$ can be extended using a Schwarz reflection along its  boundary $\{p_1\} \times \Rbb$ to a complete surface without boundary $\widetilde{ \Sigma}_2$ in such a way that $\Sigma_2$ is the quotient
$$\Sigma_2=\widetilde{ \Sigma}_2 / \sigma_1$$
where $\sigma_1$ is a rotation around the vertical geodesic $\{p_1\} \times \Rbb$ and $\widetilde{ \Sigma}_2$ is a surface of genus zero with three ends asymptotic to vertical planes.

In case c), the surface $\Sigma_3$ can be extended using  Schwarz reflections along its boundary $\{p_1\} \times \Rbb$ and $\{p_2\} \times \Rbb$  to a complete surface without boundary $\widetilde{ \Sigma}_3$ in such a way that $\Sigma_3$ is the quotient
$$\Sigma_3=\widetilde{ \Sigma}_3 / \Gamma$$
where $\Gamma$ is the infinite group of isometry generated by $\sigma_1$ and $\sigma_2$ rotations around the vertical geodesics $\{p_1\} \times \Rbb$ and $\{p_2\} \times \Rbb$.
The surface $\widetilde{\Sigma}_3$ is a one-periodic minimal surface of genus zero with countable number of flat ends asymptotic to vertical planes.
\end{thm}

The condition $\dist_{\Hbbh} ( \gammahat_1, \gammahat_2) < 2 \ln (\sqrt{2}+1)$ is equivalent to 
$\dist_{\Hbbh} (\gammahat_1, \gammahat_2) < \dist_{\Hbbh} (\etahat_1, \etahat_2)$. 
This condition is necessary to reach the conclusion of Theorem 1.2. In section 2.2 we prove the Proposition 1.3.

\begin{prop}
\label{propo1.3}
If $\dist_{\Hbbh} ( \gammahat_1, \gammahat_2) \geq 2 \ln (\sqrt{2}+1)$, there does not exist any complete minimal annulus in $\PSLhR$ whose  boundary, finite or at infinity, is contains in the closure of $\gammahat_1\times \Rbb$ and $\gammahat_2\times \Rbb$ such that for each geodesic $\gammahat$ that is ultraparallel to $\gammahat_1, \gammahat_2$  , the intersection of this annulus with $\gammahat \times \Rbb$ is compact.
\end{prop} 

The main step to construct surfaces of the Theorem \ref{thm1} is to construct proper minimal annuli bounded by four vertical lines
into $\Omega \times \Rbb$, where $\Omega\subset \Hbb^2$ is a convex compact quadrilateral domain.

\begin{thm}\label{thm3.3.1}
Let $\Omega\subset\Hbb^2$ be a \emph{bounded} convex quadrilateral domain whose boundary is composed of open geodesic arcs $\gamma_1,\eta_1,\gamma_2, \eta_2$ in this order together with their endpoints
and suppose that 
\begin{equation}\label{equ3.3.1}
	\ellH{\gamma_1}+\ellH{\gamma_2}>\ellH{\eta_1}+\ellH{\eta_2}
\end{equation}
where $\ellH{\cdot}$ is the length of a segment in $\Hbb^2$. Then there exists a proper minimal annulus in $\PSLhR$ whose boundary consists of the four vertical lines passing through the vertex of $\Omega$ such that for each
 complete geodesic $\gammahat$ that meets  $\eta_1$ and $\eta_2$, the intersection of this annulus with $\gammahat \times\Rbb$ is compact. Moreover, this minimal annulus is asymptotic to two vertical slab $ \gamma_1\times\Rbb$ and $\gamma_2\times\Rbb$.
\end{thm}


The content is organized as follows. Section 2 is devoted to the construction of compact minimal annuli by variational methods.
For each positive integer $n$, for $i=1,2$, we define $\Gamma^i_n$ as a closed curve lying in  $ \gamma_i\times\Rbb $ and converging to the two vertical geodesics that pass through the endpoints $p_i,q_i$ of $ \gamma_i $. We then show that for $n$ large enough, there exists a stable minimal annulus $\Sigmas_n$ and hence, via 
degree theory, there exists another  minimal annulus $\Sigmau_n$ with the the same boundary  $\Gamma^1_n\cup\Gamma^2_n$ (see Proposition \ref{pro4.3.6}).
We prove in Section 3, estimate of curvature and some geometric property  of the sequences of minimal surfaces $\Sigmas_n$, $\Sigmau_n$. 
We study in Section 4 the possible limit of these subsequences of compact annuli and conclude to the existence of complete properly embedded minimal annuli bounded by four vertical lines which prove Theorem \ref{thm3.3.1}. Finally, complete minimal annuli of Theorem \ref{thm1}  are constructed in Section 5.

\section{Minimal annuli on bounded domains}
\label{sect1}
\subsection{The Riemannian manifold $\PSLhR$ }

The unit vector field $\xi:=\dhr_z$ in $\PSLhR=\{ (x,y,z) \in \Rbb ; x^2+y^2 <4 \}$ is a complete Killing vector field, called \emph{vertical} and the isometries $\Tsf_h:(x,y,z)\mapsto (x,y,z+h)$ with $h\in\Rbb$ are called \emph{vertical translations}. More generally, for each point $q\in\PSLhR$, we denote by $\Tsf_{q}: \PSLhR\to\PSLhR$ the isometry which preserves the orientation of the fiber which sends $q$ to the origin $O$ of the cylinder model. We will call this isometry a translation.

The vertical vector field $\xi$ satisfies the following equation
\begin{equation}\label{equtau}
	\nablabar_X\xi= \tau X\times\xi
\end{equation}
for all vector fields $X$ in $\PSLhR$. Here $\nablabar$ is the Riemannian connection of $\PSLhR$ and $\times$ is the cross product.

The disk $\left\{(x,y)\in\Rbb^2 ; x^2+y^2<4  \right\}$ endowed with the metric $\lambda^2\left(dx^2+dy^2\right)$ is 
a model of the hyperbolic plane $\Hbbh$. 
The projection $\pi:\PSLhR\to\Hbbh,(x,y,z)\mapsto (x,y)$ is a Riemannian submersion, i.e., if a tangent vector $V$ of $\PSLhR$ satisfies $V\bot \xi$ then $\chuan{V}_{\PSLhR}=\chuan{d\pi(V)}_{\Hbbh}$.
Each fiber of the submersion $\pi$ is called \emph{vertical fiber}.
A tangent vector $V$ of $\PSLhR$ is called \emph{vertical} (resp. \emph{horizontal}) if $V$ is parallel  (resp. orthogonal) to $\xi$. 
Any tangent vector of $\PSLhR$ can be decomposed into  vertical and horizontal parts: $V=V^\Vcal+V^\Hcal$. 
Since $\xi$ is a unit vector field, $V^\Vcal=\vh{V,\xi}\xi$. 
For each vector field $X$ on $\Hbbh$, there exists a unique horizontal vector field $\overline{X}$ of $\PSLhR$ such that $\overline{X}$ is $\pi$-related to $X$. Then $\overline{X}$ is called a \emph{basic horizontal lift} of $X$.
Each orthonormal frame $(e_1,e_2)$ of $\Hbbh$ induces an orthonormal frame $(E_1,E_2,E_3)$ of $\PSLhR$ with $E_1,E_2$  the corresponding basic horizontal lifts of $e_1,e_2$ and $E_3:=\xi$. Let $X_1,X_2$ be two tangent vectors of $\Hbbh$. 
We have the following formulas
\begin{equation}\label{equ1.1.2as}
	\nablabar_{\overline{X_1}}\overline{X_2}=\overline{\nabla_{X_1}X_2}+\tau \vh{\overline{X_2}\times\xi,\overline{X_1}}\xi,\quad \nablabar_\xi\overline{X_1}=\tau\overline{X_1}\times\xi
\end{equation}
where $\nabla$ is the Riemannian connection of $\Hbbh$.
Let $P$ be a $2$-dimensional subspace of $T_p\PSLhR$ and $X,Y$ be an orthonormal basis of $P$. By definition the curvature $K(P)=\vh{R(X,Y)X,Y}$ with $R$ the Riemannian connection of $\PSLhR$. By \cite[Prop 2.1]{Daniel07}, we have
\begin{equation*}
\vh{R(X,Y)X,Y}=\left(\kappa-3\tau^2\right)-\left(\kappa-4\tau^2\right)\left(\vh{\xi,X}^2+\vh{\xi,Y}^2\right)
\end{equation*}
with $\kappa=-1$. We obtain the formula of $K(P)$:
$$	K(P)=\left(1+4\tau^2\right)a-\left(1+3\tau^2\right)$$
where $a$ is square of the length of the orthogonal  projection of the vector $\xi$ on $P$ at $p$.
When $\tau=0$, $\PSLhR$ is a model of the Riemannian product manifold $\Hbbh \times \Rbb$  and the sectional curvature of $\Hbbh \times \Rbb$  is non-positive ($a\leq 1$), while for $\tau \neq 0$, the sectional curvature of $\PSLhR$ changes sign. In particular
\begin{itemize}
\item $K(P)\le \tau^2$ with equality if and only if $\xi\in P$.
\item $K(P)\ge -\left(1+3\tau^2\right)$ with equality if and only if $\xi \bot P$.
\end{itemize}

\subsection{The Jenkins-Serrin theorem}
Let $\Omega\subset \Hbbh $ be a domain. Each function $u:\Omega\to\Rbb$ gives a vertical graph $\Gr(u)$ in $\PSLhR$.
The graph of a $C^2$ function $u$ is minimal if and only if $u$ satisfies the quasilinear elliptic PDE of second order
of divergence form, called \emph{minimal surface equation}  $ \Div_{\Hbbh}(X_u)=0 $
where $-X_u$ is the projection of the upward unit normal $\Nsf$ to $\Gr(u)$ onto $\Omega$.
Jenkins-Serrin type problem is a Dirichlet problem for the minimal surface equation with possible infinite boundary data.
A Jenkins-Serrin type theorem was proved on the bounded domains by Younes \cite{You10} and on some unbounded domains by Melo \cite{Mel14}. Such a theorem for particular case $\HbbhR=\widetilde{\mathrm{PSL}}_2(\Rbb,0)$ had been proved by Nelli-Rosenberg \cite{NR02} for bounded domains and by Collin-Rosenberg \cite{CR10} for some unbounded domains, see more in \cite{MRR11}.

We will recall the Jenkins-Serrin type theorem in $ \PSLhR$. Let $\Omega\subset\hh^2$ be a domain whose boundary $\dhr \Omega$ in $\Hbbh\cup\dhriHh$ consists of a finite number of {open  geodesic arcs} $A_i,B_i$, a finite number of {open, convex   arcs} $C_i$ {(convex towards $\Omega$)} together with their endpoints, which are called the vertices of $\Omega$ (and those in $\dhr_\infty\hh^2$ are called ideal vertices of $\Omega$). We assign to
the edges $A_i$ the data $+\infty$, to the edges $B_i$ the data $-\infty$ and assign arbitrary continuous data $f_i$ on the arcs $C_i$.  Suppose that neither of the two edges $A_i$ and neither of the two edges $B_i$ meet in a convex corner. 
We call a domain $\Omega$ a \emph{Scherk domain}.

A \emph{polygonal domain} $\Pcal$ in $\hh^2$ is a domain whose boundary $\dhr\Pcal$ is composed of finitely many {open  geodesic arcs} in $\hh^2$ together with their endpoints, which are called the vertices of $\Pcal$.
A  polygonal domain $\Pcal$ is said to be inscribed in a Scherk domain $\Omega$ if $\Pcal\subset\Omega$ and its vertices are among the vertices of $\Omega$. We notice that a vertex may be in {$\dhr_\infty\hh^2$}
and an edge may be one of the $A_i$ or $B_i$. Given a bounded polygonal domain $\Pcal$ inscribed in $\Omega$, we denote by $\ell(\Pcal)$ the  perimeter of $\dhr\Pcal$ and by $a(\Pcal)$ and $b(\Pcal)$ the {sum of the}  lengths of the edges $A_i$ and $B_i$ lying in $\dhr\Pcal$, respectively. In the case where $\Pcal$ is an ideal domain (unbounded domain), see Collin and Rosenberg\cite{CR10}   
or Melo\cite{Mel14} to get the adapted definition of $\ell(\Pcal), a(\Pcal),b(\Pcal)$.

\begin{thm}[Jenkins-Serrin type theorem in $\PSLhR$ \cite{You10,Mel14}] \label{thm4.2.6}
	Let $\Omega$ be a Scherk domain in $\hh^2$ with the families $\{A_i \},\{B_i \},\{C_i \}$.
	\begin{enumerate}
	\item If the family $\{C_i\}$ is nonempty, there exists a solution $u: \Omega \to \Rbb$ to the Dirichlet problem if and only if
	\begin{equation}\label{cond1}
	\ell(\Pcal)-2a(\Pcal)>0,\qquad \ell(\Pcal)-2b(\Pcal)>0
	\end{equation}
	for every  polygonal domain {$\Pcal$} inscribed in $\Omega$. Moreover, such a solution $u$ is unique if it exists.
	\item If the family $\{ C_i\}$ is empty, there exists a solution $u: \Omega \to \Rbb$ to the Dirichlet problem if and only if
	\begin{equation}\label{cond2}
b(\Pcal)-a(\Pcal)=0
	\end{equation}
	when $\Pcal=\Omega$ and the inequalities in \eqref{cond1} hold for all other  polygonal domains {$\Pcal$} inscribed in $\Omega$. Such a solution $u$ is unique up to an additive constant, if it exists.\end{enumerate}
\end{thm}



We point out that in the case where $\Omega$ is the quadrilateral domain of Theorem \ref{thm3.3.1} the condition (\ref{equ3.3.1})
is sharp. Theorem \ref{thm4.2.6} imply that there is no minimal annulus bounded by four vertical line in $\Omega \times \Rbb$:



\begin{prop}\label{pro3.3.2}
	If the quadrilateral domain $\Omega$ satisfies
	\begin{equation}\label{equ3.3.2}
		\ellH{\gamma_1}+\ellH{\gamma_2}\le\ellH{\eta_1}+\ellH{\eta_2}
	\end{equation}
	then there does not exist any minimal annulus in $\PSLhR$ whose boundary consists of the four vertical geodesics passing through the endpoints of $\gamma_1$ and $\gamma_2$, such that for each complete geodesic $\gammahat$ that meets both $\eta_1$ and $\eta_2$, the intersection of this annulus with $\gammahat \times\Rbb$ is compact.
\end{prop}

\begin{proof}
	Suppose instead that there is a  minimal annulus $ \Acal $ in $ \PSLhR $ asymptotic to the vertical strips $ \gamma_1\times\Rbb $ and $ \gamma_2\times\Rbb $. By Theorem \ref{thm4.2.6} in $ \PSLhR $, there is a function $ u : \Omega \to \Rbb $ solution of the minimal graph equation with boundary data $u=+\infty$ on $ \gamma_1\cup \gamma_2 $ and $ u=0 $ on $ \eta_1\cup\eta_2 $ ($ u=-\infty $ on $ \eta_1\cup\eta_2 $ if the equality occurs in (\ref{equ3.3.2})).
	
	Let $ L $ be an open geodesic segment in $\Omega$ connecting $ \eta_1 $ to $ \eta_2 $ (see Figure \ref{fig:psl-h21}) and let be $ \varphi:=\Rest{u}{L} $.
	Let $ \bar{h}\in\Rbb $ be a positive constant such that the graph $ \Tsf_{\bar{h}}\left(\Gr({\varphi})\right)$ is above the minimal annulus $ \Acal $ with a touching tangent point at $ q\in\Acal $. The number $ \bar{h} $ exists since $ \Acal\cap (L\times\Rbb) $ is compact. We point out that for any $h\ge \bar{h} $, the minimal graph $ \Tsf_{h}(\Gr(\varphi)) $ is above the minimal 
annulus $\Acal$.

	\begin{figure}[h!]
		\centering
	\includegraphics[width=0.2\linewidth]{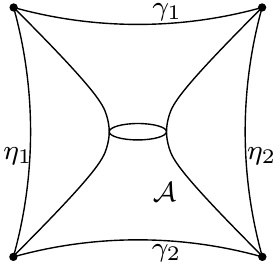}
		\qquad\qquad
	\includegraphics[width=0.2\linewidth]{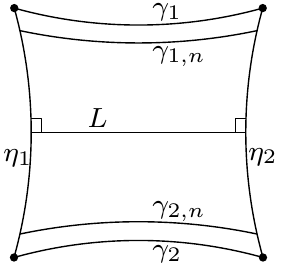}
		\caption{The quadrilateral domain $\Omega$.}
		\label{fig:psl-h21}
	\end{figure}
	
We claim that $ \Tsf_{\bar{h}}(\Gr({u}))$ is above $\Acal $ with a touching point at $ q $, contradicting the maximum principle.
	For each $n$, we denote by $ \gamma_{1,n} $ and $ \gamma_{2,n}$ a pair of geodesics in $ \Omega$ 
converging to $ \gamma_1$ and $ \gamma_2 $ respectively. For $ i=1,2 $, we denote by $ \Omega_{i,n} $ the domain in $\Hbbh $ bounded by $\eta_1,\eta_2,L $ and $ \gamma_{i,n} $ (see Figure \ref{fig:psl-h21}).
Hence, the sequence of domains $\Omega_{i,n} $ increases as $n\to +\infty $ and for $ i=1,2 $, $ \Omega_{i}:=\bigcup_{i=1}^{\infty}\Omega_{i,n}$ are the two components of $ \Omega\setminus L$.
By  Theorem \ref{thm4.2.6}, there is a minimal solution $v_{i,n} $, for $ i=1,2$ on $ \Omega_{i,n}$ such that $ v_{i,n}=\varphi $ on $ L $, $ v_{i,n}=+\infty $ on $ \gamma_{i,n} $ and $ v_{i,n}=0 $ on $ \eta_1\cup\eta_2 $ ($ v_{i,n}=-\infty $ on $ \eta_1\cup\eta_2 $ if the equality occurs in (\ref{equ3.3.2})). Using the maximum principle  and comparing the values at the boundary of the solutions we see that  $v_{i,n}\ge v_{i,n+1}\ge u$ on $ \Omega_{i,n}$ for $ i=1,2 $.
Therefore, the decreasing sequence  $ v_{i,n} $  converges to a minimal solution equal to $ u $ on $ \dhr\Omega_{i,n} $.By the uniqueness of Jenkin-Serrin solutions in $ \PSLhR $, the sequence of minimal graphs
$\Gr(v_{i,n}) $ converges to $\Gr ({\Rest{u}{\Omega_{i}}})$, for $ i=1,2$.
	
	From the hypothesis, the part of $ \Acal $ above $ \Omega_{i,n} $ is bounded and for $ i,n $ fixed,  $ \Tsf_{h}\left(\Gr ({v_{i,n}})\right)$ is above $ \Acal$ for $h$ large enough. From this position, we decrease $h$ until it touches $\Acal $.
Since this minimal graph cannot be contained in  $\Acal $, by the maximum principle, this graph touches $ \Acal $ at a point in its boundary, hence, over $L$ at $q$. Then, for $ i=1,2 $ and $ n \gg 0$, we have $ \Tsf_{\bar{h}}\left(\Gr ({v_{i,n}})\right) $ is above $ \Acal $. Passing to the limit, we have $ \Tsf_{\bar{h}}\left(\Gr({u})\right) $ is above 
$ \Acal $ with a contact at $q$, contradicting the maximum principle.
\end{proof}

\noindent
{\it Proof of the Proposition \ref{propo1.3}.} 
In the case where $\dist_{\Hbbh} ( \gammahat_1, \gammahat_2) = 2 \ln (\sqrt{2}+1)$, there exists a minimal graph $\Gr(u)$ of type Scherk with infinite value on the boundary of $\Omegahat$. The ideal domain in this case has the symmetry of a square and correspond to the case (2)  of Theorem \ref{thm4.2.6}. In the case where  $\dist_{\Hbbh} ( \gammahat_1, \gammahat_2)  >2 \ln (\sqrt{2}+1)$, we can apply the case (1) of Theorem \ref{thm4.2.6} for ideal Scherk domain to find a function $u: \Omegahat \to \Rbb$ with $u=+\vc$ on $\gammahat_1 \cup \gammahat_2$ and $u=0$ on $\etahat_1 \cup \etahat_2$, for $i=1,2$. We can use this function $u$ to get a contradiction exactly as in the proof of the Proposition \ref{pro3.3.2}.
$\hfill \square$


\subsection{Compact minimal annuli}
\label{subsect3.3.1}

Let $\Omega\subset\Hbb^2$ be a \emph{bounded} convex quadrilateral domain whose boundary is composed of open geodesic arcs $\gamma_1,\eta_1,\gamma_2$ and $\eta_2$ in this order as well as their endpoints.  For simplicity, we assume the two complete geodesics $\gammahat_i$ containing $\gamma_i$ are ultraparallel (no common point at infinity). We denote by $\Omegahat$ the ideal quadrilateral domain of $\Hbbh$ whose vertices are the endpoints at infinity of the geodesics $\gammahat_1$ and $\gammahat_2$ and whose sides are $\gammahat_1$, $\gammahat_2$, $\etahat_1$, $\etahat_2$. We also assume that the two complete geodesics containing $\eta_i$ are ultraparallel (see Figure \ref{fig3.2.1}(a)) and
\begin{equation*}
\ellH{\gamma_1}+\ellH{\gamma_2}>\ellH{\eta_1}+\ellH{\eta_2}.
\end{equation*}
The curve $\widetilde \Gamma^i_n$ will denote the boundary of $\gamma_i\times[-n,n]$ in the cylinder model. For each  $i=1,2$, let $\Gamma^i_n\subseteq \gamma_i\times\Rbb,n\ge 1$  be smooth closed curves satisfying the following conditions (see Figure \ref{fig3.2.1}(b).)
\begin{enumerate}
	\item $\Gamma^i_n$ is the boundary of a smooth convex domain obtained after smoothing the four corners of the convex domain 
$\gamma_i\times[-n,n]$ in the same manner for all $n$.  Hence, the geometry of $\Gamma^i_n$ is uniformly bounded. $\Gamma^i_n$ can be chosen as close as we want to $\widetilde \Gamma^i_n$  by smoothing the corner with small arc.
	
	\item $\Gamma^i_n$ is contained in the minimal disk  delimited by $\Gamma^i_{n+1}$ for all $n\ge 0$.
	\item As $n\to\infty$, $\Gamma^i_n$ converges smoothly to the boundary of $\gamma_i\times\Rbb$.
\end{enumerate}

	\begin{figure}[h!]
		\centering
		\includegraphics[width=0.8\linewidth]{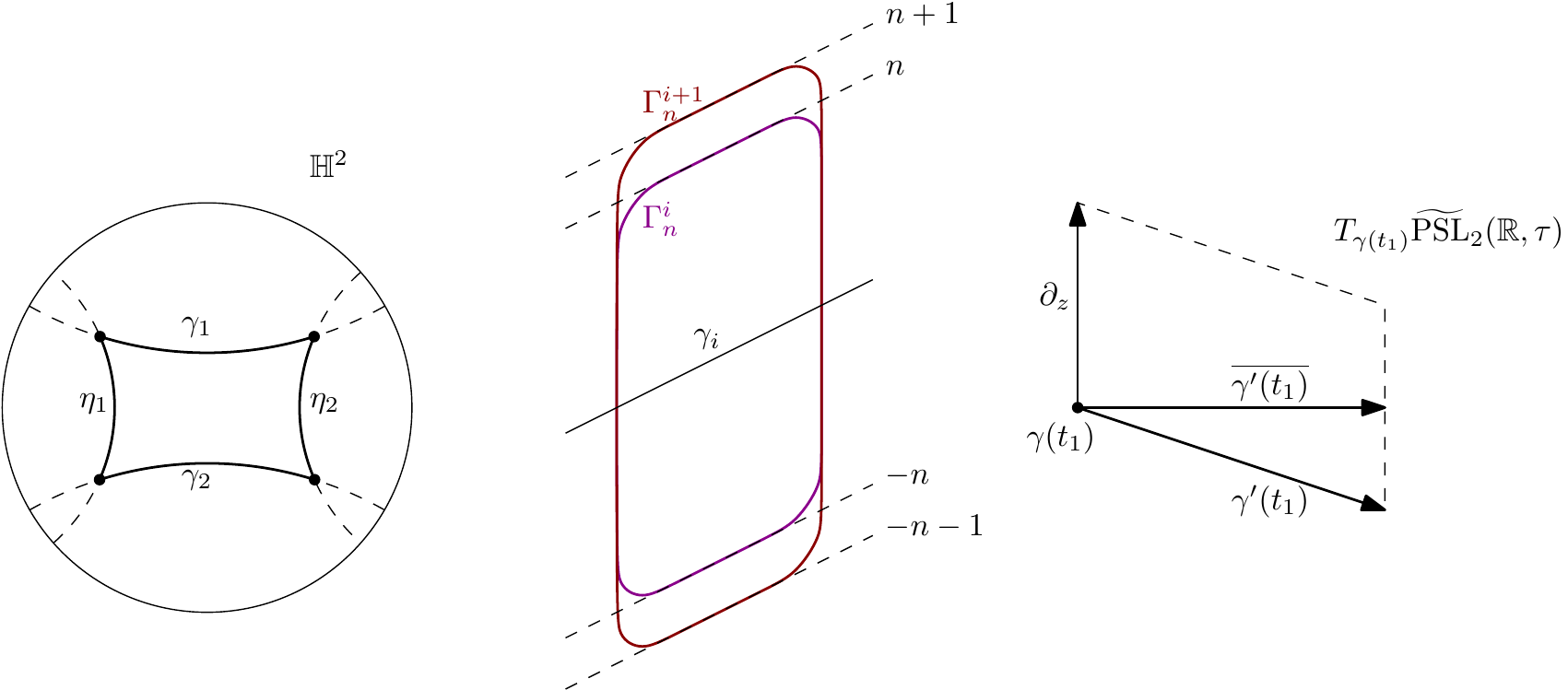}
		\caption{(a) The domain $ \Omega$. (b) $ \Gamma^i_n $. (c) Parallelogram generated by $\dhr_z$ and $\gamma'(t_1)$.}\label{fig3.2.1}.
	\end{figure}

	\noindent
We use Douglas criterion to construct a sequence of compact minimal annuli whose boundary is $ \Gamma^1_n\cup\Gamma^2_n$.

\begin{prop}\label{pro4.2.6}
Let $\Omega\subset\Hbbh$ be a bounded convex quadrilateral domain whose boundary is composed of open geodesic arcs $\gamma_1,\eta_1, \gamma_2, \eta_2$ in this order, together with their endpoints. We assume that 
	\begin{equation}\label{cond4.2.6}
	\ell_{\Hbbh}(\gamma_1)+\ell_{\Hbbh}(\gamma_2)> \ell_{\Hbbh}(\eta_1)+\ell_{\Hbbh}(\eta_2).
	\end{equation}
Then, for $n$ large enough, there exists a least area minimal annulus in $\PSLhR$ with boundary 
$\Gamma^1_n \cup \Gamma^2_n$.
\end{prop}
\begin{proof}
Let $D^i_n$ be the least area minimal disk with boundary $\Gamma^i_n$. Then, $D^i_n\subset \gamma_i \times[-n,n]$. By the Douglas criteria \cite{Jos85} (see also \cite[Theorem 1]{MY82c}), it is sufficient to show that for $n$ large enough, there is an annulus $\Acal$ (not necessarily minimal) with boundary $\Gamma^1_n\cup \Gamma^2_n$ having area less than $\sum_{i=1}^{2}\Area_{\PSLhR}(D^i_n)$.
In order to compute these areas, let us compute the area of $\gamma \times[0,h]$.

Using a parametrization of $\gamma$, we have by definition $$\mathrm{Area}_{\PSLhR}(\gamma \times [0,h])=\int_{\gamma \times[0,h]}A(t_1,t_2)\;dt_1dt_2$$
where $A(t_1,t_2)$ is the area of the parallelogram whose adjacent sides are $\gamma'(t_1)$ and $\dhr_z$ with respect to the metric of $\PSLhR$ (see Figure \ref{fig3.2.1}(c)). Since  $A(t_1,t_2)=1\cdot \chuan{\overline{\gamma}'(t_1)}_{\PSLhR}=\chuan{\gamma'(t_1)}_{\Hbbh}$ with $ \overline{\gamma}'(t_1) $ the horizontal lift of $ \gamma'(t_1) $ in $\PSLhR$, we have
	\begin{equation*}
		\mathrm{Area}_{\PSLhR}(\gamma\times[0,h])=h\int_{\gamma}\chuan{\gamma'(t_1)}_{\Hbbh}\; dt_1=h\ell_{\Hbbh}(\gamma)
	\end{equation*}

%
\noindent
and we deduce that
	\begin{equation*}
	 \mathrm{Area}_{\PSLhR}(\gamma_i\times[-n,n])= 2n\ell_{\Hbbh}(\gamma_i).
	\end{equation*}	
	Consider the annulus $\widetilde{\Acal}_n=\Tsf_n\left(\overline{\Omega}\times\{0\}\right)\cup \Tsf_{-n}\left(\overline{\Omega}\times \{0\}\right)\cup\left(\bigcup_{i=1}^2(\eta_i\times[-n,n])\right)\subset\PSLhR$.
Its boundary is $\widetilde{\Gamma}^1_n \cup \widetilde{\Gamma}^2_n$
and its area is
	\begin{equation*}
	\mathrm{Area}_{\PSLhR}(\widetilde{\Acal}_n)=2\mathrm{Area}_{\PSLhR}(\overline{\Omega})+2n(\ell_{\Hbbh}(\eta_1)+\ell_{\Hbbh}(\eta_2)).
	\end{equation*}
It follows from these calculations and the hypothesis \ref{cond4.2.6}, that for $n$ large enough
	\begin{align*}
	&\mathrm{Area}_{\PSLhR}(\widetilde{\Acal}_n)-\sum_{i=1}^{2}\Area_{\PSLhR}{(\gamma_i\times[-n,n])}\\
	=& 2\mathrm{Area}_{\PSLhR}(\Omega)+2n(\ell_{\Hbbh}(\eta_1)+\ell_{\Hbbh}(\eta_2)-\ell_{\Hbbh}(\gamma_1)-\ell_{\Hbbh}(\gamma_2)) <0,
	\end{align*}
Now we can add to $\widetilde{\Acal}_n$ the part of $\gamma_i \times [-n,n]$ between 
$\widetilde{\Gamma}^i_n$ and $\Gamma^i_n$ to obtain an annulus ${\Acal}_n$ bounded by $\Gamma^1_n \cup \Gamma^2_n$ with
$$\mathrm{Area}_{\PSLhR}({\Acal}_n)-\sum_{i=1}^{2}\Area_{\PSLhR}(D^i_n) =\mathrm{Area}_{\PSLhR}(\widetilde{\Acal}_n)-\sum_{i=1}^{2}\Area_{\PSLhR}{(\gamma_i\times[-n,n])}<0.$$
This completes the proof for the existence of a least area minimal annulus bounded by $\Gamma^1_n \cup \Gamma^2_n$.
\end{proof}

Let $\Scal_n$ be the set of  minimal annuli whose boundaries are $\Gamma^1_n\cup\Gamma^2_n$ and $\Scals_n$ a subset of $\Scal_n$ consisting of the stable minimal surfaces. Proposition \ref{pro4.2.6}
imply that $\Scals_n \neq \emptyset$ for $n$ large enough. We define a partially ordered relation $\preceq$ on $\Scal_n$ as follows. Let $\Acal_1,\Acal_2\in\Scal_n$. Then $\Acal_1\preceq\Acal_2$ if $\Acal_1$ is contained in the closure of the bounded component of $\left(\Omegabar\times\Rbb\right)\setminus \Acal_2$. We next point out the existence of a "largest" minimal annulus with boundary $ \Gamma^1_n\cup\Gamma^2_n $.

\begin{prop}\label{pro4.3.2}
	For $n$ large enough, there exists a stable, embedded minimal annulus $\Sigmas_{n}$ whose boundary is $\Gamma^1_{n}\cup\Gamma^2_{n}$. Moreover, for each minimal annulus $\Acal$ with boundary $\Gamma^1_n\cup\Gamma^2_n$, $\Acal$ is contained in the closure of the bounded component of $\left(\Omegabar\times\Rbb\right)\setminus \Sigmas_n$.
\end{prop}
\begin{proof}
Let $\Acal_1,\Acal_2 \in \Scal_n$ be two non-flat minimal annuli. 
The complement of $\Acal_1\cup\Acal_2$ in $\Omegabar\times\Rbb$ has only one unbounded component. Denote by $W$ the closure of this component.The curves $\Gamma^1_n$ and $\Gamma^2_n$ are homotopic in $W$ but homotopically nontrivial in $W$.
Hence, by Geometric Dehn's lemma \cite{MY82c}, $\Gamma^1_n\cup\Gamma^2_n$ is the boundary of a least-area embedded annulus 
$\Acal$ in $W$ and $\Acal_1$,  $\Acal_2$ are contained in the closure of the bounded component of $\left(\Omegabar\times\Rbb\right)\setminus  \Acal$.
It follows that
\begin{itemize}
	\item For each $\Acal\in \Scal_n$, there exists $\Acal'\in \Scals_n$ satisfying $\Acal\preceq \Acal'$.
	\item If $\Scals_n$ has a maximal element, then it is the greatest element of $\Scal_n$.
	\end{itemize}

Let $\Acal_j$ be an increasing sequence of elements of $\Scals_n$.
By compactness of the set of stable embedded minimal annuli with boundary $\Gamma^1_n\cup\Gamma^2_n$, see \cite{Anderson85} and \cite{Whi87}, the sequence $\Acal_j$ has a subsequence converging to a stable, embedded minimal annulus $\Acal_0$, see \cite[Lemma 2.1]{MW91} for a similar argument. Hence, $\Acal_0\in\Scals_n$ and $\Acal_j\preceq\Acal_0$ for all $n$. By Zorn's lemma, $\Scals_n$ has a maximal element. If an annulus $\Acal' \in \Scal_n$ is not contains in $\Acal_0$, then by Geometric Dehn's lemma, there is $\Acal^* \in \Scals_n$ with $\Acal_0 \preceq\Acal^*$, a contradiction. Hence $\Acal_0$ is the greatest element and contains any annulus of $\Scal_n$
\end{proof}


We recall that a compact minimal surface $\Sigma$ in a Riemannian $3 $-manifold $M$ is stable if the first eigenvalue of Jacobi operator $L=\Delta_\Sigma+|A|^2+\Ric_M(N,N) $ is non-negative.  If this first eigenvalue is strictly positive, $\Sigma$ is called strictly stable and if it is equal to $0$, $\Sigma$ is called stable-unstable. The surface $\Sigmas_n$ is strictly stable or stable-unstable.

In the case that $\Sigmas_n$ is strictly stable, we use the degree theory of White \cite{White87a,White89} to prove the existence of a second minimal annulus with boundary $\Gamma^1_n\cup\Gamma^2_n$.

\begin{prop}\label{pro4.3.6}
If $\Sigmas_n$ is strictly stable, then there exists another minimal annulus $\Sigmau_n$ having the same boundary as $\Sigmas_n$. 
\end{prop}

\begin{proof}
Suppose that $\Scal_n=\{\Sigmas_n\}$. Let $\Ccal$ be the space of pairs of smooth simple closed curves $(\alpha_1,\alpha_2)$ with $\alpha_i\subset\gammahat_i\times\Rbb$  where $\gammahat_i\subset\Hbbh$ is the complete geodesic containing $\gamma_i$.
Let $\mathcal{M}$ be the space of embedded minimal annuli with boundary curves in $\Ccal$. These annuli are transverse to $\gammahat_i\times\Rbb$ by the maximum principle at the boundary.
By \cite[Global structure theorem]{White87a} and \cite{White89}, $\mathcal{M}$ and $\Ccal$ are  Banach manifolds 
	\begin{equation*}
	\Pi: \mathcal{M}\to\Ccal,\qquad \Sigma\mapsto \dhr\Sigma
	\end{equation*}
	is a proper smooth Fredholm map of index $0$.
The boundary of $\Sigmas_n$ is an element of $\Ccal$ and by  the boundary maximum principle, $\Sigmas_n$ is not tangent to $\gammahat_1\times\Rbb$ and ${\gammahat_2}\times\Rbb$, thus $\Sigmas_n\in\mathcal{M}$.
	Since $\Sigmas_n$ is strictly stable, $d\Pi_{\Sigmas_n}:T_{\Sigmas_n}\mathcal{M}\to T_{(\Gamma^1_n,\Gamma^2_n)}\Ccal$ is an isomorphism.
	It follows that $\Sigmas_n$ is a regular point of $\Pi$.
	If  $\Sigmas_n$ is the unique minimal annulus bounded by  $\Gamma^1_n\cup\Gamma^2_n$, so $\left(\Gamma^1_n,\Gamma^2_n\right)\in\Ccal$ is a regular value of $\Pi$.
By \cite{White87a}, there exists an integer, denoted by  $\deg \Pi$, such that for each regular value  $(\alpha_1,\alpha_2)$ of $\Pi$, we have the equality
	\begin{equation*}
	\deg \Pi=\sum_{\Sigma\in\mathcal{M},\dhr\Sigma=\alpha_1\cup\alpha_2}(-1)^{\Index \Sigma}
	\end{equation*}	
	where $\Index\Sigma$ is the Morse index of the compact minimal surface $\Sigma$.
	Therefore, $\deg\Pi= (-1)^{\Index \Sigmas_n}=1$.
On the other hand, there exists curves $(\alpha_1,\alpha_2)\in\Ccal$ which doesn't bound any minimal annulus, hence $\deg \Pi=0$ which is a contradiction. This proves that $\Scal_n$ is not reduce to the single element $\{\Sigmas_n\}$. 
\end{proof}

\begin{deff}\label{rem3.3.5}
We denote by $\Sigmas_n$ the stable minimal annulus with boundary $\Gamma^1_n\cup\Gamma^2_n$ constructed in proposition \ref{pro4.3.2}.  $\Sigmau_n$ is defined to be $\Sigmas_n$ if $\Sigmas_n$ is stable-unstable and by the annulus of the proposition \ref{pro4.3.6} if $\Sigmas_n$ is strictly stable.
\end{deff}

\begin{rem}  The degree theory of B. White \cite{White87a,White89} consider Banach manifolds ${\cal C}$  of ${\cal C}^{k,\alpha}$ immersed curve and the Banach space ${\cal M}$ of ${\cal C}^{k+2,\alpha}$ minimal immersion which extend curves of ${\cal C}$ in a ${\cal C}^{k,\alpha}$ way. If the boundary $\Gamma^1_n \cup \Gamma^2_n$ are of class ${\cal C}^{k,\alpha}$, then the least area minimizing annulus $\Sigmas_n$  is ${\cal C}^{k,\alpha}$ up to its boundary. Hence annuli constructed by degree theory is a surface ${\cal C}^{k,\alpha}$ up to its boundary.
\end{rem}

\section{Tangent plane and uniform curvature estimate}

\subsection{Tangent geometry of horizontal compact minimal annuli}\label{subsect3.3.2}

Our goal is to prove a uniform bound on the curvature of the sequences 
$\Sigmas_n$ and $\Sigmau_n$. To estimate the curvature,  we first determine some geometric properties of these surfaces, in particular the upper bound on the number of tangent points of these minimal annuli with certain  minimal foliations of $\PSLhR$. The properties we construct 
in the sequel do not depend on the stability of the annuli, so, without loss of generality, 
we denote by $\Sigma_n$ a minimal annulus in $\PSLhR$ whose boundary 
is $\Gamma^1_n\cup\Gamma^2_n$ for $n\gg 0$. The foliations we consider are listed bellow.

\begin{enumerate}	
\item Foliation $\Fcal$ where each leaf is of the form $\alpha\times\Rbb$ where $\alpha$ is a complete geodesic of $\Hbbh$. We can consider for example, foliation where each leaf is perpendicular to a given complete geodesic $\beta$ of $\Hbbh$. We will denote this foliation by $\Fcal^{\beta}$  if this is the case.
	\item Foliation $\Fcalh$ where each leaf is a level set of the coordinate function $\zsf:\PSLhR\to\Rbb, (x,y,z)\mapsto z$. Since each level set is minimal, $\Fcalh$ is a  minimal foliation of $\PSLhR$.
	\item Foliation $\Fcal^{\Gr(w)}$ where each leaf is a vertical translation of a graph $\Gr(w)$ where $w$ is a solution of the minimal graph equation define on a domain  $\Omega\subset\Hbbh$. Hence, this is a minimal foliation of $\Omega\times\Rbb$.
More precisely we only consider a special case. 
For $\Omega\subset\Hbbh$ a bounded quadrilateral domain whose boundary consists of compact geodesics 
$\gamma_1,\eta_1,\gamma_2,\eta_2$ with their endpoints $p_1,q_1,p_2,q_2$, we consider the solution of Dirichlet problem with the following boundary data. 

 On $\gamma_i$, the upper part of the curve $\Gamma^i_n$ is the graph of a function $f_{i,n}$ and  we remark that for $n>n_0$ we have $f_{i,n}=\Tsf_{n-n_0}(f_{i,n_0})$. Now we extend these functions $f_{i,n_0}$ to a continuous function on $\dhr \Omega$ by two functions $g_i : \eta_i \to \Rbb$. For $n_0$ large enough, we define by Theorem \ref{thm4.2.6}, the function $w:\Omega \to \Rbb$ such that
$\Gr(w)$ is a minimal surface with boundary  data

\[
\left\{
\begin{array}{lll}
w=g_i & \hbox{ on } \eta_i \hbox{ for } i=1,2 \\
w=f_{i,n_0} & \hbox{ on } \gamma_{i} \hbox{ for } i=1,2.  \\
\end{array}\right.
\]

We remark that $\Gr(w)$ is bounded by a curve without vertical segment. Hence by construction, there are two real numbers $a>b$ satisfying: $\Gamma^i_n$ does not intersect $\Tsf_h(\Gr(w))$ if $h\notin [b,a]$,  $\Gamma^i_n\cap \Tsf_h(\Gr(w))$ is nonempty and connected if $h=a$ or $h=b$, finally, for each $a>h>b$, $\Gamma^i_n$ intersects transversally $\Tsf_h(\Gr(w))$ at two distinct points. Notice that the same construction is available for the lower part of $\Gamma_n^i$.
\end{enumerate}

All these foliations satisfy the hypothesis of the following proposition:

\begin{prop}
\label{pro3.1}
Consider a foliation $ \Fcal=\bigcup_{t\in[0,1]} \Lambda_t $ such that each leaf is a simply connected minimal surface. If a leaf has no empty boundary we assume that $\partial \Lambda_t \subset \partial {\Omega} \times \Rbb$  and  assume that
	\begin{enumerate}
		\item $ {\Omega} \times [-n,n]  $ is entirely contained in one side of $ \Lambda_0 $ and $ \Lambda_1 $.  $\Lambda_0$ and $ \Lambda_1 $ doesn't intersect $\Sigma_n \setminus \partial \Sigma_n$ and eventually contains $\{1\hbox{ segment}\}$ or $\{1 \hbox{ point}\}$  of each boundary curves $ \Gamma^1_n $ and $ \Gamma^2_n $.		
		\item There exists two constants $ 0\le a<b\le1 $ (possibly equal to $ 0 $ and $ 1 $) such that
		\begin{enumerate}
			\item $ 0<t<a $ or $ b<t<1 $: $ \Lambda_t $ intersects only one of the boundary curve $ \Gamma^1_n$ or $ \Gamma^2_n  $ in exactly $ 2 $ points ($\partial \Sigma_n \cap \Lambda_t=\{2 \hbox{ points}\}$).
			\item $ \Lambda_a $ and $ \Lambda_b $ intersects  both $\Gamma^1_n$ and 
			$\Gamma^2_n$. One in $\{1\hbox{ segments}\}$ or $\{1 \hbox{ point}\}$ and the other one in $\{2 \hbox{ point}\}$.
\item $ a<t<b $: $ \Lambda_t $ intersects each connected component $ \Gamma^1_n $ and $ \Gamma^2_n  $ in exactly $ 2 $ points 
			($\partial \Sigma_n \cup \Lambda_t=
			\{4 \hbox{ points}\}$).
			
		\end{enumerate}
	\end{enumerate}
Hence	the annulus $ \Sigma_n $ meets $ \Fcal $ in a tangential way at most at $2$ points.
\end{prop}

\begin{proof}
	By the maximum principle, $ \Sigma_n $ is in one side of $ \Lambda_0 $ and $ \Lambda_1 $ and cannot have an interior point of contact with these leaves.
Consider $ t\in (0,1) $ such that $ \Lambda_{t} $ is tangent to $ \Sigma_n $ at $ p $. The local intersection of $ \Sigma_n $ and $ \Lambda_t $ near $ p $ consists of $ k $ curves through $ p $, $ k\ge 2 $, meeting at angles $ \frac{\pi}{k} $ at $ p $. The intersection $ \Sigma_n\cap\Lambda_t $ cannot bound a component $\Sigma  $ such that $ \dhr\Sigma\cap \dhr\Sigma_n=\emptyset $ by  the maximum principle. We have the following Lemma:

	\begin{lem}\label{lem1.haus}
		\begin{enumerate}[label=\rm(\alph*)] 
			\item If $ D\subset\Sigma_n $ is a minimal disk which satisfies that $ \forall t\in[0,1] $, $ \Lambda_t\cap\dhr D$ 
			is either $\{2 \hbox{ points}\}$ or  $\{1 \hbox{ segment}\}$ or an empty set, then $ \Lambda_t $ cannot be tangent to $ D $ at an interior point.
			\item If $ A\subset\Sigma_n $ is a minimal annulus which satisfies that $ \forall t\in [0,1] $, $ \Lambda_t\cap\dhr A= \emptyset$, then $ \Lambda_t $ cannot be tangent to $ A $ at an interior points.
			\item
			If $ A\subset\Sigma_n $ is a topological annulus with $ \dhr A=\Gamma_1\cup\Gamma_2 $ satisfies that for all $ t\in[0,1] $, $ \Lambda_t\cap\Gamma_1 =\{2\hbox{ points}\}$  and $ \Lambda_t\cap\Gamma_2= \emptyset$ (or vice-versa), then there is only one leaf which can be tangent to $ A $ at only one point. If $ \Lambda_t\cap\Gamma_1 =\{1 \hbox{ point}\}$ or $\{1 \hbox{ segment}\}$ and $ \Lambda_t\cap\Gamma_2= \emptyset$ (or vice-versa), then $ \Lambda_t $ cannot be tangent to $ A $ at an interior points.
		\end{enumerate}
	\end{lem}
\begin{proof}[Proof of Lemma]
(a) Assume that $ \Lambda_t $ intersect $ D $ at an interior point $q$ in a tangential way. Then $ \Gamma (t) =\Sigma_n\cap\Lambda_t $ has at least two curves at $q$ with at least $4$ endpoints. If two endpoints of $ \Gamma (t) $ are at the same point of the boundary or in a segment of $\Lambda_t \cap \partial D$, 
then $ \Gamma (t) $ bounds a topological disk $ \Sigma \subset D $, with $ \dhr\Sigma\subset\Lambda_t $, a contradiction with the maximum principle. Hence, 
there is at most two curves with endpoints into $ \dhr D $ and there remains at least two other branches which have to connect together. These two branches bound a compact connected domain $\Sigma \subset D$ with $ \dhr\Sigma\subset\Lambda_t $, a contradiction too.

(b) If $ \Lambda_t\cap\dhr A=\emptyset $ and $ \Lambda_t $ meets $ A $ tangentially at $q$ then $ \Gamma (t)=\Lambda_t\cap A $ has at least two curves which form a figure eight in the annulus.
\begin{figure}[h!]
	\centering
	\includegraphics[width=0.45\linewidth]{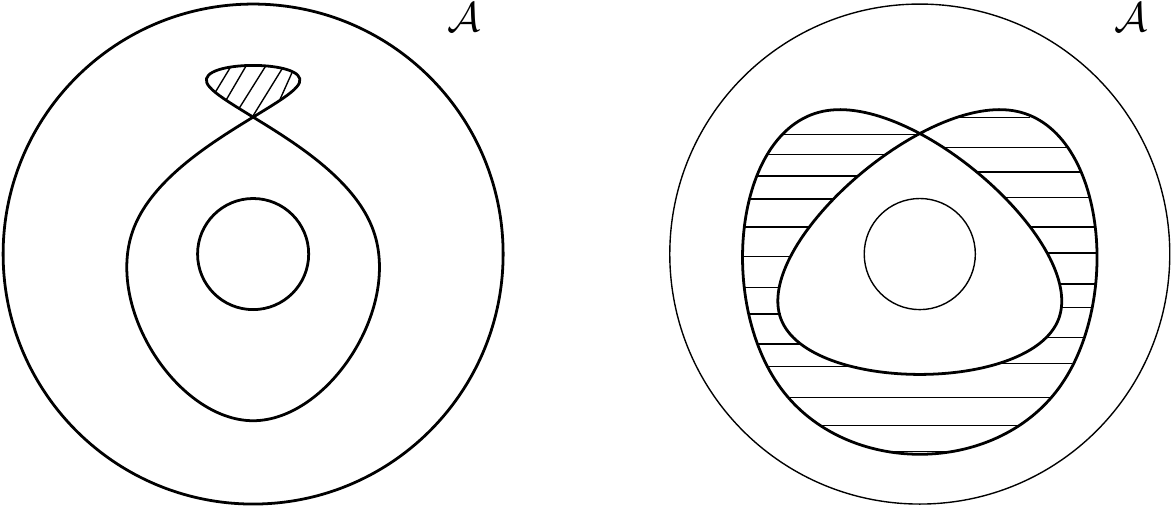}
	\caption{(Left) $\Gamma (t)$ bound one disk. \hskip 0.3cm (Right) $\Gamma (t)$ contains two non trivial loop.}
	\label{fig3.4.1}
\end{figure}
Since $ \Gamma (t) $ cannot connect the boundary, $\Gamma (t)$ bound at least one disk or one annulus (see Figure \ref{fig3.4.1}). A contradiction with the maximum principle.

(c) By the maximum principle,  $A \setminus \Lambda_{t}$  has connected components which exist if and only if their boundary contains connected component of $ \dhr\Sigma_n\setminus\Lambda_t$.  If $ \Lambda_{t} \cap \dhr A $ consists of two points in $\Gamma_1 $ and no intersection point with $\Gamma_2$, then $\Lambda_{t}$ disconnect $\dhr A$ into three connected components. If $ \Lambda_{t} \cap \dhr A $ consists of one point or one segment in $\Gamma_1 $ and no intersection point with $\Gamma_2$, then $\Lambda_{t}$ disconnect $\dhr A$ into two connected components. Hence $A \setminus \Lambda_{t}$ has two or three connected components.

Locally at a tangent point  $q$, there are at least four local connected components of $A \setminus \Lambda_{t_0}$ . Then there is a curve $\alpha_0 \subset (A \setminus \Lambda_{t_0})$ joining two points $x,y$ of different local connected component $\Sigma_1$ and $\Sigma_2$ in a neighborhood of $q$ outside $\Lambda_{t_0}$. 
Then join  $x$ to $y$ by a local path $\beta_0$ going through $q$ to  obtain a  compact cycle $(\alpha_0 \cup \beta_0) \subset (A \setminus \Lambda_{t_0})$ except at $q$; see Figure  \ref{fig3.4.2}-(left). The cycle $(\alpha_0 \cup \beta_0)$ separate $A$ into two connected components say $A_1$ and $A_2$ and each of them contains at least two branches of $\Gamma (t_0)= \Lambda_{t_0} \cap  A $. Then for $i=1,2$, $A_i\cap \partial A \neq \emptyset$ by the  maximum principle. If $\Gamma_1 \subset A_1$ and $\Gamma_2 \subset A_2$, then the branches contained in $A_2$ have to connect together to form a cycle $\gamma $ in $A$. The cycle $ \gamma $ cannot bound a disk hence it's not trivial in $ A $ (see Figure \ref{fig3.4.2}-(Right)). If two branches of $ \Gamma (t_0)\cap A_1$ connect together, then it would produce an other non trivial cycle $\gamma'$ in $A$, but $\gamma \cup \gamma'$ will bound a subannulus with boundary contained in $\Lambda_{t_0}$, a contradiction with the maximum principle (see Figure \ref{fig3.4.1}-(Right)).

Hence two branches of $ \Gamma (t_0)\cap A_1$ has to connect $\Gamma_1$ into two points and produce necessarily three connected components $A \setminus \Lambda_{t}$ . Such a configuration cannot occur in the case where $\Lambda_{t} \cap \dhr A $ consists of one point or one segment in $\Gamma_1$ and no point in $\Gamma_2$ ($A \setminus \Lambda_{t}$ has only two connected components).

\begin{figure}[h!]
	\centering
	\includegraphics[width=0.4\linewidth]{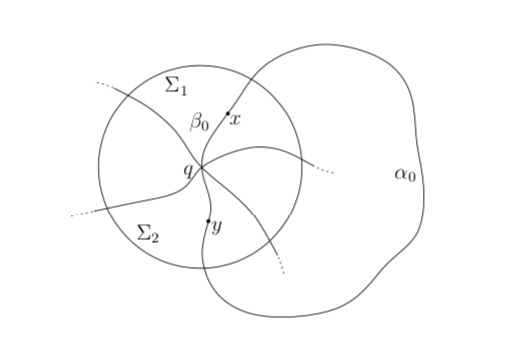}
	\hskip  1cm
	\includegraphics[width=0.3\linewidth]{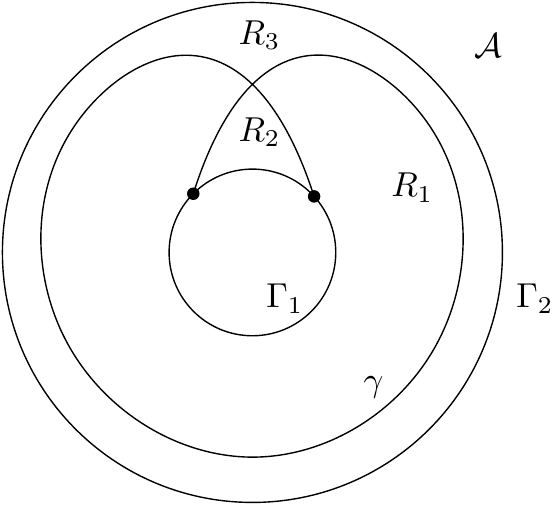}
	\caption{(Left)The cycle $(\alpha_0 \cup \beta_0)\subset A \setminus \Lambda_{t_0}$. \hskip 0.5cm (Right)$A \setminus \Gamma(t)$ has three connected components.}
	\label{fig3.4.2}
\end{figure} 

Now the curve $\Gamma (t_0) $ disconnects the annulus in $3$ connected components having in its boundary a component of $ \dhr A \setminus \Lambda_{t_0} $.  Two of these components $R_1$ and $R_2$ are topological disk and each leaf of the foliation with $ t\ne t_0 $ meet the boundary of these disc at most in $ 2 $ points or $1$ segment which imply that there is no other leaf tangent to $A$ at points of $R_1$ and $R_2$.
The other component $ R_3 $ is an annulus where any leaf $ \Lambda_t $ with $ t\ne t_0 $ will meet the boundary in an emptyset. Applying the case $(b)$, such a leaf cannot touch tangentially $R_3$. In conclusion beside $ \Lambda_{t_0} $, there is no other leaf of $ \Fcal $ tangent to an interior point of $ A $.

\end{proof}

 We remark that the hypothesis of proposition \ref{pro3.1} say that $ \Lambda_t \cap \dhr\Sigma_n$ has at most $ 4 $ connected components. The set $ \Lambda_t \cap (\Gamma^1_n\cup\Gamma^2_n)$ can be
 
 \begin{itemize}
\item $\{2 \hbox{ points}\}\subset\Gamma^1_n$  and  $\emptyset$ in $\Gamma^2_n $ (and vice-versa)
\item $\{2 \hbox{ points}\} \subset\Gamma^1_n $
and $\{1 \hbox{ segment}\} $ or $\{1 \hbox{ point}\}  \subset\Gamma^2_n $ (and vice-versa)
\item $\{2 \hbox{ points}\}  \subset\Gamma^1_n $
and $\{2 \hbox{ points}\} \subset\Gamma^2_n $
\item $\{1 \hbox{ segment}\} \subset\Gamma^1_n $ and  $\emptyset$ in $\Gamma^2_n $ (and vice-versa).
\item  $\{1 \hbox{ segment}\}$ or $\{1 \hbox{ point}\} \subset\Gamma^1_n $
and $\{1 \hbox{ segment}\}$  or $\{1 \hbox{ point}\} \subset
\Gamma^2_n $.
 \end{itemize}
 This prove 
\vskip 0.2cm
 {\it Claim 1.}  In all the cases, $ \dhr\Sigma_n \setminus\Lambda_t$ has at most $ 4 $ connected components.
 
 \begin{lem}\label{lem2.haus}
a) A leaf $ \Lambda_t $ tangent to $ \Sigma_n $ at $ q $ defines at most $ 4 $ connected components in $ \Sigma_n\setminus \Lambda_t $ and cannot touch $ \Sigma_n $ in a tangential way at more than two interior points. 

b) If $ \Lambda_t $ is tangent to $ \Sigma_n $ at exactly two points then $ \Sigma_n\setminus\Lambda_t  $ has exactly $4$ connected components. 

c) If $ \Sigma_n\setminus\Lambda_t  $ has $4$ connected components, each one is a topological disk.
 \end{lem}
 
 \begin{proof}[Proof of Lemma]
The set of curves $ \Gamma(t)=\Sigma_n \cap\Lambda_t$ consists locally in $ k $ curves passing through $ q $, $ k\ge 2$. $ \Gamma (t) $ disconnects $ \Sigma_n $ into several connected components. 
Since these components exist if and only if their boundary contains connected component of $ \dhr\Sigma_n\setminus\Lambda_t $, there is at most $ 4 $ connected components by Claim 1.


If $ \Lambda_t $ is tangent to $ \Sigma_n $ in $s$ points, then $ \Gamma (t) $ has a graph structure which at least $4$ edges at each interior vertices. The $b$ end points of $\Gamma (t)$ are boundary vertices on $\partial \Sigma_n$. In the case where the boundary is one segment, we identify this boundary component with one boundary vertex by removing an equal number of boundary edges and boundary vertices contains in $\partial \Sigma_n$. By the maximum principle any cycle of $\Gamma (t)$ has to be homotopic to the boundary of $\Sigma_n$. Two cycles would bound an annulus which is a contradiction too. Hence, $\Gamma (t)$ is a tree or contains at most one cycle. If $E$ denote the number of edges and $V$ the number of vertices, we have $2E=2V-2 \geq 4s+b$ when $\Gamma (t)$ is a tree and $2E=2V \geq 4s+b$ if $\Gamma (t)$ contains a cycle. Since  $V=s+b$ and $b \leq 4$ we have that $s \leq 2$ and $ \Lambda_t $ is tangent to $ \Sigma_n $ at most at two points.

In the case where $ \Lambda_t $ is tangential to  the annulus in exactly two points, then $ \Gamma (t) $ contains a non trivial cycle and two curves with 4 endpoints in 
$\dhr\Sigma_n $. Hence $\partial \Sigma_n \setminus \Gamma (t)$ has $ 4 $ connected components and the two boundary $\Gamma_{1,n}$ and $\Gamma_{2,n}$ are separated by the cycle.
Hence $\Sigma_n \setminus \Gamma (t)$ has 4 connected components.

If $\Sigma_n \setminus \Gamma (t)$ has 4 connected components and  one of the component is an annulus $R_1$, then $ \dhr R_1 $ contains $ \Gamma^1_n $ or $ \Gamma^2_n $ which is entirely contains in $ \dhr R_1 $. This forces the three other components to have points in one of boundary curves say $ \Gamma^1_n $, a contradiction since $ \Gamma^1_n\setminus \Lambda_t $ has at most two connected components. This prove that each connected components of $\Sigma_n \setminus \Gamma (t)$ is topologically a disc with
4 different connected components of $\partial \Sigma_n \setminus \Gamma (t)$ in its boundary.
 
 \end{proof} 
 
 Now we prove the proposition \ref{pro3.1}. We consider $t_0 \in [0,1]$ such that $\Lambda _{t_0}$ is tangent to $\Sigma_n \setminus \partial \Sigma_n$ at $q$
 and other leaves $\Lambda_t$ are tranverse to $\Sigma_n$, for $t \in [0,t_0[$.
 
 \vskip 0.5cm
 We remark that from the Lemma \ref{lem2.haus}, if $ \Sigma_n \setminus \Lambda_t $ has $ 4 $ connected components, each of them is a topological disk and by Lemma \ref{lem1.haus}-a), there is no other leaf in the foliation $ \Fcal $ that touches tangentially the annulus.  The reason is that each connected component is a disk with boundary 
 contains in $ \Lambda_t $ and in only one component of $ \dhr\Sigma_n $, i.e. $ \Gamma^1_n $ or $ \Gamma^2_n $. 
 Hence any other leaf will meet the boundary of these disks in $ \{2\hbox{ points}\}$ or $\{1\hbox{ segment}\}$, contradicting Lemma \ref{lem1.haus}-a).

There remains to study the case where $ \Sigma_n\setminus \Lambda_t$ has only $3$ connected components. In this case, $ \Lambda_t $ meet tangentially $ \Sigma_n $ in only one point $ p $. (If $ \Lambda_t $ touches in $ 2 $ points then $ \Sigma_n\setminus\Lambda_t $ has $ 4 $ connected components by Lemma \ref{lem1.haus} b)). In the case where $ \Sigma_n\setminus \Lambda_t $ has $ 3 $ connected components, there are two cases:

Case 1: $ \dhr\Sigma_n\setminus\Gamma (t) $ has $ 3$ connected components, each one are in one components $ R_1, R_2, R_3 $ of $ \Sigma_n\setminus\Gamma (t) $. If $ R_1, R_2, R_3 $ are disks, then the foliation cannot touch in tangent way $\Sigma_n  $ in an other leaf. It can happen that one of the components, say $R_3$ is a topological annulus which contains in its boundary one of the boundary curve  $ \Gamma^1_n $ or $ \Gamma^2_n $ (see Figure \ref{fig3.4.4}(Left)).

\begin{figure}[h!]

\hskip 3cm
	\includegraphics[width=0.3\linewidth]{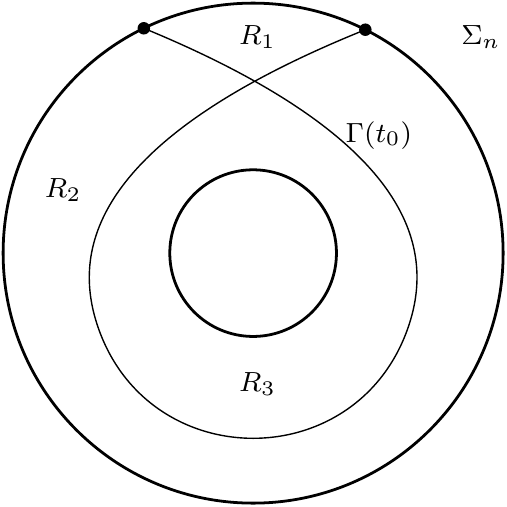}
	\hskip 2cm
	\includegraphics[width=0.45\linewidth]{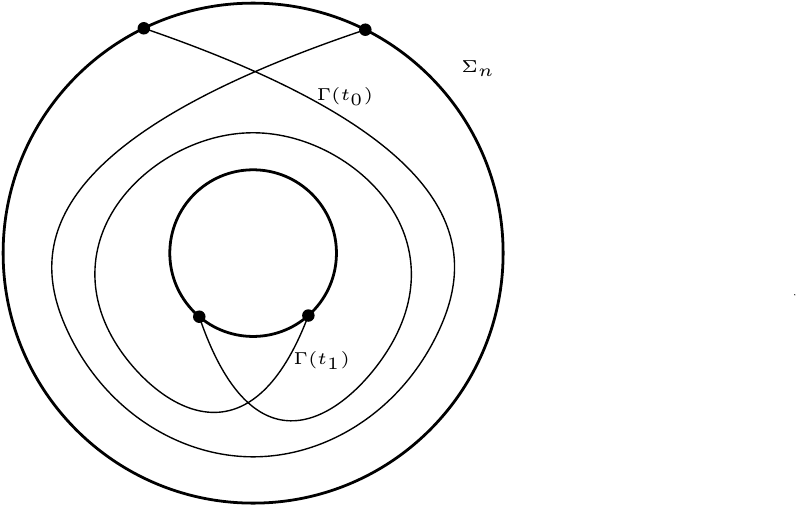}
	\caption{(Left) $\Sigma_n \setminus \Gamma$ has 3 connected component .
	\hskip 0.3cm (Right) $\Sigma_n \setminus (\Gamma (t_0) \cup \Gamma (t_1))$.}
	\label{fig3.4.4}
\end{figure}

In this case $ R_1 $ and $ R_2 $ cannot meet the foliation tangentially in an other leaf. In an other way, since $ R_3 $ is an annulus, any leaf meeting $ R_3 $ will intersect the boundary of $ R_3 $ at most at two points or $ 1 $ segment. In this case, one can find other tangential leaf $\Lambda_{t_1}$ that will disconnect $ R_3 $ along $\Gamma (t_1)$ with disks and annulus satisfying the condition of Lemma \ref*{lem1.haus} (See Figure \ref{fig3.4.4} (Right)) and there is at most two point of tangency with the foliation.

\begin{figure}[h!]

\hskip 3cm
	\includegraphics[width=0.3\linewidth]{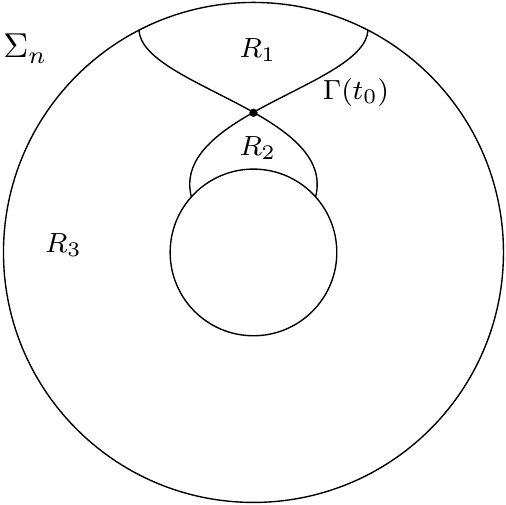}
	\hskip 2cm
	\includegraphics[width=0.3\linewidth]{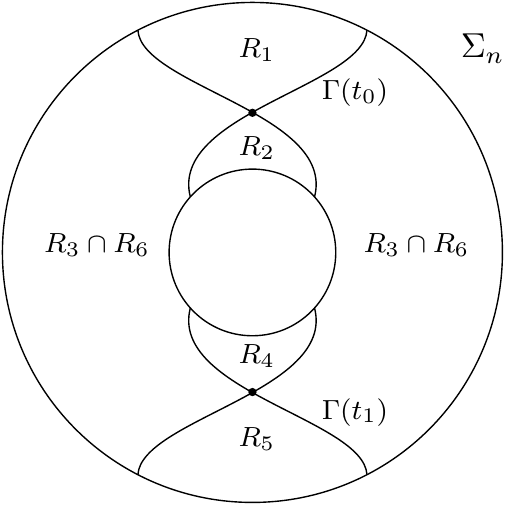}
	\caption{(Left) $\partial \Sigma_n \setminus \Gamma$ has 4 connected components .
	\hskip 0.3cm (Right) $\Sigma_n \setminus (\Gamma (t_0) \cup \Gamma (t_1))$.}
	\label{fig3.4.5}
\end{figure}

%

Case 2: $ \dhr\Sigma_n\setminus \Gamma (t) $ has $ 4 $ connected components. Two components $ R_1 $ and $R_2 $ have only one connected component of $ \dhr\Sigma_n\setminus \Gamma (t) $ in their boundary and $R_3 $ has in its boundary two connected components of $ \dhr\Sigma_n\setminus\Gamma (t)$ (see Figure \ref{fig3.4.5} ). All components are topological disk. In this case $ R_1 $ and $ R_2 $ are not tangent to $ \Fcal $. The component $R_3 $
can meet the foliation tangentially in an other leaf. We consider $t_1 \in ]t_0, 1]$ the value such that $\Lambda_{t_1}$ meet $R_3$ tangentially
and $\Lambda_t$ is transverse to $R_3$ for $t_0<t<t_1$. 
Since $R_3 $ is a disk, $\Gamma (t_1)$ will disconnect $R_3$ in different connected components each one a disk  (there is no cycles in $\Gamma (t_1)$) .
Since $\Gamma (t_1)$ has at least 4 endpoints, there will be 4 connected components in $R_3 \setminus \Gamma (t_1)$ and $\Gamma (t_1)$ is disconnecting
$\Sigma_n$ in 3 connected components. By same argument as before, there is two discs $R_4$ and $R_5$ which cannot be tangent to an other leaf and $R_6$
a component having two components of $\partial \Sigma_n \setminus \Gamma_1$. The intersection $R_3 \cap R_6$ are then disjoint from $R_4$ and $R_5$, and
$\Lambda_t \cap (R_3 \cap R_6)= \emptyset$ if $t <t_0$ or $t> t_1$. Then there is no other leaf than $\Lambda_{t_0}$ and $\Lambda_{t_1}$ which are tangent to $\Sigma_n$.

 \end{proof}

%


\subsection{Curvature estimate}
\label{subsect3.3.3}

For each $n\gg 0$, denote by $\Sigma_n$ a minimal annulus whose boundary is $\Gamma^1_n\cup\Gamma^2_n$.
The main purpose of this section is to show a uniform curvature bound for the sequence $\Sigma_n$ without stability
assumption.
If there is a sequence of points $ p_n\in\Sigma_n $ such that the norm of the second funndamental form ${\chuan{A_{\Sigma_{n}}(p_n)}}\to+\infty $ then we study in the following the limit of subsequence of blow up annuli $ \lambda_n\Tsf_{p_n}(\Sigma_n) $ where $ \lambda_n:= \chuan{A_{\Sigma_{n}}(p_n)}$ and $ \Tsf_{p_n} $ is the translation in $ \PSLhR $ which sends $ p_n$ to the origin $O$. To understand this limit, we also consider the intersection of $ \lambda_n\Tsf_{p_n}(\Sigma_n)$ with $ \lambda_n \Fcal $ where $ \Fcal $ is a foliation described in the previous section.
More precisely, we state and prove the following proposition.

\begin{thm}\label{pro4.4.4}
	The sequence of minimal surfaces  $\left\{\Sigma_{n}\right\}_n$ has uniformly bounded second fundamental form
	\begin{equation}\label{equ4.4.2}
	\sup_{n}\sup_{\Sigma_{n}}\chuan{A_{\Sigma_{n}}}<
	+\infty.
	\end{equation}
\end{thm}


\begin{proof}[\bf Proof of Theorem \ref{pro4.4.4}]

Suppose at the contrary that  \eqref{equ4.4.2} does not hold and without loss of generality, we assume that there exists a point $p_n$ of $\Sigma_n$ such that
	
\begin{equation}
\lambda_{n}:=\chuan{A_{\Sigma_{n}}(p_n)}=\sup_{p \in\Sigma_{n}}\chuan{A_{\Sigma_{n}}(p)}\to+\infty,
\end{equation}
as $ n\to+\infty$.

	Therefore, $\Tsf_{p_n}(\Sigma_n)\subset\PSLhR$ is a minimal surface whose curvature reaches the maximum value  $\lambda_n$ at $O$.
	For each $n$, the scalar multiplication by $\lambda_n$ gives a diffeomorphism
	\begin{equation}\label{equ3.3.5}
	\Rbb^3\to\Rbb^3,\quad (x,y,z)\mapsto (\lambda_nx,\lambda_ny,\lambda_nz).
	\end{equation}
	Denote by $\Sigmatilde_n=\lambda_n\Tsf_{p_n}(\Sigma_n)$.	For each $\kappa<0$ and $\tau\in\Rbb$, define $\Ebbbkt$ to be the smooth $3$-manifold $ D^2(\kappa)\times\Rbb $ endowed with the following Riemannian metric
	
\begin{equation*}
	\lambda^2\left(dx^2+dy^2\right)+\left(\tau\lambda(ydx-xdy)+dz\right)^2
\end{equation*} 
with $\lambda=\frac{1}{1+\frac{\kappa}{4}(x^2+y^2)}$.
We have $\PSLhR=\Ebb^3(-1,\tau)$.
Moreover, the restriction of the diffeomorphism \eqref{equ3.3.5} of $\Rbb^3$ gives us a conformal diffeomorphism $\PSLhR\to\Ebb^3\left(\frac{-1}{\lambda_n^2},\frac{\tau}{\lambda_n}\right)$  with the conformal factor $\lambda_n$.
	Therefore, $\Sigmatilde_n$ is a minimal surface in $\Ebb^3_n:=\Ebb^3\left(\frac{-1}{\lambda_n^2},\frac{\tau}{\lambda_n}\right)$ where the curvature reaches the maximum value  $1$ at the  point $O$. 
Since $\lambda_n$ tends to $+\infty $ as $n\to+\infty$, the sequence of Riemannian manifolds $\Ebb^3_n$ converges smoothly to the Euclidean space $\Rbb^3$ as $n\to +\infty$.

We will consider in the following $\Sigmatilde_n$ as a sequence of surfaces in $\Rbb^3$ (non necessarily minimal) with bounded second fundamental form for the Euclidean metric satisfying some equation of mean curvature type and converging to $\Sigmatilde_\infty$ a minimal surface for the Euclidean metric.
We study in the following proposition  \ref{lem3.3.12} possible limit of any subsequence of $\Sigmatilde_n$ to $\Sigmatilde_\infty$:

	\begin{prop}\label{lem3.3.12}
	There exists a subsequence $\left\{\Sigmatilde_{k} \right\}_{k}$ of the sequence $\left\{\Sigmatilde_n \right\}_n$ and a minimal surface $\Sigmatilde_\infty$ in $\Rbb^3$ satisfying the following conditions:
	\begin{enumerate}
	\item $\Sigmatilde_\infty $ is embedded in $\Rbb^3$;
	\item $\Sigmatilde_\infty $ is contained in the accumulation set of  $\left\{\Sigmatilde_k \right\}_k$;
	\item $O\in \Sigmatilde_\infty$ and $\td{A_{\Sigmatilde_\infty }(O)}=\lim\limits_{k\to+\infty }\td{A_{\Sigmatilde_k}(O)}=1 $ where $A_{\Sigmatildevc}$ and $A_{\Sigmatilde_k}$ are respectively the second fundamental form of $\Sigmatildevc$ in $\Rbb^3$ and of $\Sigmatilde_k$ in $ \Ebb^3_k $. 
	
	\item $\Sigmatilde_\vc$ is complete, i.e., any divergent path in $\Sigmatilde_\vc$ has infinite length;
	\item The surface $\Sigmatilde_\vc$ has finite total curvature, without boundary.
	\end{enumerate} 
	\end{prop}

Assuming this proposition we first prove the theorem. 	Since $\chuan{A_{\Sigmatildevc}(O)}=1$,
	$\Sigmatildevc$ is not a flat plane.
	It follows from \cite[Theorem 3.1]{HK97} that $\Sigmatildevc$ has at least  $2$ ends.
	Let $\mutilde \subset\Sigmatildevc$ be a smooth Jordan curve which is the boundary of an end of $\Sigmatildevc$.
	Thus $\mutilde$ is homotopically nontrivial and  $\mutilde$ divides $\Sigmatildevc$ into two parts which contain at least one end.
	Let $\mutilde_n\subset\Sigmatilde_n$ be a sequence of smooth Jordan curves such that $\mutilde_n$ converges to $\mutilde$ as $n\to+\vc$.
	Thus for $n\gg 0$, $\mutilde_n$ is homotopically nontrivial in $\Sigmatilde_n$.
	Indeed, assume that there exists a subsequence of $\mutilde_n$, denoted by $\mutilde_{n_k}$, being homotopically trivial, i.e., $\mutilde_{n_k}$ is the boundary of a disk  $D_{n_k}\subset \Sigmatilde_{n_k}$. 
	Since $D_{n_k}$ stays in a fix ball around $\mutilde_{n_k}$, there exists a subsequence of $D_{n_k}$ converging to a disk in $\Sigmatildevc$ whose boundary is $\mutilde$. This contradicts the fact that $\mutilde$ is homotopically nontrivial.
	
	For each $n$, define a smooth Jordan curve $\mu_n\subset \Sigma_n$ such that $\mutilde_n$ is the image of $\mu_n$ under  the canonical diffeomorphism $\Sigma_n\to \Sigmatilde_n$, i.e., $\lambda_n\Tsf_{p_n}(\mu_n)=\mutilde_n$.
	We see that $\lambda_n\ella{\PSLhR}{\mu_n}=\ella{\Ebb^3_n}{\mutilde_n}$ converges to $\ella{\Rbb^3}{\mutilde}$ as $n\to+\vc$, then $\ella{\PSLhR}{\mu_n}\to 0$ as $n\to+\vc$.
	Since the projection $\pi:\PSLhR\to\Hbbh$ is a Riemannian submersion, the length of the curve $\pi(\mu_n)\subset\Hbbh$ is smaller than or equal to the length of $\mu_n$ in $\PSLhR$, then $\ellH{\pi(\mu_n)}\to 0$ as $n\to+\vc$.
	Hence $\lim_{n \to +\infty} \Distaa{\Hbbh}{\pi(\mu_n),\gamma_1}+\Distaa{\Hbbh}{\pi(\mu_n),\gamma_2} \geq \Distaa{\Hbbh}{\gamma_1,\gamma_2}$.
	Without loss of generality, we assume that there is $d>0$ such that 
	$\Distaa{\Hbbh}{\pi(\mu_n),\gamma_1}\ge d$ for each $n\gg 0$.
	
	\begin{figure}[!h]
	\centering
	\definecolor{qqqqff}{rgb}{0.,0.,1.}
	\definecolor{uuuuuu}{rgb}{0.26666666666666666,0.26666666666666666,0.26666666666666666}
	\begin{tikzpicture}[line cap=round,line join=round,>=triangle 45,x=1.0cm,y=1.0cm]
	\clip(1.8133333333333272,-17.175555555555867) rectangle (6.253333333333334,-12.775555555555783);
	\draw(4.,-1.) circle (2.cm);
	\draw [shift={(4.,-5.013776276774779)}] plot[domain=4.190770683288303:5.234007277481076,variable=\t]({1.*3.48*cos(\t r)+0.*3.48*sin(\t r)},{0.*3.48*cos(\t r)+-1.*3.48*sin(\t r)});
	\draw [shift={(1.312703075574585,-1.)}] plot[domain=-0.3437325085206018:0.343732508520602,variable=\t]({1.*1.919536725832275*cos(\t r)+0.*1.919536725832275*sin(\t r)},{0.*1.919536725832275*cos(\t r)+1.*1.919536725832275*sin(\t r)});
	\draw [shift={(6.687296924425415,-1.)}] plot[domain=-0.3437325085206018:0.343732508520602,variable=\t]({-1.*1.919536725832275*cos(\t r)+0.*1.919536725832275*sin(\t r)},{0.*1.919536725832275*cos(\t r)+1.*1.919536725832275*sin(\t r)});
	\draw [shift={(4.,3.0137762767747796)}] plot[domain=4.456725923684318:4.9680520370850605,variable=\t]({1.*3.48*cos(\t r)+0.*3.48*sin(\t r)},{0.*3.48*cos(\t r)+1.*3.48*sin(\t r)});
	\draw [shift={(4.,-5.01377627677478)}] plot[domain=1.3151332700945253:1.826459383495268,variable=\t]({1.*3.48*cos(\t r)+0.*3.48*sin(\t r)},{0.*3.48*cos(\t r)+1.*3.48*sin(\t r)});
	\draw [rotate around={10.007979801441351:(3.97,-0.8788888888888926)}] (3.97,-0.8788888888888926) ellipse (0.1838516480714256cm and 0.06325684546792722cm);
	\draw [shift={(-2.3802301687837533,4.5569103684212475)}] plot[domain=5.419639477607387:5.600315482814881,variable=\t]({1.*8.221106364068397*cos(\t r)+0.*8.221106364068397*sin(\t r)},{0.*8.221106364068397*cos(\t r)+1.*8.221106364068397*sin(\t r)});
	\draw [shift={(10.443703602566467,4.646648671355426)}] plot[domain=3.8276196024623053:4.006217762089784,variable=\t]({1.*8.331023786753224*cos(\t r)+0.*8.331023786753224*sin(\t r)},{0.*8.331023786753224*cos(\t r)+1.*8.331023786753224*sin(\t r)});
	\draw [domain=1.8133333333333272:6.253333333333334] plot(\x,{(-40.-0.*\x)/5.});
	\draw(4.,-15.) circle (2.cm);
	\draw [shift={(4.,-19.013776276774777)}] plot[domain=4.456725923684318:4.9680520370850605,variable=\t]({1.*3.48*cos(\t r)+0.*3.48*sin(\t r)},{0.*3.48*cos(\t r)+-1.*3.48*sin(\t r)});
	\draw [shift={(4.,-10.986223723225221)}] plot[domain=4.190770683288303:5.234007277481076,variable=\t]({1.*3.48*cos(\t r)+0.*3.48*sin(\t r)},{0.*3.48*cos(\t r)+1.*3.48*sin(\t r)});
	\draw [shift={(1.312703075574585,-15.)}] plot[domain=-0.3437325085206018:0.343732508520602,variable=\t]({1.*1.919536725832275*cos(\t r)+0.*1.919536725832275*sin(\t r)},{0.*1.919536725832275*cos(\t r)+-1.*1.919536725832275*sin(\t r)});
	\draw [shift={(6.687296924425415,-15.)}] plot[domain=-0.3437325085206018:0.343732508520602,variable=\t]({-1.*1.919536725832275*cos(\t r)+0.*1.919536725832275*sin(\t r)},{0.*1.919536725832275*cos(\t r)+-1.*1.919536725832275*sin(\t r)});
	\draw [rotate around={-10.007979801441351:(3.97,-15.121111111111107)}] (3.97,-15.121111111111107) ellipse (0.1838516480714256cm and 0.06325684546792722cm);
	\draw [shift={(-2.3802301687837533,-20.55691036842125)}] plot[domain=5.419639477607387:5.600315482814881,variable=\t]({1.*8.221106364068397*cos(\t r)+0.*8.221106364068397*sin(\t r)},{0.*8.221106364068397*cos(\t r)+-1.*8.221106364068397*sin(\t r)});
	\draw [shift={(10.443703602566467,-20.646648671355425)}] plot[domain=3.8276196024623053:4.006217762089784,variable=\t]({1.*8.331023786753224*cos(\t r)+0.*8.331023786753224*sin(\t r)},{0.*8.331023786753224*cos(\t r)+-1.*8.331023786753224*sin(\t r)});
	\draw (3.76,-14) node[anchor=north west] {$\gamma_1$};
	\draw (3.65,-14.42) node[anchor=north west] {$\triangle_n$};
	\begin{scriptsize}
	\draw [fill=uuuuuu] (3.119953348679835,-0.3531092115512158) circle (1.0pt);
	\draw [fill=uuuuuu] (3.119953348679835,-1.6468907884487842) circle (1.0pt);
	\draw [fill=uuuuuu] (4.880046651320165,-0.3531092115512159) circle (1.0pt);
	\draw [fill=uuuuuu] (4.880046651320165,-1.6468907884487842) circle (1.0pt);
	\draw [fill=uuuuuu] (5.037490673968971,-1.69202777614632) circle (1.0pt);
	\draw[color=uuuuuu] (7.973333333333336,-12.882222222222453) node {$D$};
	\draw [fill=uuuuuu] (2.96140238291898,-1.6923737131288683) circle (1.0pt);
	\draw[color=uuuuuu] (5.893333333333333,-13) node {$\Hbbh$};
	\draw [fill=uuuuuu] (3.9974163578378987,-0.6307880259340667) circle (1.0pt);
	\draw[color=uuuuuu] (7.16,-12.882222222222453) node {$F$};
	\draw [fill=qqqqff] (2.,-8.) circle (1.0pt);
	\draw [fill=qqqqff] (7.,-8.) circle (1.0pt);
	\draw [fill=uuuuuu] (3.119953348679835,-15.646890788448784) circle (1.0pt);
	\draw [fill=uuuuuu] (4.880046651320165,-15.646890788448784) circle (1.0pt);
	\draw [fill=uuuuuu] (4.880046651320165,-14.353109211551216) circle (1.0pt);
	\draw [fill=uuuuuu] (5.037490673968971,-14.30797222385368) circle (1.0pt);
	\draw[color=uuuuuu] (5.173333333333332,-14.148888888889143) node {$p'_1$};
	\draw [fill=uuuuuu] (3.119953348679835,-14.353109211551216) circle (1.0pt);
	\draw [fill=uuuuuu] (2.96140238291898,-14.307626286871132) circle (1.0pt);
	\draw[color=uuuuuu] (3,-14.148888888889143) node {$q'_1$};
	\draw [fill=uuuuuu] (3.9974163578378987,-15.369211974065934) circle (1.0pt);
	\draw[color=uuuuuu] (4.22,-15.37) node {$\xi_n$};
	\end{scriptsize}
	\end{tikzpicture}
	\caption{The triangular domain $\triangle_n$ with vertices $p'_1,q'_1$ and $\xi_n$.}\label{fig3.3.4a}
	\end{figure}
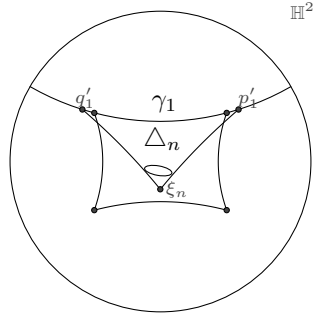

	For $n\gg 0$, the smooth Jordan curve $\mu_n\subset\Sigma_n$ is homotopically nontrivial, then it divides $\Sigma_n$ into two components.
	Denote by $\Acal_n$ the compact sub-annulus of $\Sigma_n$ whose boundary is $\Gamma^1_n$ and $\mu_n$.
	Let two points $p'_1,q'_1$ such that $\gamma_1$ is properly contained in  the geodesics joining $p'_1$ and $q'_1$.
	Since $\pi(\mu_n)$ is contained in the convex bounded  polygon $\Omega$ and $\ellH{\pi(\mu_n)}\to 0$ as $n\to+\vc$, then for $n\gg 0$, there exists a compact geodesic triangular domain of $\Hbbh$ containing $\pi(\mu_n)$ and its vertices are $p'_1,q'_1,\xi_n$. There is a vertex $\xi_n$
such that this domain that we denote by $\triangle_n$
is minimum. Its mean that $\triangle_n$ is the intersection of all such triangular domains, (see Figure \ref{fig3.3.4a}).
Since $\Distaa{\Hbbh}{\xi_n,\gamma_1}>\Distaa{\Hbbh}{\pi(\mu_n),\gamma_1}\ge d$, the value of the angle $\theta_n$ at vertex $\xi_n$ of $\triangle_n$ satisfies $\theta_n<\pi-\theta$ for some $\theta>0$.
We can choose $\theta$ small enough so that $\theta_n\ge \theta$.

	Since $\Acal_n$ is a minimal annulus with  boundary $\Gamma^1_n\cup \mu_n$, $\Acal_n\subset \triangle_n\times\Rbb$ and  $\theta\le \theta_n\le\pi-\theta$ for all $n$ with $\mutilde_n\to \mutilde$ as $n\to +\vc$. Then there exists a subsequence of $\lambda_n\Tsf_{p_n}\left(\triangle_n\times\Rbb\right)$ converging to the closed subset  $S$ of $\Rbb^3$ delimited by two vertical half-planes. The angle of the two half-planes is in the interval $[\theta,\pi-\theta]$. The sequence $\lambda_n\Tsf_{p_n}(\Acal_n)$ converges to a part of $\Sigmatildevc$ delimited by $\mutilde$
and the wedge $S$ contains at least one end of $\Sigmatildevc$.
	This contradicts the fact that an end of a complete embedded minimal surface {\color{black}of finite total curvature} is asymptotic to a plane or an end of a catenoid.	
\end{proof}

We now prove the proposition \ref{lem3.3.12}:

\begin{proof}[\bf Proof of Proposition \ref{lem3.3.12}]

 Let $q_n \in \Sigmatilde_n$, a sequence of points such that $q_n \to q \in \Rbb^3$ as $n \to +\infty$. In a first case,
assume that $\dist_{\Sigmatilde_n} (q_n, \partial \Sigmatilde_n) \geq 1$ for all $n$. Then there exists a disk $B_n \subset \Sigmatilde_n$ of radius $1$ centered at $q_n$ . This disk $B_n$ is locally the euclidean graph with bounded slope of a function $u_n$ on a disc $D_n$ of radius $c_0 > 0$ on its tangent plane. By the Uniform graph lemma and the Arzela-Ascoli theorem, there is a subsequence $B_{n_k}$ which converges to a minimal disk $B_\infty (q)$ in $\Rbb^3$ passing through $q$.

We now consider the case that $\dist_{\Sigmatilde_n} (q_n, \partial \Sigmatilde_n)\leq 1$ for all $n$. In the neighborhood of $q_n$, $\Sigmatilde_n$ is the graph of a function $u_n$ on a domain $U_n$ of its tangent plane where $||\nabla u_n ||$ and $|| \nabla ^2 u_n ||$ are uniformly bounded. Let $D_n$ be the disc of radius $c_0>0$ on its tangent plane at $q_n$. We can insure that  $U_n \subset D_n$ and the boundary of $U_n$ consists of a part of $\partial D_n$ together with a curve $\beta_n$. The graph of $u_n$ on $\beta_n$ is a part of the boundary of $\Sigmatilde_n$. Hence $\beta_n$ is a smooth curve whose curvature goes to $0$ as $n \to +\infty$. By Uniform graph lemma, Arzela-Ascoli theorem and Schauder boundary estimates \cite[Corollary 6.7]{GT01},  there is a subsequence $u_{n_k}$ which converges to a graph of the function $u_{\infty}$ define on a domain $U \subset D_{\infty} (0)$, the tangent plane of a surface 
passing through $q$. The domain $U$ is delimited by the chord $\beta$ of an uniform circle around $q$. Since $u_{n|\beta_n}$ is a part of $\partial \Sigmatilde_n$ which has bounded geometry, the graph of $u$ is linear on $\beta$.

Now, starting at $O \in \Sigmatilde_n$ for all $n$, by analytic continuation and diagonal process, we obtain up to a subsequence, a limit surface $\Sigmatilde_{\infty}$ of $\Rbb^3$. Due to the geometry of $\partial \Sigmatilde_n= T_{p_n} (\Gamma^1_n \cup \Gamma^2_n)$, the limit $\Sigmatilde_{\infty}$ is eventually bounded by a straight line coming from
$\partial \Sigmatilde_n$.

So we have constructed an immersed surface $\Sigmatilde_{\infty}$ in $\Rbb^3$ satisfying the condition 2,3,4 in proposition \ref{lem3.3.12}. Without loss of generality, we assume that the sequence of surfaces $\Sigmatilde_n, n \gg 0$ converges to $\Sigmatilde_{\infty}$ of $\Rbb^3$. Since the sequence $\Gamma^i_n$ has uniformly bounded curvature, if $\partial \Sigmatilde_{\infty}$ is non empty, it is necessarily a straight line.

In order to control the total curvature of $\Sigmatilde_{\infty}$, we need to understand the following lemma and its intersection with some foliations.

\begin{lem}\label{lem2.3.18}
	Let $ N_n $ be a sequence of Riemannian $ 3 $-manifolds converging smoothly to the Riemannian $ 3 $-manifold  $ N $.
	For each $ n $, denote by $ M_n $ a minimal surface in $ N_n $ and by 
	$ \Fcal_{n} $ a minimal foliation of $ N_{n} $
	such that the sequence $ M_n $ converges to a minimal surface  $ M $ in $ N $ with multiplicity $ s $,
	and the sequence $ \Fcal_{n} $ converges to a minimal foliation $ \Fcal $ of $ N $.
	Assume $M$ is tangent to a leaf of the foliation $\Fcal$ at a point $O$, moreover $M$ and this leaf don't coincide in any neighborhood of $O$.
	Then, in any neighborhood of $O$, for $n$ large enough, $M_n$ is tangent to $\Fcal_{n}$ at least at $s$ points.

	%

\end{lem}
\begin{proof}
	Using appropriate charts, without loss of generality, assume that $N  $ is  $ (U,g) $ with $ U\subset \Rbb^3
	$ an open neighborhood of $O$  and $ N_n=\left(U,g_{n}\right) $ where $ g,g_{n} $ are Riemannian metrics satisfying the condition that $ g_{n} $ converges smoothly to $ g $ as $ n\to +\infty $.
	Assume in addition $\Fcal$ is the foliation of level sets of the coordinate function $z  :\Rbb^3\to\Rbb$.
	
	Since $M_n$ converges to $M$ with multiplicity $s$, there is a tubular neighborhood $W_M(O,r,\epsilon)\subset N$ of $O$ such that for $n\gg 0$, $M_n\cap W_M(O,r,\epsilon)$ consists of $s$ normal graphs on the disk $D_M(O,r)$, moreover, these graphs converge to $D_M(O,r)$ as $n\to+\infty$ (Notice that $D_M(O,r)$ is the disk of the radius $r$ about $O$ in $M$ and $W_M(O,r,\epsilon)\subset N$ is the tubular neighborhood of the radius $\epsilon $ about the disk $D_M(O,r)$).
Thus we can assume that $M_n$ is a normal graph on $M$ of a function $u_n$. We need to show that in any neighborhood of $O$, $M_n$ is tangent to $\Fcal_{n}$ for all $n$ large enough.
	
	Denote by $\Nsf_M, \Nsf_{M_n},\Nsf_{\Fcal}, \Nsf_{\Fcal_{n}}$ respectively the unit normal vector fields of $M,M_n$ and $\Fcal,\Fcal_{n}$ for metric $g,g_n$.
	Assume that $d z \left(\Nsf_\Fcal\right)>0$ on $\Rbb^3$, $\Nsf_M=\Nsf_\Fcal$ at $O$ and $\Nsf_{M_n}(p,u_n(p))$ converges uniformly to $\Nsf_M(p)$ as $n\to
+\infty$.
	Define the functions $f,f_n:M\to\Rbb^2$ by
	\begin{align*}
	f(p)&=\left(dx\left(\Nsf_M(p)-\Nsf_\Fcal(p)\right),dy\left(\Nsf_M(p)-\Nsf_\Fcal(p)\right)\right)\\
	f_n(p)&=\left(dx\left(\Nsf_{M_n}(p,u_n(p))-\Nsf_{\Fcal_{n}}(p,u_n(p))\right),dy\left(\Nsf_{M_n}(p,u_n(p))-\Nsf_{\Fcal_{n}}(p,u_n(p))\right)\right).
	\end{align*}

After shrinking $U$ if necessary, we may assume that $d z \left(\Nsf_{M_n}\right)>0$ on $U$ for all $n$ large enough.
	Thus, it suffices to show that for $n$ large enough, there is $p\in U$ so that $f_n(p)=(0,0)$.
	We do it by using winding number as follows.
	Let be a small disk $D\subset\subset U$ about $O$.
	Denote by $\alpha$ the boundary of $D$.
	By hypothesis, $M$ is tangent to $\Lambda_{t_0} \in \Fcal$ at $O$. Then, $f(O)=(0,0)$ and if $\alpha$ is small then $f(p)\ne (0,0)$, for all $p\in\alpha$ and $W\left(\Rest{f}{\alpha},(0,0)\right)\ne 0$ where $W\left(\Rest{f}{\alpha},(0,0)\right)$ represents the total number of times the curve $ \Rest{f}{\alpha} $ travels counterclockwise around the point $ (0,0) $.
	Since $M_n$ converges to $M$, $u_n,f_n$ converge respectively to $0$ and $f$ as $n\to+\infty$.
Hence for $n\gg 0$, we have $W\left(\Rest{f_n}{\alpha},(0,0)\right)\ne 0$ and  there is $p\in D$ such that $f_n(p)=(0,0)$, which proves the lemma.
\end{proof}

We will prove finite total curvature by using a theorem of Mo-Osserman \cite[Theorem 1]{MO90}: A complete minimal surface whose Gauss map take five distinct values only a finite number of times has finite total curvature.

\begin{Asse0}
	The surface $\Sigmatilde_{\infty}$ has finite total curvature.
	\end{Asse0}

	\begin{proof}[Proof of the Claim]
After extracting a subsequence, we can assume that the projection $q_n=\pi (p_n)$ of the maximum of the curvature $p_n$ of $\Sigma_n$ over the plane $\Hbbh$, converges to a point $q_{\infty} \in \Omega$. We consider three distinct foliations of $\PSLhR$ with leaf of bounded geometry, and that are uniformly transverse in a neighborhood of $p_n$. We blow up with $p_n$ translated at $O$ and we obtain at the limit three distinct foliations of $\Rbb^3$ by parallel planes. We prove that $\Sigmatilde_\infty$ has finite number of points tangent to these foliations. 

	\textbf{Case 1.} $\dhr\Sigmatildevc$ is empty set. 
	Denote by $\Fcal^1=\Fcalh$ a foliation by horizontal leaves,  $\Fcal^2 =\{\gamma_t \times \Rbb\}_{t \in [0,1]}$ a foliation by vertical planes with $\gamma_t$ a foliation
from $\gamma_1$ to $\gamma_2$	by geodesics of 
$\Hbbh$, and similarly $\Fcal ^3=\{ \eta_t \times \Rbb \}_{t \in [0,1]}$
a foliation by vertical planes with $\eta_t$ a foliation
from $\eta_1$ to $\eta_2$.
For $j=1,2,3$, the sequences $\lambda_n T_{p_n}(\Fcal^j)$ converge to a foliation $\Fcal_\infty^j$ of $\Rbb^3$, whose each leaf is a plane, hence a foliation by parallel planes. By compacity of $\Omega$, the three foliations $\Fcal^j$ are uniformly transverse to each other and its limits are transverse foliations of $\Rbb^3$. Let 
$q \in  \Sigmatilde_\infty$  a point of tangency of  $\Sigmatilde_{\infty}$
 and $\Fcal_{\infty}^j$ . 
  By lemma \ref{lem2.3.18}, there exists $n$ such that the surface 
$\Sigmatilde_n=\lambda_n T_{p_n} (\Sigma_n)$ is tangent to $\lambda_n T_{p_n} (\Fcal^j)$
in a neighborhood of $q$. 
This give rise to interior tangency of 
$\Sigma_n$ with $\Fcal^j$. 
These interior tangency are in $\Sigmatilde_n$ close to $q$, 
hence the number of tangency of
$\Sigmatilde_\infty$
with $\Fcal_{\infty}^j$
is bounded by the number of internal tangency of
$\Sigma_n$
with
$\Fcal^j$
for $n \gg 0$. 
This is less than 2 by Proposition \ref{pro3.1}. Hence the Gauss map of $\Sigmatilde_\infty$ is normal to the planes of foliations $\Fcal^j$ a finite number of time. For $j = 1,2,3$, it is six values on $\Sbb^2$. We deduce from Mo-Osserman theorem that $\Sigmatilde_n$ converge to a finite total curvature surface $\Sigmatilde_\infty$ with finite multiplicity.
	
\textbf{Case 2.} $\dhr\Sigmatildevc$ is a  straight line $L$. Hence
$\dist_{\PSLhR} (p_n, \partial \Sigma_n) \to 0$
as $n \to+ \infty$ and $q_\infty \in \partial \Omega$. Let be
 $\sigma_L$ the symmetry along $L$.
In this case, we use three foliations $\Fcal^j,j = 1,2,3,$ such that the limits in $\Rbb^3$, $\Fcal_\infty^j$ are invariant by the symmetry $\sigma_L$. 
This need different foliations that depend on the position of $q_\infty$ in regards to its distance with the vertices of $\Omega$.

If $q_\infty$ is a vertex of $\Omega$ (for instance between $\gamma_1$ and 
$\eta_1$), the boundary $L$ of the surface 
$\Sigmatildevc$ is a vertical line. We consider the same foliations $\Fcal^1,\Fcal^2,\Fcal^3$ as in case 1. Then the number of interior tangency of $\Sigmatilde_\infty$ with the 
planes of $\Fcal_\infty^j$ ($j=1,2,3$) are limited for the same reason as in
case 1. It remains to prove that 
$\Sigmatilde_\infty$ has no boundary tangency along $L$. This is immediate for $\Fcal^1$ because $L$ is vertical and each leaf is horizontal. For the others, note that for each $n$ there is no interior contact with $\gamma_1 \times \Rbb$ and 
$\eta_1 \times \Rbb$. Then 
$\lambda_n T_{p_n}(\gamma_1 \times \Rbb)$ and 
$\lambda_n T_{p_n}(\eta_1 \times \Rbb)$ converges respectively to the plane $P_\gamma$ and $P_\eta$  of $\Fcal^2$ and $\Fcal^3$ that contains $L$. By convexity of $\Omega$, the surface $\Sigmatilde_\infty  \backslash L $ is on one side of $P_\gamma$ and $P_\eta$. 
Since $\Sigmatilde_\infty$ is no flat, maximun principle at the boundary prove that no tangency point is possible along $L$ with a leaf containing $L$ i.e. $P_\gamma$ and $P_\eta$.

If $q_\infty$ is on the boundary of $\Omega$, but
not at a vertex, $q_\infty$ is the limit of a sequence of points which are the projection on 
$\Hbbh$ of $p_n \in \Sigma_n$ close to a non vertical part of $\Gamma^1_n \cup \Gamma^2_n$. Hence $q_\infty$ is at some interior point of $\gamma_1$ or 
$\gamma_2$. Suppose $p_n$ accumulate to the upper part of $\Gamma^1_n$ 
(the proof of the lower part is similar).
Define $f_{i,n_0}$ on $\gamma_i$ for $i=1,2$ such that the graph of $f_{i,n_0}$ (up to a constant) represent the upper graph part of 
$\Gamma^i_{n_0}$. For $j = 1, 2, 3,$ consider $w^j$ a minimal graph on
$\Omega$ which assigns $f_{i,n_0}$ on $\gamma_i$ and $g_i^j: \eta_i \to \Rbb$ where we consider three
different sets of boundary data which satisfy the inequality $g_i^1 < g_i^2<g_i^3$ in the interior of $\eta_i$ for $i=1,2$ such that by maximum principle
$$w^1<w^2<w^3 \hbox{ on } \Omega \hbox{ and } w^1=w^2=w^3=f_{i,n_0} \hbox{ on } \gamma_i.$$
By the maximum principle at the boundary, $\Sigma_{n_0}$ is on one side of ${Gr(w^j)}$, without tangent point at the boundary.
 
Denote by $\Fcal^j :=\Fcal^{Gr(w^j)}$, the foliation described in section 3.1-(3) for the function $w^j$. For $n>n_0$, there is an upper leaf (say $\Lambda^j_0$) that meet $\Gamma^i_n$ along a vertical translation of the graph of $w^j$ 
and $\Sigma_n$ is below $\Lambda^j_0$.
Let be $\Fcal_{\infty}^j$ the limit of $\lambda_n T_{p_n} (\Fcal^j)$ for $n \to +\infty$. 
Since the  $\Fcal_{\infty}^j$ for $j=1,2,3$ are obtained by the expression of $Gr(w^j)$ at one point of the boundary, they are distinct family of parallel plane, each one with one leaf  containing $L$, without tangency point with $\Sigmatilde_\infty$ by the maximun principle at the boundary. Then constructing $\Sigmatilde_\infty$ and $\Lambda^j_0$ give rise to a plane $P^j$ passing through $L$ with $\Sigmatilde_\infty$ below this plane for $j=1,2,3$. Since the intersection is $L$, no tangency is possible along $L$ with the surface by maximum principle at the boundary.

Hence as in case 1, the number of interior tangency for $\Sigmatilde_\infty$ with foliations of parallel plane of $\Fcal_\infty^j$, $j=1,2,3,$ is less than two. The foliations 
$\Fcal_\infty^j$ are invariant by 
$\sigma_L$ (one leaf contains $L$) $j = 1,2,3$, and again, an application of Mo-Osserman theorem to $\Sigmatilde_\infty \cup \sigma_L (\Sigmatilde_\infty)$ gives the claim in this case.
\end{proof}

We next show that the surface $\Sigmatildevc$ has no boundary. Suppose the contrary that $\dhr\Sigmatilde_{\infty}\ne\emptyset$.  In this case, describe in case 2, $\dhr\Sigmatilde_{\infty}$ is a straight line $L$.
Moreover, $\Sigmatilde_{\infty}$ is located in a subspace $\Theta$ of $\Rbb^3$ delimited by two planes that have the line of intersection $L$ and whose dihedral angle is less than $\pi$ (consider the wedge bounded by $P_\gamma$ and $P_\eta$ in one case and the wedge bounded by
$P_\gamma$ and $P^1$ in the other case). Let $\Sigma^*_\infty$ (resp. $\Theta^*$) be the union of $\Sigmatildevc$ (resp. $\Theta$) and its image under the symmetry about $L$. By Claim, $\Sigma^*_\infty$ is a complete, embedded  minimal surface of finite total curvature.
	Moreover, $\Sigma^*_\infty$ is not a plane, by \cite{Sch83} and Half-space theorem \cite{HM90}, $\Sigma^*_\infty$ has an end $E$ asymptotic to  end of a catenoid. Since a catenoidal end $E$ does no contains a straightline, $E\subset\Theta$ or $E\subset \Theta^*\setminus \Theta$, both an acute wedge between two planes. Since a catenoidal end cannot be contains in a wedge, this contradicts $\Sigmatilde_\infty$ has a boundary. This finishes the proof of Proposition \ref{lem3.3.12}.	
\end{proof}

\subsection{Some geometric properties of $\Sigma_n$}
\label{subsect3.3.4}
In this section, we study some topological properties of possible limit of a subsequence of $ \Sigma_n $ bounded by 
$\Gamma^1_n \cup\Gamma^2_n$, assuming that $ \Sigma_n $ has uniform bounded curvature.  For each $t \in (-n,n)$, denote by $\Gamma (t)$ the intersection of $\Sigma_n$ with $\{z = t\}$. 
Define $h : \Sigma_n \to \Rbb$ the restriction to $\Sigma_n$ of the coordinate function $\zsf: \PSLhR \to \Rbb$ in the cylinder model. Thus  $\Gamma (t)=h^{-1}(t)$.

A point $p \in \Sigma_n$ is said to be {\it horizontal point} if $p$ is a critical point of the function $h$, i.e., $\Sigma_n$ is tangent to the plane $\{z = h(p)\}$ at $p$. By Proposition \ref{pro3.1}, there are at most two points where $\Sigma_n$ meets $\Fcal^h$ in  a  tangential  way. We denote these horizontal points by $p^+(\Sigma_n),p^-(\Sigma_n)$ which may coincide, with $h(p^+(\Sigma_n)) \geq  h(p^-(\Sigma_n))$. We will denote by $h^+(\Sigma)$(resp. $h^-(\Sigma)$) the value $h(p^+(\Sigma))$ (resp. $h(p^-(\Sigma))$) of the horizontal point $p^+(\Sigma)$ (resp. $p^-(\Sigma)$) of the surface $\Sigma$. For each $t \in (-n,n)$, we define in the cylinder model
 $\Sigma^+_n (t)= \Sigma_n \cap \{ z \geq t \}$ and  $\Sigma^-_n (t) = \Sigma_n \cap \{ z \leq t \}$.

We denote by $\Sigmacheck_n$ a vertical translation of $\Sigma_n$, i.e., $\Sigmacheck_n=\Tsf_h(\Sigma_n)$ for some 
$h\in\Rbb$ that we will mention in each case.

\begin{lem}\label{pro4.4.1}
	We have the following properties of $\Sigma_n$.
	\begin{enumerate}
		\item $\Sigma_n$ has one or two  horizontal points.
		\item If $t>h^+(\Sigma_n)$ (resp. $t<h^-(\Sigma_n)$) then $\Sigma_n^+(t)$ (resp. $\Sigma_n^-(t)$) consists of two simply connected components.
		The curve $\Gamma(t)$ has two connected components diffeomorphic to $[0,1]$. Each one of them is joining two points in a same component of $\dhr\Sigma_n$.
		\item If $h^+(\Sigma_n)-h^-(\Sigma_n)> \sqrt{1+4\tau^2}\pi$ then
		\begin{enumerate}
			\item For each $t\in (h^-(\Sigma_n),h^+(\Sigma_n))$, $\Sigma^+_n(t)$ and $\Sigma^-_n(t)$ are simply connected. Moreover, $\Gamma(t)$ consists of two components diffeomorphic to $[0,1]$. Each one is joining two points in two distinct components of $\dhr\Sigma_n$.
\item The set $\Sigma_n\cap \left\{h^-(\Sigma_n)< z <h^+(\Sigma_n) \right\}$ consists of two simply connected components.
		\end{enumerate}
	\end{enumerate}
\end{lem}

\begin{figure}[h!]

\hskip 3cm
	\includegraphics[width=0.3\linewidth]{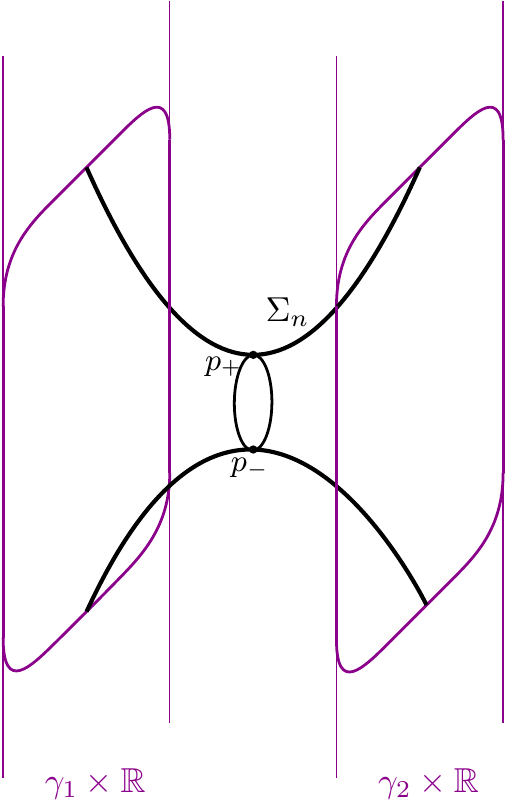}
	\hskip 2cm
	\includegraphics[width=0.3\linewidth]{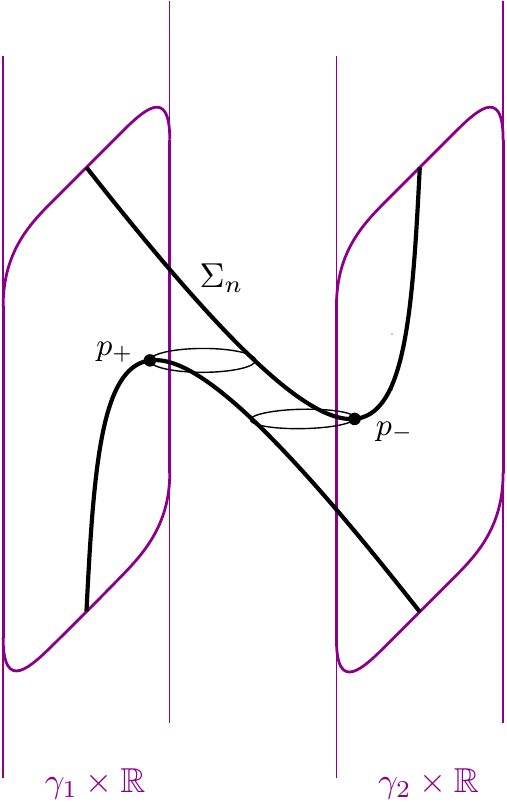}
	\caption{(Left)$\Gamma (h^+(\Sigma_n))$has no cycle.
	\hskip 0.3cm (Right) $\Gamma (h^+(\Sigma_n))$ contains a cycle. }
	\label{fig3.7}
\end{figure}
\begin{proof}By Proposition  \ref{pro3.1}, $\Sigma_n$ has at most two tangent points with the foliation $\Fcal^h$, this prove 1). By the maximum principle $\Sigma_n$ has no interior point of contact with $\{ z=n \}$. Consider the vertical plane $\eta_{12} \times \Rbb$
with $\eta_{12}$ a geodesic segment that joins the midpoints of $\eta_1$ and $\eta_2$ which doesn’t intersect  $\gamma_1 \cup \gamma_2$ (by convexity of $\Omega$). Then $(\eta_{12} \times \Rbb) \cap \Sigma_n$ is compact and $h$  attains its maximun value at $t_0 <n $. This prove that $\Gamma (t) \cap (\eta_{12} \times \Rbb) = \emptyset$ for 
$t_0 < t \leq n$.
Hence $\Gamma (t)$ has two connected components, one connecting points of $\Gamma^1_n$ and one connectiong points of $\Gamma^2_n$. By Morse theory, $\Gamma (t)$ have a change of topology with a critical point of the function $h$. This explain that $\Sigma_n^+ (t)$  has two connected components for $t>h_n^+$.

To prove 3), we remark that the situation describe in Figure \ref{fig3.7} (Right) concerns a case where $\Gamma(h^+(\Sigma_n))$ has non trivial cycle. In this case, this define an annulus $\Acal$ contains in $\Omega \times \Rbb$ with boundary contains in $\{ z= h^+(\Sigma_n)\}$
and $\{ z=h^-(\Sigma_n)\}$.  We  consider the one parameter family of rotational examples called rotational catenoid parameterized by the size of its neck (C. Penafiel  \cite{Pen}, section 3). The boundary at infinity of these annuli is contained in $\{h^-(\Sigma_n)<z<h^-(\Sigma_n)+  \sqrt{1+4\tau ^2} \pi \}$. There is one example with large neck which contains $\Omega \times \Rbb$ in one side. To find a contradiction with the maximun principle, it suffices to decrease the size of the neck up to a first point of contact with $\Acal$ which cannot be on the boundary when $h^+(\Sigma_n) -h^-(\Sigma_n) > \sqrt{1+4\tau ^2} \pi $. In this case $\Gamma (t) $ cannot contains any non  trivial cycle. Hence we are in the situation of the Figure \ref{fig3.7} (Left) describe in (3).
\end{proof}

By Theorem \ref{thm4.2.6}, there exists a minimal solution $u^\pm_i$ on $\Omega$ such that $u^\pm_i=\pm\infty$ on $\gamma_i$ and 
$u^\pm_i=0$ on  $\eta_1\cup \eta_2\cup \gamma_j$ with $i \neq j$ and  $i,j=1,2$.
Define $u^+=\sup\{u^+_1,u^+_2\}$ and $u^-=\inf\{u^-_1,u^-_2 \}$. 
The following lemma is similar in spirit to a theorem of Jenkins-Serrin \cite[Theorem 4]{JS66}.

\begin{lem}\label{lem3.3.14}
	For each $n$, the minimal surface $\Sigma_n$ is below the graph of $u^+ + h^+(\Sigma_n)$ and above the graph of $u^-+h^-(\Sigma_n)$.
\end{lem}

\begin{proof}
We prove that $\Sigma_n$ is below the graph of $u^+ + h^+(\Sigma_n)$.
Let $\Sigmacheck_n$ be the vertical translation of $\Sigma_n$ such that $h^+(\Sigmacheck_n)=0$. We need to prove that $\Sigmacheck_n$ is below $\Gr(u^+)$ for all $n$. Since $\Gr(u^+)$ is above $\{z=0\}$, it suffices to show that $\Sigmacheck_n\cap \{z>0\}$ is below $\Gr(u^+)$. By Proposition \ref{pro4.4.1}, $\Sigmacheck_n\cap \{z>0\}$ consists of the two simply connected components $\Delta_1$ and $\Delta_2$ whose boundaries are in $\left(\gammahat_1\times\Rbb\right)\cup \{z=0\}$ and $\left(\gammahat_2\times\Rbb\right)\cup \{z=0\}$ respectively. We will prove that $\Delta_i$ is below the minimal graph $\Gr(u^+_i)$ for $i=1,2$.

Let $\Omega'\supset\Omega$ be a bounded convex quadrilateral domain whose boundary is composed of open geodesics 
	$\gamma'_1,\eta'_1,\gamma'_2,\eta'_2$ in this order together with their endpoints, moreover 
	$\gamma_1\subset\subset\gamma'_1$, $\gamma_2\subset\subset\gamma'_2$.
By Theorem \ref{thm4.2.6}, there exists a minimal solution $v_i$ on $\Omega'$ that $v_i=+\infty$ in 
	$\gamma'_i$ and 
	$v_i=0$ on 
	$\eta'_1\cup \eta'_2\cup \gamma'_j$ with $i \neq j$ and  $i,j=1,2$.
	When $\Omega'$ converges to $\Omega$ then $v_i$ converges to $u^+_i$.
	So it suffices to show that $\Delta_i$ is below $\Gr(v_i)$.
	For $h\ll 0$, we see that $\Tsf_h(\Delta_i)$ is below $\Gr(v_i)$. Let increase $h$ and suppose that $\Tsf_h(\Delta_i)$ touches $\Gr(v_i)$ at the   point $p$ with $h(p)=h_0$.
	By the maximum principle, the first point of contact $p$ can not be an interior point of both $\Gr(v_i)$ and $\Tsf_{h_0}(\Delta_i)$.
	Moreover, we see that $\dhr\Gr(v_i)\cap \Tsf_h(\Delta_i)=\emptyset$ for all $h$. Hence, 
	$p\in \dhr \Tsf_{h_0}(\Delta_i)\cap \Gr(v_i)$. However, $\dhr \Tsf_{h_0}(\Delta_i)\cap \Gr(v_i)=\emptyset$ when $h\le 0$.
This means $h_0>0$. So $\Delta_i$ is below $\Gr(v_i)$. The same arguments prove that $\Sigma_n$ is above the graph of $u^-+h^-(\Sigma_n)$.
\end{proof}

\begin{lem}\label{lem3.3.15}
	Let $\Xi$ be a domain in $\Hbbh$ delimited by two half geodesics $\gamma,\eta$ starting at $O$ and forming an angle $<\pi$.
	Let $\Sigma \subset \Xi\times\Rbb$ be a minimal surface in $\PSLhR$ and a point $p\in \Sigma$ satisfying $\chuan{A_\Sigma}\le C$ and $\Dist_\Sigma(p,\dhr \Sigma)\ge \delta$.
	Then there is $\epsilon=\epsilon(C,\delta)$ such that $p\notin D_\epsilon(O)\times\Rbb$
where $D_\epsilon(O)$ is the disk of the radius  $\epsilon$ about $O$  in $\Hbbh$.
\end{lem}
\begin{proof}
	Assume the contrary that for each $n>0$, there exists a minimal $\Sigma_n$ and a point $p_n\in \Sigma_n$ such that $\chuan{A_{\Sigma_n}}\le C$, $\Dist_{\Sigma_n}(p_n,\dhr \Sigma_n)\ge \delta$ and $p_n\in D_{1/n}(O)\subset \Hbbh$.
Denote by $\widetilde{\Sigma}_n,\widetilde{p}_n$ the image of $\Sigma_n,p_n$ under the conformal diffeomorphism $\PSLhR\to \Ebb^3\left(\frac{-1}{n^2},\frac{\tau}{n}\right), (x,y,z)\mapsto (nx,ny,nz)$. We know that the sequence
$\Ebb^3\left(\frac{-1}{n^2},\frac{\tau}{n}\right)$ converges to the Euclidean space $\Rbb^3$ as $n\to+\infty$.
Then, in the Riemannian manifold $\Ebb^3\left(\frac{-1}{n^2},\frac{\tau}{n}\right)$, we have  $\widetilde{p}_n\in D_1(O)\subset \Hbb^2\left(\frac{-1}{n^2}\right)$, $\Dist_{\widetilde{\Sigma}_n}\left(\widetilde{p}_n,\dhr\widetilde{\Sigma}_n\right)\ge n\delta$, $\chuan{A_{\widetilde{\Sigma}_n}}\le \frac{C}{n}$ and $\widetilde{\Sigma}_n$ is minimal. 
The sequence of minimal surfaces $\widetilde{\Sigma}_n$ has a subsequence converging to a whole plane and this plane is contained in a domain of $\Rbb^3$ delimited by two half-planes forming a wedge with angle $<\pi$, which gives a contradiction.	
\end{proof}

\begin{prop}\label{pro3.3.15b}
	Let $\Sigmacheck_n$ be the image of $\Sigma_n$ under a vertical translation such that  $h^+(\Sigmacheck_n)=0$  and assume that $\left\{h^-(\Sigmacheck_n) \right\}_n$ converges to $-\vc$ as $n\to+\vc$. Then there exists a subsequence of  $\Sigmacheck_n$ converging to a simply connected minimal surface $\Sigmacheck_\vc$ with multiplicity one.
\end{prop}

\begin{proof}   It follows from the uniform bound for the curvature of 
$\Sigma_n$ that $\Sigmacheck_n$ has also uniform bound for the curvature. The sequence $p^+(\Sigmacheck_n) \in \Omega \times \{0\}$ has a convergent subsequence to some point $p_\infty$ and there is a subsequence of  $ \Sigmacheck_n $ which converges to a surface  $\Sigmacheckvc$ passing through $p_\infty$ with multiplicity smaller or equal to $2$ by Lemma \ref{lem2.3.18}.
	
Now we prove that the multiplicity is $1$. By Lemma \ref{pro4.4.1}, for any $t<0$, $\Sigmacheck^+_n (t)$ is simply connected for $n$ large enough.  The curve $\Gamma (0)=\Sigmacheckvc \cap \{ z =0 \}$ is passing through $p_\infty$, has no cycle and connect the four vertical lines $\{p_i\} \times \Rbb$ and $\{q_i\} \times \Rbb$ passing by the vertices of $\Omega$.  Hence
the limit $ \Sigmacheckvc$ contains half-line of each vertical line along its boundary. Consider a boundary point $p \in \dhr  \Sigmacheckvc \cap (\{q_i\} \times \Rbb)$, and a ball $B_\epsilon (p)$ which intersect the surface in a half-disc $D_\infty=
\Sigmacheckvc \cap B_\epsilon (p) $. The half disc $D_\infty$ is the limit of half-discs $D_n \subset \Sigmacheck_n $. By Lemma \ref{lem3.3.15}, $D_n$ is the unique component of  $\Sigmacheck_n \cap (D_\epsilon (q_i) \times \Rbb)$ intersecting the ball $B_\epsilon (p)$. If not (in the case where the multiplicity is 2), it would have a second component of $\Sigmacheck_n$ into $B_\epsilon (p) \cap (D_\epsilon (q_i) \times \Rbb)$ converging to $D_\infty$ and any point close to of these second component will be the center of a disc into the surface of radius  $\delta>\epsilon/2>0$. This prove that the mulitiplicity is one close to the bounndary, hence one on $\Sigmacheckvc$ by connectivity.

We next show that $\Sigmacheck_\vc$ is simply connected. 
Let $\beta :\Sbb^1\to \Sigmacheck_\vc$ be a closed continuous curve.  We have to show that $\beta$ is the boundary of a disk of $\Sigmacheck_\vc$. Since $\Sigmacheck_\vc$ is the limit of the sequence of minimal surfaces $\Sigmacheck_n$, there exists
a sequence of closed curves $\beta_n:\Sbb^1\to\Sigmacheck_n$ such that $\beta_n\to \beta$ as $n\to
+\infty$. Define $\Acal_{n}=\Sigmacheck_{n}\cap\left\{z\ge h(p^-(\Sigmacheck_{n}))/2\right\}$.
Since $\beta_n\to\beta$ as $n\to +\infty$, and  $h^-(\Sigmacheck_n)\to-\vc$ as $n\to+\vc$ so $\beta_n\subset\Acal_n$ for $n$ large enough. By the third assertion of Proposition \ref{pro4.4.1}, $\Acal_n$ is simply connected, so $\beta_n$ is the boundary of a disk $D_n\subset\Sigmacheck_n$. In fact, each $D_n$ is contained in a convex ball around $\beta_n$, so there exists a subsequence of $D_n$ converging to a minimal disk $D\subset \Sigmacheck_\vc$ whose boundary is $\beta$, this proves $\Sigmacheck_\vc$ is simply connected. 
\end{proof}

In view to prove that the limit surface $\Sigmacheck_\vc$ is a graph on $\Omega$ in the Proposition \ref{pro3.3.15a}, we need to study subset of $\Sigmacheck_\vc$ with boundary on some vertical plane.

\begin{lem}\label{lem2.1.3}
	Let $\etahat \subset \Hbbh$ be a complete geodesic with endpoints at infinity $\widehat{q}_1, \widehat{q}_2$ and $\etahat (s)$ the set of equidistant curve to $\etahat$ into $\Hbbh$ with algebraic distance $s \in \Rbb$. Let $\Sigma\subset\PSLhR$ be a proper immersed minimal surface with bounded curvature and
\begin{equation}\label{equ4.2.7}
\left(\partial \Sigma \cap \{ M_1\leq z \leq M_2 \}\right) \subset\left(\etahat \times\Rbb\right)   \hbox{ and }
\left(  \dhrinfty \Sigma \cap \{ M_1\leq z \leq M_2 \} \right) \subset \left(\{\widehat{q}_1,\widehat{q}_2\}  \times \Rbb\right).
\end{equation} 
	
	\noindent
	a) For any constant $\epsilon >0$ there exists a constant $d(\epsilon)>0$, such that if $M_2-M_1>2d(\epsilon)$ we have
	$$\Sigma \cap \{ M_1 +2d(\epsilon) \leq z \leq M_2- d(\epsilon) \} \subset \Omega (\epsilon) \times \Rbb,$$
	where $\Omega (\epsilon)$ is the convex domain bounded by $\etahat(-\epsilon)$ and $\etahat (\epsilon)$.
	\vskip 0.3cm
	\noindent
b)	If $M_1=-\infty$, then for any constant $\epsilon >0$ there exists a constant $d(\epsilon)>0$ such that
$$\Sigma \cap \{ z \leq M_2- d(\epsilon) \} \subset \Omega (\epsilon) \times \Rbb,$$	
and for $\epsilon>0$ small enough, each connected component of $\Sigma \cap \{ z \leq M_2- d(\epsilon) \}$ is the horizontal graph of a function $v: \etahat \times \Rbb \to \Rbb$, which satisfy uniformly that  $|v(q,z)| + |\nabla v (q,z)|  \to 0$ as $z \to -\infty$.
		\vskip 0.3cm
	\noindent
	c) If $M_2=+\infty$ and $M_1=-\infty$, then $\Sigma \subset (\etahat \times \Rbb).$

\end{lem}
\begin{proof}
The proof of this lemma depends on the existence of a one parameter family of embedded minimal surface with parameter $\epsilon>0$ describe by C. Penafiel (see \cite{Pen}, Section 5) used as barriers. For any geodesic $\etahat$ and $\epsilon>0$, there exists an embedded minimal surface $S_{\etahat, \epsilon}$ which is a bigraph over the exterior of $\Omega (\epsilon)$, with boundary $\etahat (\epsilon)$. The surface $S_{\etahat, \epsilon}$ is foliated by horizontal lift of equidistant curve $\etahat (s)$ for $s \in [\epsilon, +\infty[$. The surface $S_{\etahat, \epsilon}$ contains the horizontal lift of $\etahat (\epsilon)$, a notion of equatorial curve for these surfaces, where the tangent plane is vertical. There exists a constant $c(\epsilon)>0$ depending on $\epsilon$ and $\alpha$ depending on  $\widehat{q}_1, \widehat{q}_2$ such that the boundary   of $S_{\etahat, \epsilon}$ at infinity is the set of two vertical segments $\{\widehat{q}_1 \}\times [-c(\epsilon), c(\epsilon)]$ and $\{\widehat{q}_2 \}\times [-c(\epsilon)+\alpha, c(\epsilon)+\alpha]$ and two horizontal lifts of a horizontal segment of $\partial_{\infty}\Hbbh$ with endpoints $(\widehat{q}_1, -c(\epsilon))$  and $(\widehat{q}_2 ,-c(\epsilon)+\alpha)$ and respectively one curve at infinity with endpoints $(\widehat{q}_1,c(\epsilon))$ and 
$(\widehat{q}_2 ,c(\epsilon)+\alpha)$. Moreover when $\epsilon \to 0$, the constant $c(\epsilon) \to + \infty$.   We have to remark that one can translate this surface using horizontal lift of hyperbolic translation in $\Hbbh$. Doing this, one obtain surfaces $S_{\etahat', \epsilon}$ with constant $\alpha$ changing but  uniformly bounded. Let denote by $\alpha_0$ the maximum value of constant $\alpha$ for  any horizontal lift of hyperbolic translation and let $d(\epsilon)=c(\epsilon)+\alpha_0$.

For $M_2-M_1>3d(\epsilon)$ large enough, we consider $h_2$ such that 
$\Tsf_{h_2} S_{\etahat, \epsilon}$ has the following properties:
it is a bigraph on the exterior  of 
$\Omega(\epsilon)$, the equator curve is contained in $\{ z \geq M_2-d(\epsilon)\}$, the surface $\Tsf_{h_2}S_{\etahat, \epsilon} \subset \{ M_2-2d(\epsilon) \leq z \leq M_2\}$   and (\ref{equ4.2.7}) imply that $( \partial_\infty \Tsf_{h_2}S_{\etahat, \epsilon}  \cap \partial_\infty \Sigma ) \subset  \left(\{\widehat{q}_1,\widehat{q}_2\}  \times \Rbb\right)$.
Using horizontal lift of hyperbolic translation in $\Hbbh$, we obtain a surface 
$\Tsf_{h_2}S_{\etahat', \epsilon}$ with  $(\partial_\infty \Tsf_{h_2}S_{\etahat', \epsilon}  \cap \partial_\infty \Sigma) =\emptyset$ if $\etahat' \cap \etahat = \emptyset$ and no common points at infinity. Since $\Sigma$ is properly immersed, $\Sigma \cap  \Tsf_{h_2}S_{\etahat', \epsilon}$ is a compact curve which bound a compact subset $K$ of $\Sigma$. The set of horizontal lift of hyperbolic translation of $\Tsf_{h_2}S_{\etahat, \epsilon}$ foliate the space, hence one can find a last point of contact with $K$. This prove that $K=\emptyset$ by the maximum principle and the surface $\Sigma$ is in one side of $\PSLhR \setminus \Tsf_{h_2}S_{\etahat, \epsilon}$. Using vertical translation, $\Tsf_{h}S_{\etahat, \epsilon}$ has no point of contact with $\Sigma$ for any $h \in [h_1, h_2]$ with $h_1=h_2-(M_2-M_1)+2d(\epsilon)$ and one can conclude a) by the maximum principle.

Now when $M_1=-\infty$, for any $\epsilon >0$, there exists $h_2(\epsilon)$, such that the surface $\Tsf_{h_2}S_{\etahat, \epsilon} \subset \{ M_2-2d(\epsilon) \leq z \leq M_2\}$ which does no intersect $\Sigma$ (here the constant $d(\epsilon)$ is  large depending on $\epsilon$). For any  $\epsilon >0$, one can translate vertically the surface for any $h \in (-\infty, h_2]$ without intersecting $\Sigma$. This prove that $\Sigma \cap \{ z \leq M_2-d(\epsilon)\} \subset  \Omega (\epsilon) \times \Rbb$.
Now let use Fermi coordinate $(q,z,s)$ where $(q,z)$ are coordinate on $\etahat \times \Rbb$ and $s \in [-\epsilon, \epsilon]$ the coordinate along horizontal geodesic orthogonal to the vertical plane. The horizontal vector field $n(s)$ will denote the tangent field along this horizontal geodesic at distance $s$ of $\etahat \times \Rbb$.  Since $\Sigma$ has uniform bounded curvature, it is around any point $p$, a graph of height $\delta >0$ on its tangent plane $T_p \Sigma$ over a disc $D_{\kappa}$ of radius $\kappa >0$. If $N(p)$ denote the unit normal vector at $p \in \Sigma$, for $\epsilon>0$ small enough, there exists a constant $0<C(\epsilon,\delta)<1$ such that  $|\vh{N(p),n(s(p)}| \leq C$ imply that the graph of $\Sigma$ over $D_{\kappa}$ cannot be contains in $\Omega (\epsilon) \times \Rbb$. This prove that $\Sigma$ is the horizontal graph $(q,z, v(q,z))$ of the function $v : \etahat \times \Rbb \to \Rbb$ which satisfy $|v|  \leq \epsilon$ for $z \leq M_2 -d(\epsilon)$. Letting $\epsilon \to 0$, the constant $C(\epsilon,\delta) \to 1$  prove
that the function $|v(q,z)| \to 0$ and $|\nabla v (q,z)| \to 0$ uniformly in $q \in \etahat$, when $z \to -\infty$.

When $M_2=+\infty$ and $M_1=-\infty$, we can redo the same argument for any
$\epsilon >0$ and any vertical translation $h \in \Rbb$. This prove c).
\end{proof}

\begin{prop}\label{pro3.3.15a}
The simply connected minimal surface $\Sigmacheck_\vc$ in Proposition \ref{pro3.3.15b} is a vertical graph on $\Omega$.
\end{prop}

\begin{proof}

To prove this proposition, we assume the contrary that there is a point $ p\in\Sigmacheckvc $ such that $T_p\Sigmacheck_\vc $ is vertical.
There is a unique complete geodesic $ \gammahat \subset\Hbbh$ with end points   $\widehat{p}, \widehat{q}$ such that  $\gammahat \times\Rbb $ is tangent to $ \Sigmacheckvc  $ at $ p$, and the intersection $ (\gammahat \times\Rbb)\cap\Sigmacheckvc=\Gamma $ is locally a set of $ k $ curves passing through $p$, with $ k\ge 2 $.  

Then locally at $p$, $ \Sigmacheckvc\setminus\Gamma $ has at least four local connected components which are alternately on one side and on the other one side  of $\gammahat \times \Rbb$. If two of them, say $ \Sigma_1$ and $\Sigma_2 $ connect into $ \Sigmacheckvc\setminus\Gamma$, there is a path $ \alpha_0 $
joining two points, $x$ in $\Sigma_1 $ with $y$ in $\Sigma_2$ close to $p$. Then join  $x$ to $y$ by a local path $\beta_0$ going through $p$ to  obtain a  compact cycle $(\alpha_0 \cup \beta_0) \subset  \Sigmacheckvc\setminus\Gamma$. This cycle is in one side of $ \PSLhR\setminus(\gammahat \times\Rbb)$   and bound a compact disk (since $ \Sigmacheckvc $ is simply connected) which intersect the other side in a non emptyset, contradicting the maximum principle with a foliation by vertical planes (see Figure  \ref{fig3.4.2}-(left)).

This proves that $ \Sigmacheckvc\setminus\Gamma $ has at least 4 connected components. Since there does not exist any non empty open subset (compact or non compact) of $\Sigmacheckvc \setminus\Gamma $ whose boundary is contained in $\gammahat \times \Rbb$ and $\dhr_{\infty} \Sigma
\subset  \left(\{\widehat{p},\widehat{q}\}  \times \Rbb\right)$ by Lemma 3.11-b), we proved the claim:

\begin{Clm}
$ \Sigmacheckvc\setminus\Gamma $ has at least $ 4 $ connected components, each one has in its boundary a connected component of $ \dhr\Sigmacheckvc\setminus\Gamma $.
\end{Clm}


We consider two cases depending on the number of components of the boundary of $ \Sigmacheckvc $.
\vskip 0.25cm

\textbf{Case 1.} When $ n-h^+(\Sigma_n) $ is bounded, the boundary $ \dhr\Sigmacheckvc $ consists of two connected components $ \Gamma^1_\infty $ and $ \Gamma^2_\infty $ with  $ \Gamma^i_\infty\subset\gamma_i\times\Rbb $ for $ i=1,2 $.

\begin{figure}[!h]
	\centering
	\definecolor{ffqqqq}{rgb}{1.,0.,0.}
	\definecolor{xdxdff}{rgb}{0.49019607843137253,0.49019607843137253,1.}
	\definecolor{ffxfqq}{rgb}{1.,0.4980392156862745,0.}
	\definecolor{qqqqff}{rgb}{0.,0.,1.}
	\begin{tikzpicture}[line cap=round,line join=round,>=triangle 45,x=1.0cm,y=1.0cm]
	\clip(-6.34,0.92) rectangle (6.96,4.84);
	\fill[line width=0.4pt,color=ffxfqq,fill=ffxfqq,fill opacity=0.1] (-6.4,3.9) -- (7.06,3.92) -- (7.06,1.92) -- (-6.4,1.9) -- cycle;
	\draw [line width=0.4pt,color=ffxfqq] (7.06,3.92)-- (7.06,1.92);
	\draw [line width=0.4pt,color=ffxfqq] (-6.4,1.9)-- (-6.4,3.9);
	\draw [line width=2.8pt] (-0.9000418607731877,3.908172300355463)-- (0.4599848541717275,3.9101931424281897);
	\draw [line width=2.8pt] (-0.9399526196100938,1.908112997593447)-- (0.42016324818404627,1.9101339721369748);
	\draw [shift={(-6.33960342515995,2.972041416850006)},dash pattern=on 5pt off 5pt]  plot[domain=-0.4417291342498535:0.38186868216627784,variable=\t]({1.*2.4996160949468047*cos(\t r)+0.*2.4996160949468047*sin(\t r)},{0.*2.4996160949468047*cos(\t r)+1.*2.4996160949468047*sin(\t r)});
	\draw [shift={(7.535089275870371,2.970805584196651)},dash pattern=on 5pt off 5pt]  plot[domain=2.806380367417612:3.5169492449827824,variable=\t]({1.*2.875237495953841*cos(\t r)+0.*2.875237495953841*sin(\t r)},{0.*2.875237495953841*cos(\t r)+1.*2.875237495953841*sin(\t r)});
	\draw (1.54,4.68) node[anchor=north west] {$\Gamma^1_\infty$};
	\draw (1.58,1.84) node[anchor=north west] {$\Gamma^2_\infty$};
	\draw (-3.66,3.5) node[anchor=north west] {$\eta^\infty_1(-t)$};
	\draw (3.5,3.48) node[anchor=north west] {$\eta^\infty_2(-t)$};
	\draw (-6.16,3.46) node[anchor=north west] {$\infty_1$};
	\draw (6.22,3.44) node[anchor=north west] {$\infty_2$};
	\draw (-1.46,2.88) node[anchor=north west] {$\check{\Sigma}_\infty$};
	\draw [color=ffqqqq,domain=-6.34:6.96] plot(\x,{(--5.317037577550618--0.0020208420727265697*\x)/1.3600267149449152});
	\draw [color=ffqqqq,domain=-6.34:6.96] plot(\x,{(--2.5971543858874435--0.002020974543527787*\x)/1.36011586779414});
	\begin{scriptsize}
	\draw [fill=qqqqff] (-6.4,3.9) circle (1.0pt);
	\draw [fill=qqqqff] (7.06,3.92) circle (1.0pt);
	\draw [fill=qqqqff] (7.06,1.92) circle (1.0pt);
	\draw [fill=qqqqff] (-6.4,1.9) circle (1.0pt);
	\draw [fill=xdxdff] (-0.9000418607731877,3.908172300355463) circle (1.0pt);
	\draw [fill=xdxdff] (0.4599848541717275,3.9101931424281897) circle (1.0pt);
	\draw [fill=xdxdff] (-0.9399526196100938,1.908112997593447) circle (1.0pt);
	\draw [fill=xdxdff] (0.42016324818404627,1.9101339721369748) circle (1.0pt);
	\draw [fill=xdxdff] (-4.020034972291524,3.9035363521957036) circle (1.0pt);
	\draw [fill=xdxdff] (-4.0799159693550875,1.9034473759742123) circle (1.0pt);
	\draw [fill=xdxdff] (4.819886075110945,3.916671450334489) circle (1.0pt);
	\draw [fill=xdxdff] (4.860034574879121,1.9167311063519747) circle (1.0pt);
	\draw [fill=qqqqff] (1.32,3.04) circle (1.0pt);
	\draw[color=qqqqff] (1.46,3.28) node {$p$};
	\end{scriptsize}
	\end{tikzpicture}
	\caption{Case $\dhr\Sigmacheckvc$ consists of two components.}\label{fig3.3.4}
\end{figure}

By Claim 1, $ \dhr\Sigmacheckvc\setminus\Gamma $ has at least $4$ connected components and the geodesic $\gammahat$ has to intersect the interior of $\gamma_1 $ and $ \gamma_2$ (If $ \gammahat$ passes through a vertex of $ \Omega $, $ \dhr\Sigmacheckvc\setminus\Gamma $ will have at most $3$ connected components, contradicting Claim 1). Since $ \Omega $ is convex, $\gammahat$ does not intersect $\eta_1$ and  $\eta_2$.
Since $h^+(\Sigmacheck_n)-h^-(\Sigmacheck_n) \to -\infty$, Lemma \ref{pro4.4.1}-3) assure that for any $t>0$, the set
$\Gamma(-t)=\Sigmacheckvc \cap\{ z=-t\}$ will be two curves  
$\eta_1^{\infty}(-t)$ and $\eta_2^{\infty}(-t)$ connecting $\Gamma^1_\infty$ and $ \Gamma^2_\infty$ (see Figure \ref{fig3.3.4} for a conformal representation of
$\Sigmacheckvc$). We consider $\Sigma_1$ the connected component of
$\Sigmacheckvc \cap \{ z<0\}$ with $\partial \Sigma_1 \subset (\eta_1 \times \Rbb) \cup \{z=0\}$. $\Sigma _1$ satisfy hypothesis of  Lemma \ref{lem2.1.3}-b).
For any $\epsilon >0$, there exists $M<0$ such that  $(\Sigma_1 \cap \{z\leq M  \})\subset \Omega(\epsilon) \times \Rbb$ and then for any $p \in \Sigma_1 \cap \{z\leq M \}$, we have  $ \dist_{\Hbbh}(\pi(p),\eta)\le\epsilon$. 
This imply that for $t>0$ large enough, the curve $\eta_1^{\infty}(-t)$ is a curve
close enough to $\eta_1$.

Doing the same for $\eta_2$, this shows that $ (\gammahat \times\Rbb) \cap\Sigmacheckvc $ is compact and for $t>0$ large enough $ \gammahat \times\Rbb  $ separates the two components of $ \{z=-t\}\cap\Sigmacheckvc = \{\eta^{\infty}_1(-t),\eta^{\infty}_2(-t)\} $. We remark now that $ D=\Sigmacheckvc\cap\{z\ge -t\} $ is a disk bounded by one curve $ (\dhr\Sigmacheckvc\cap\{z\ge -t\}) \cup \eta_1(-t)\cup\eta_2(-t)=\dhr D $ and $ \gammahat \times\Rbb $ intersects $ \dhr D $ in exactly two points, contradicting Lemma \ref{lem1.haus}-a).

\vskip 0.25cm
\textbf{Case 2} When $ n-h^+(\Sigma_n)\to+\infty $ as $ n\to +\infty $, the boundary of $ \Sigmacheckvc $ consists of $ 4 $ vertical lines passing through the vertex of $ \Omega $.

	\begin{figure}[!h]
		\centering
		\definecolor{qqqqff}{rgb}{0.,0.,1.}
		\definecolor{xdxdff}{rgb}{0.49019607843137253,0.49019607843137253,1.}
		\definecolor{ffqqqq}{rgb}{1.,0.,0.}
		\begin{tikzpicture}[line cap=round,line join=round,>=triangle 45,x=1.0cm,y=1.0cm]
		\clip(-4.35777777777778,-3.346666666666676) rectangle (4.348888888888887,3.3466666666666662);
		\draw [samples=50,domain=-0.99:0.99,rotate around={-135.:(0.,0.)},xshift=0.cm,yshift=0.cm,color=ffqqqq] plot ({1.4142135623730951*(1+(\x)^2)/(1-(\x)^2)},{1.4142135623730951*2*(\x)/(1-(\x)^2)});
		\draw [samples=50,domain=-0.99:0.99,rotate around={-135.:(0.,0.)},xshift=0.cm,yshift=0.cm,color=ffqqqq] plot ({1.4142135623730951*(-1-(\x)^2)/(1-(\x)^2)},{1.4142135623730951*(-2)*(\x)/(1-(\x)^2)});
		\draw [samples=50,domain=-0.99:0.99,rotate around={-45.:(0.,0.)},xshift=0.cm,yshift=0.cm,color=ffqqqq] plot ({1.4142135623730951*(1+(\x)^2)/(1-(\x)^2)},{1.4142135623730951*2*(\x)/(1-(\x)^2)});
		\draw [samples=50,domain=-0.99:0.99,rotate around={-45.:(0.,0.)},xshift=0.cm,yshift=0.cm,color=ffqqqq] plot ({1.4142135623730951*(-1-(\x)^2)/(1-(\x)^2)},{1.4142135623730951*(-2)*(\x)/(1-(\x)^2)});
		\draw [shift={(-0.01671499191002556,1.145258200566257)},dash pattern=on 3pt off 3pt]  plot[domain=1.0161879621170677:2.079367507570306,variable=\t]({1.*0.9882165040856301*cos(\t r)+0.*0.9882165040856301*sin(\t r)},{0.*0.9882165040856301*cos(\t r)+1.*0.9882165040856301*sin(\t r)});
		\draw [shift={(-0.00592688272573929,-1.0821296692483628)},dash pattern=on 3pt off 3pt]  plot[domain=4.226184062899762:5.213370636398994,variable=\t]({1.*1.0512203723884561*cos(\t r)+0.*1.0512203723884561*sin(\t r)},{0.*1.0512203723884561*cos(\t r)+1.*1.0512203723884561*sin(\t r)});
		\draw [shift={(-3.31243229583139,0.0022124311375819506)},dash pattern=on 3pt off 3pt]  plot[domain=-0.8413914251540398:0.8282746599520041,variable=\t]({1.*0.4479933145057295*cos(\t r)+0.*0.4479933145057295*sin(\t r)},{0.*0.4479933145057295*cos(\t r)+1.*0.4479933145057295*sin(\t r)});
		\draw [shift={(3.367437804540414,0.009823527734983907)},dash pattern=on 3pt off 3pt]  plot[domain=2.4323869520533634:3.899000872666865,variable=\t]({1.*0.49864568883574195*cos(\t r)+0.*0.49864568883574195*sin(\t r)},{0.*0.49864568883574195*cos(\t r)+1.*0.49864568883574195*sin(\t r)});
		\draw (-2.8644444444444472,0.3333333333333289) node[anchor=north west] {$\eta^\infty_1(-t)$};
		\draw (1.7355555555555528,0.3733333333333289) node[anchor=north west] {$\eta^\infty_2(-t)$};
		\draw (-0.5577777777777806,2.1466666666666647) node[anchor=north west] {$\gamma^\infty_1(t)$};
		\draw (-0.5444444444444473,-1.5066666666666737) node[anchor=north west] {$\gamma^\infty_2(t)$};
		\draw (-4.171111111111114,0.2933333333333288) node[anchor=north west] {$\infty_1$};
		\draw (3.6422222222222196,0.32) node[anchor=north west] {$\infty_2$};
		\draw (-0.3711111111111139,3.226666666666666) node[anchor=north west] {$\infty_3$};
		\draw (-0.39777777777778056,-2.813333333333342) node[anchor=north west] {$\infty_4$};
		
		\draw (0.34888888888888614,-0.16) node[anchor=north west] {$\check{\Sigma}_\infty$};
		\begin{scriptsize}
		\draw [fill=xdxdff] (-0.4979070119967356,2.008407144116613) circle (1.0pt);
		\draw [fill=xdxdff] (0.5036901052580469,1.9853477159089454) circle (1.0pt);
		\draw [fill=xdxdff] (-0.49713481112272473,-2.011526808475973) circle (1.0pt);
		\draw [fill=xdxdff] (0.49896037726362874,-2.0041671554846614) circle (1.0pt);
		\draw [fill=xdxdff] (-3.009522540213389,0.33227862115599716) circle (1.0pt);
		\draw [fill=xdxdff] (-3.0138778741516807,-0.33179844763334065) circle (1.0pt);
		\draw [fill=xdxdff] (2.989025869731999,0.33455715794445795) circle (1.0pt);
		\draw [fill=xdxdff] (3.005112319049486,-0.3327662642294511) circle (1.0pt);
		\draw [fill=qqqqff] (-0.504444444444445,0.10666666666666298) circle (1.0pt);
		\draw[color=qqqqff] (-0.4111111111111139,0.2666666666666621) node {$p$};
		\end{scriptsize}
		\end{tikzpicture}
		\caption{Case $\dhr\Sigmacheckvc$ consists of four components}\label{fig3.3.5}
	\end{figure}

Consider the plane $ \gammahat \times\Rbb $, tangent to $ \Sigmacheckvc $. Again $ \gammahat $ cannot pass through a vertex of $ \Omega $ since $ \dhr\Sigmacheckvc\setminus\Gamma $ will have at most only $ 3 $ connected components, contradicting Claim 1.
We use a vertical translation of one-parameter family of surfaces $ S_{\etahat_i,\epsilon} $ and $ S_{\gammahat_i,\epsilon} $ to prove that $ \Sigmacheckvc\cap\{-t\le z \le t\} $ is a disk bounded by $ 4 $ vertical segments and $ 4 $ horizontal curves, $ \eta^\infty_1(-t),\eta^\infty_2(-t)\subset\{z=-t\} $ connecting each one two vertical segments and $ \gamma^\infty_1(t), \gamma^\infty_2(t)\subset\{z=t\} $ connecting to two other one in such way that $ (\eta^
{\infty}_1(-t),\eta^{\infty}_2(-t) )$ has horizontal projection $ \epsilon $-close to $ \eta_1 $ and $ \eta_2 $ (same for $ \gamma^{\infty}_1(t) $ and $ \gamma^{\infty}_2(t) $). The result is a compact disk bounded by one boundary curve such that for any leaf $ \gammahat \times\Rbb $, we have $ (\gammahat \times\Rbb) \cap\dhr\Sigmacheckvc(t)$ in only two points (we use $ \gammahat $ disjoint from vertices of $ \Omega $ and curves $ \eta^\infty_1(-t),\eta^\infty_2(-t), \gamma^\infty_1(t), \gamma^\infty_2(t) $ have bounded curvature, $ \epsilon $-close to $ \eta_1, \eta_2,\gamma_1, \gamma_2 $). Hence $ \gammahat \times\Rbb $ cannot be tangent by lemma \ref{lem1.haus}-a). This conclude the Case 2.

\vskip 0.25cm
In summary of Case 1 and Case 2, the surface $\Sigmacheckvc$ is the graph of a function $u: U\subset \Omega \to \Rbb$. Letting $\epsilon \to 0$ in  Lemma \ref{equ4.2.7}-b) show that the  boundary curve of $U$ are $\eta_i$ and $\gamma_i$ for $i=1,2$ and $U=\Omega$.
\end{proof}

\begin{cor}\label{pro3.3.13}
	If $h^+(\Sigma_n)-h^-(\Sigma_n)\to+\infty$, then the two sequences $n-h^+(\Sigma_n)$ and $n+h^-(\Sigma_n)$ are bounded.
\end{cor}

\begin{proof}
	We now show that $n-h^+(\Sigma_n)$ is bounded. The bound of $n-h^-(\Sigma_n)$ will follow by the  same way.
Let's assume the contrary that there exists a subsequence of $n-h^+(\Sigma_n)$ converging to $+\vc$ as $n\to+\infty$. Without loss of generality, we still denote this subsequence by $\Sigma_n$. Denote by $\Sigmacheck_n$ the image of  $\Sigma_n$ under the translation such that $h^+(\Sigmacheck_n)=0$ for all $n$ and  $h^-(\Sigmacheck_n)\to-\vc$ as $n\to+\vc$. By Proposition \ref{pro3.3.15a}, the sequence $\Sigmacheck_n$ has a subsequence converging to a minimal graph $\Sigmacheck_\vc$ of a function $u:U\to\Rbb$ where $U$ is a subdomain of $\Omega$.
Moreover, since $n-h^+(\Sigma_n)\to+\vc$ as $n\to+\vc$, the boundary of $\Sigmacheck_\vc$ consists of the four vertical lines passing through the end-points of $\gamma_1$ and $\gamma_2$. By the Case 2 of the proof of Proposition \ref{pro3.3.15a}, we have  $U=\Omega$  and  $u=-\infty$ on $\eta_1\cup\eta_2$ and $u=+\vc$ on $\gamma_1\cup\gamma_2$. By  Theorem  \ref{thm4.2.6}, we have $\ellH{\gamma_1}+\ellH{\gamma_2}=\ellH{\eta_1}+\ellH{\eta_2}$ which is a contradiction with the geometric hypothesis on the domain $\Omega$. Therefore $n-h^+(\Sigma_n)$ is bounded.
\end{proof}

\section{Minimal annuli bounded by four vertical lines.}
\subsection{The convergence of the sequence $\Sigma_{n}$} \label{subsect3.1.6}

In this section, we first show one of the two sequences $ \Sigma^s_n $ and $ \Sigma^u_n $ has a subsequence such that $h^+(\Sigma_{n})-h^-(\Sigma_{n}) $ is bounded. Then, we next prove that such a subsequence converges to a minimal annulus with boundary four vertical lines.

\begin{prop}\label{pro3.3.14}
	There does not exist simultaneously a sequence of minimal surfaces $\Sigmas_{n}$ and a sequence of minimal surfaces $\Sigmau_n$ satisfying $h^+(\Sigmas_{n})-h^-(\Sigmas_{n})\to+\infty$ and $h^+(\Sigmau_{n})-h^-(\Sigmau_{n})\to+\infty$. If $\Sigmas_{n}$ is stable-unstable the sequence $h^+(\Sigmas_{n})-h^-(\Sigmas_{n})$ is bounded.
\end{prop}
The proof of this proposition is divided into three steps. 
It follows from Corollary \ref{pro3.3.13} that for $ \Sigma_n$ (which is either $ \Sigma^s_n $ or $ \Sigma^u_n $), if $h^+(\Sigma_{n})-h^-(\Sigma_{n})\to+\infty$ we have $n-h^+(\Sigma_{n}),h^-(\Sigma_{n})+n$ are bounded sequences.
Lemma \ref{lem3.3.18} gives us three possible limits for subsequences of translations of $\Sigma_n$. We next construct a Jacobi field on this limit, see Lemma \ref{lem3.3.18a}. Finally we finish the proof  by showing such a Jacobi field cannot exist on the limits obtained in Lemma \ref{lem3.3.18}.

\begin{lem}\label{lem3.3.18} Let $\Sigma_n$ be a sequence of minimal surface with $h^+(\Sigma_{n})-h^-(\Sigma_{n})\to+\infty$.
	\begin{enumerate}
	\item Let $\Sigmacheck_n$ be the image of $\Sigma_n$ under a vertical translation such that  $\left\{h^+(\Sigmacheck_n)\right\}_n$ converges to $+\vc$ as $n\to+\vc$ and $\left\{h^-(\Sigmacheck_n) \right\}_n$ converges to $-\vc$ as $n\to+\vc$.
	The sequence of minimal annuli $\Sigmacheck_n$ has a subsequence converging to a minimal surface $\Sigmacheck_\vc$ where $\Sigmacheckvc=(\eta_1\times\Rbb) \cup(\eta_2\times\Rbb)$.
	\item The sequence of minimal annuli $\Sigmacheck_n:=\Tsf_{-n}(\Sigma_n)$ has a subsequence converging to a minimal surface $\Sigmacheck_\vc$ where $\Sigmacheckvc$ is a vertical graph of type Scherk on $\Omega$ assuming the continuous data on $\gamma_1\cup\gamma_2$ and $-\vc$ on $\eta_1\cup\eta_2$.
	\item The sequence of minimal annuli $\Sigmacheck_n:=\Tsf_{n}(\Sigma_n)$ has a subsequence converging to a minimal surface $\Sigmacheck_\vc$ where $\Sigmacheckvc$ is a vertical graph of type Scherk on $\Omega$ assuming the continuous data on $\gamma_1\cup\gamma_2$ and $+\vc$ on $\eta_1\cup\eta_2$.
	\end{enumerate}
\end{lem}
\begin{proof}
(1) For $m,n\in\Nbb$, define $\Acal_{n,m}=\Sigmacheck_n\cap \left\{-m\le z\le m\right\}$. We have 
	\begin{equation*}
	\Acal_{n,1}\subset \Acal_{n,2}\subset\cdots\subset \Acal_{n,n}\subset \dots \subset \Acal_{n,2n}=\Sigmacheck_n.
	\end{equation*}
	Since $h^+(\Sigma_n)-h^-(\Sigma_n)\to +\vc$, by Proposition \ref{pro4.4.1},
	for each $n\gg 0$, 
	$\Acal_{n,1}$ consists of two components $\Acal_{n,1}^1$ and $\Acal_{n,1}^2$. Each $\Acal_{n,1}$ contains in its boundary a vertical part of the boundary of $\eta_i \times [-1,1]$. Using a subsequence of $\Acal_{n,1}$, we can ensure the convergence of $\{ \Acal_{n,1}^1   \}$ and $\{ \Acal_{n,1}^2  \}$.
	%
Continuing extracting subsequence from $\Acal_{n,2}, \Acal_{n,3}...$, diagonal process gives a subsequence of minimal annuli of $\left\{\Sigmacheck_{n} \right\}_n$ that converges to a minimal surface
$\Sigmacheckvc$.

By construction, $\Sigmacheckvc$ consists of two components $\Acal^1,\Acal^2$ where $\Acal^i$ is a minimal surface whose boundary is the boundary of $\eta_i\times\Rbb$.
	By Lemma \ref{lem2.1.3}-c),
	$\Acal^i=\eta_i\times\Rbb$ and we can conclude that $\Sigmacheckvc=(\eta_1\times\Rbb)\cup(\eta_2\times\Rbb)$.
	
(2) Since $\Sigmacheck_n=\Tsf_{-n}(\Sigma_n)$, the sequence $h^+(\Sigmacheck_n)$ is bounded and the sequence $h^-(\Sigmacheck_n)\to-\vc$ as $n\to+\vc$.
	By Proposition \ref{pro3.3.15a} and Lemma \ref{lem2.1.3}  , the sequence $\Sigmacheck_n$ has a subsequence converging to a minimal surface $\Sigmacheckvc$. Moreover, $\Sigmacheckvc$ is the graph of a function $u:\Omega  \to\Rbb$ where  the function $ u $ assumes the continuous value on $\gamma_1\cup\gamma_2$ and $-\infty$ on $\eta_1\cup\eta_2$.
	

%
(3) The proof is definitely similar to the one in (2).
\end{proof}

By extracting a subsequence, we can assume that  $\Sigmas_n$ is either strictly stable for all $n$ or $\Sigmas_n$ is stable-unstable for all $n$. Fix $n$, we  define a point $p_n\in\Sigmas_n$ as follows.
\begin{enumerate}
\item
In the case that $\Sigmas_n$ is stable-unstable, let $u_n$ be the eigenfunction of the Jacobi operator associate to the first eigenvalue $0$. The function $u_n=0$
on $\dhr\Sigmas_n$ and $u_n >0$ on $\Sigmas_n$. Let $p_n$ 
be a point at which $u_n$ attains its maximum value on $\Sigma_n$.

\item
In the case that $\Sigmas_n$ is strictly stable, the continuous function $\Sigmau_n\to \Rbb, q\mapsto \Distaa{\PSLhR}{q,\Sigmas_n}$ attains its maximum value at $q_n$. We call the maximum value the "distance" between $\Sigmas_n$ and $\Sigmau_n$, denoted by $d_n$.
Let $p_n\in \Sigmas_n$ such that $\Distaa{\PSLhR}{q_n,p_n}=\Distaa{\PSLhR}{q_n,\Sigmas_n}$. 
\end{enumerate}

We consider the two sequences $\left\{n-h(p_n) \right\}_n, \left\{n+h(p_n) \right\}_n$ of  real numbers. Since $\left(n-h(p_n)\right)+\left(n+h(p_n)\right)=2n\to+\vc$ as $n\to+\vc$, so after extracting  a subsequence, there are the following three possibilities:
\begin{description}
	\item[Case 1:] Both of the sequences $\left\{n-h(p_n) \right\}_n, \left\{n+h(p_n) \right\}_n$ converge to $+\vc$ as $n\to+\vc$.
	\item[Case 2: ] The sequence $\left\{n-h(p_n) \right\}_n$ is bounded while the sequence $ \left\{n+h(p_n) \right\}_n$ converges to $+\vc$ as $n\to+\vc$.
	\item[Case 3: ] The sequence $\left\{n-h(p_n) \right\}_n$ converges to $+\vc$ as $n\to+\vc$ while the sequence $ \left\{n+h(p_n) \right\}_n$  is bounded.
\end{description}

For each $n$, we denote the images of $\Sigmas_n,\Sigmau_n,\Gamma^1_n,\Gamma^2_n$ and $p_n$
under the translation $\Tsf_{h}:\PSLhR\to\PSLhR$
by respectively
$\Sigmachecks_n,\Sigmachecku_n,\Gammacheck^1_n,\Gammacheck^2_n$ and $\pcheck_n$ with $h=-h(p_n)$ in Case 1, $h=-n$ in Case 2 and $h=n$ in Case 3.
By Lemma \ref{lem3.3.18}, the two sequences $\Sigmachecks_n$ and $\Sigmachecku_n$ have a subsequence which converges to the same minimal surface $\Sigmacheckvc$, where $\Sigmacheckvc$ is $(\eta_1\times\Rbb)\cup(\eta_2\times\Rbb)$ if Case 1 occurs and is a vertical graph of type Scherk on $\Omega$ if the others occur. In any case, the sequence $h(\pcheck_n)$ is bounded. Hence, by extracting a subsequence, we can assume that the sequence $\pcheck_n$ converges to the point $\pcheckvc\in\Sigmacheckvc$. 

Since $\Sigmau_n$ and $\Sigmas_n$ has the same limit
by Lemma \ref{lem3.3.18} and since their curvature  is uniformly bounded, $\Sigmau_n$ is a local graph over $\Sigmas_n$ of a function $u_n$. By Proposition \ref{pro4.3.2},
$u_n >0$ ($\Sigmau_n$ is one side of  $\Sigmas_n$) and we have  $u_n \to 0$, uniformly as $n \to +\infty$ ($d_n \to 0$). We  prove that a subsequence of a renormalization of $u_n$  will converge to a Jacobi field on $\Sigmacheckvc$. In the case where $\Sigmas_n$ is stable-unstable, we  will consider a subsequence of a renormalization of $u_n$ where $u_n$ is the first eigenvalue of the Jacobi operator.

\begin{lem}\label{lem3.3.18a}
	There exists a Jacobi field $w_\vc$ on $\Sigmacheckvc$ satisfying
	$0\le w_\vc\le 1$ on $\Sigmacheckvc$, $w_\vc=0$ on $\dhr\Sigmacheckvc$ and $w_\vc=1$ at $\pcheckvc$.	 
\end{lem}

\begin{proof}
	In order to prove this result, we need some local computation in any $3$- Riemannian manifold, see for example \cite[Chapter 7 \S 1]{CM11}.
	{\color{black}We give in Appendix \ref{sectB} the necessary results.}
	\vskip 0.25cm
	\textbf{1.} $ C^{2,\alpha} $-convergence of $ \Sigmacheck^s_n $.
	\vskip 0.25cm	
	Fix $p\in \Sigmacheckvc$.
	Since $\Sigmachecks_n$ converges to $\Sigmacheckvc$ as $n\to+\infty$, there are $r,\epsilon>0$ such that for all $n\gg 0$, $\Sigmachecks_n\cap W_{\Sigmacheckvc}(p,r,\epsilon)$ consists of one component which is a normal graph of a function $v_n$ on $\mathcal{U}:=D_{\Sigmacheckvc}(p,r)$.
	Moreover, $\chuan{v_n}_{C^0(\mathcal{U})}$ tends to $0$ as $n\to+\infty$.
	Shrinking $r$ if necessary, there exists Fermi coordinates $(x_1,x_2,x_3)$ on $W_{\Sigmacheckvc}(p,r,\epsilon)$ such that $(x_1,x_2)$ is the coordinates on $\mathcal{U}$.
	Since the sequence of minimal surfaces $\Sigmachecks_n$ has uniform bound for the curvature, the sequence of functions $v_n$ satisfies $K:=\sup_n\chuan{v_n}_{C^1(\mathcal{U})}<+\infty$.
	
	We first prove that $v_n\to 0$ in $C^{2,\alpha}(\mathcal{U})$-topology with some $0<\alpha<1$.
	Since $\chuan{v_n}_{C^{0}(\mathcal{U})}\to 0$ as $n\to+\infty$, it suffices to prove there is $\alpha\in (0,1)$ such that for all $\mathcal{U}'\subset\subset\mathcal{U}$, we have estimate
	\begin{equation}\label{equ3.3.25}
	\chuan{v_n}_{C^{2,\alpha}(\mathcal{U}')}\le C\chuan{v_n}_{C^0(\mathcal{U})}
	\end{equation}
	with $C=C\left(K,\Distaa{}{\mathcal{U}',\dhr\mathcal{U}}\right)$.
	Indeed, since $\Gr(v_n)$ is minimal, $v_n$ satisfies the minimal surface equation (\ref{equ4.5.4}) on $\mathcal{U}$:
	\begin{align}\label{equ3.3.26}
	0
	=&\sum_{i,j=1}^{2}(g^{v_n}) ^{ij}\left(\frac{\dhr^2v_n}{\dhr x_i\dhr x_j}+\Gamma^3_{ij}+\frac{\dhr v_n}{\dhr x_i}\Gamma^3_{3j}+\frac{\dhr v_n}{\dhr x_j}\Gamma^3_{i3}+\frac{\dhr v_n}{\dhr x_i}\frac{\dhr v_n}{\dhr x_j}\Gamma^3_{33}\right)\notag\\
	&-\sum_{i,j,m=1}^{2}\frac{\dhr v_n}{\dhr x_m}(g^{v_n}) ^{ij}\left(\Gamma^m_{ij}+\frac{\dhr v_n}{\dhr x_i}\Gamma^m_{3j}+\frac{\dhr v_n}{\dhr x_j}\Gamma^m_{i3}+\frac{\dhr v_n}{\dhr x_i}\frac{\dhr v_n}{\dhr x_j}\Gamma^m_{33}\right)
	\end{align}
	where $v_n$ and its partial derivatives are calculated at $(x_1,x_2)\in\mathcal{U}$ while $(g^{v_n})^{ij},\Gamma^k_{ij}$ are calculated at $(x_1,x_2,v_n)$.
	Since $\chuan{v_n}_{C^{1}(\mathcal{U})}<K$ for all $n$, $v_n$ is a solution of a uniformly elliptic equation.
	By H\"older interior estimate \cite[Theorem 13.3]{GT01}, there exists $\alpha\in (0,1)$ such that for all $\mathcal{U}'\subset\subset\mathcal{U}$ we have estimates
	\begin{equation}\label{equ3.3.27}
	\chuan{v_n}_{C^{1,\alpha}(\mathcal{U}')}\le C_1
	\end{equation}
	with $C_1=\left(K,\Distaa{}{\mathcal{U}',\dhr\mathcal{U}}\right)$.
In an other respect, from \eqref{equ3.3.26}, (\ref{equ3.3.16}) and \eqref{equ3.3.10}, $v_n$ satisfies the following uniformly elliptic equation
	\begin{equation*}
	\sum_{i,j=1}^{2}a^{ij}_{v_n}(x_1,x_2)\frac{\dhr^2 v_n}{\dhr x_i\dhr x_j}+ \sum_{i=1}^{2}b^i_{v_n}(x_1,x_2)\dhrou{v_n}{ x_i}=f_{v_n}(x)
	\end{equation*}
	where $a^{ij}_{v_n}(x_1,x_2)=(g^{v_n}) ^{ij}(x_1,x_2,v_n)$,  $f_{v_n}(x)=-\sum_{i,j=1}^{2}g^{ij}\Gamma^{3}_{ij}$.
	It follows from \eqref{equ3.3.27} and Schauder interior estimate \cite[Theorem 6.2]{GT01} that, for all $\mathcal{U}'\subset\subset\mathcal{U}$ we have the estimate
	\begin{equation}\label{equ3.3.28}
	\chuan{v_n}_{C^{2,\alpha}(\mathcal{U}')}\le C_2 \left(\chuan{v_n}_{C^0(\mathcal{U})}+\chuan{f_{v_n}}_{C^{0,\alpha}(\mathcal{U})}\right)
	\end{equation}	
	with $C_2=C_2\left(\alpha,C_1\right)$.
	
	The estimate \eqref{equ3.3.25} is obtained from \eqref{equ3.3.28} and an estimate of 
	$\chuan{f_{v_n}}_{C^{0,\alpha}(\mathcal{U})}$ in terms of $\chuan{v_n}_{C^{0}(\mathcal{U})}$.
	Following the notations of Section 6, $H$ and $A$ are respectively the mean curvature and the second fundamental form of the level sets of the function $x_3$, given by the Fermi coordinates of $ \Sigmacheckvc $.
	By \eqref{equ3.3.17},  $f_{v_n}(x_1,x_2)=-2H(x_1,x_2,v_n)$.
    Hence, by \eqref{equ4.5.7a}
	\begin{equation}
	\td{f_{v_n}(x_1,x_2)}=\td{2H(x_1,x_2,v_n)-2H(x_1,x_2,0)}\le C_3 \chuan{v_n(x_1,x_2)}
	\end{equation}
	with $C_3=\sup_{W_{\Sigmacheckvc}(p,r,\epsilon)}\left(\chuan{A}^2+\Ric(N,N)\right)<+\infty $.
	So, $\chuan{f_{v_n}}_{C^0(\mathcal{U})}\le C_3\chuan{v_n}_{C^{0}(\mathcal{U})}$.
	Finally, we need an estimate of $$\left[f_{v_n}\right]_{\alpha;\mathcal{U}}
	=\sup_{x,y\in\mathcal{U},x\ne y}\frac{\td{f_{v_n}(x)-f_{v_n}(y)}}{\td{x-y}^\alpha}
	=2\sup_{x,y\in\mathcal{U},x\ne y}\frac{\td{H(x,v_n(x))-H(y,v_n(y))}}{\td{x-y}^\alpha}$$
	in terms of $\chuan{v_n}_{C^0(\mathcal{U})}$.
	We have
	\begin{equation}\label{equ3.3.29}
	\left[f_{v_n}\right]_{\alpha;\mathcal{U}}
	\le 2\sup_{x,y\in\mathcal{U},x\ne y}\frac{\td{H(x,v_n(x))-H(y,v_n(x))}}{\td{x-y}^\alpha}+2\sup_{x,y\in\mathcal{U},x\ne y}\frac{\td{H(y,v_n(x))-H(y,v_n(y))}}{\td{x-y}^\alpha}
	\end{equation}
	Since $\Sigmacheckvc$ is minimal, i.e., $H(x,0)=0$, so $H(x,t)=\int_{0}^{t}\dhrou{H}{x_3}(x,\tau)d\tau$.
	Thus
	\begin{align}\label{equ3.3.30}
	\td{H(x,v_n(x))-H(y,v_n(x))}&\le \int_{0}^{v_n(x)}
	\td{\dhrou{H}{x_3}(x,\tau)-\dhrou{H}{x_3}(y,\tau)}d\tau\notag\\
	&\le \chuan{H}_{C^{1,\alpha}\left(W_{\Sigmacheckvc}(p,r,\epsilon)\right)}\td{x-y}^\alpha\chuan{v_n}_{C^0(\mathcal{U})}.
	\end{align}
Since the sequence of minimal surfaces $\Sigmachecks_n$ has uniformly bounded curvature, there is a constant $C_4$ satisfying $\chuan{v_n}_{C^1(\mathcal{U})}\le C_4\chuan{v_n}_{C^0(\mathcal{U})}$ for all $n$.
	This yields
	\begin{align}\label{equ3.3.31}
	\td{H(y,v_n(x))-H(y,v_n(y))}
	&\le \chuan{H}_{C^1\left(W_{\Sigmacheckvc}(p,r,\epsilon)\right)}\td{v_n(x)-v_n(y)}\notag\\
	&\le \chuan{H}_{C^1\left(W_{\Sigmacheckvc}(p,r,\epsilon)\right)}\chuan{v_n}_{C^1(\mathcal{U})}\td{x-y}\notag\\
	&\le \chuan{H}_{C^1\left(W_{\Sigmacheckvc}(p,r,\epsilon)\right)}C_4\chuan{v_n}_{C^0(\mathcal{U})}\max\left\{1,\mathrm{diam}(\mathcal{U})^{1-\alpha} \right\}\td{x-y}^\alpha.
	\end{align}
	The estimates \eqref{equ3.3.29},\eqref{equ3.3.30} and \eqref{equ3.3.31} deduces
	an estimate of $\left[f_{v_n}\right]_{\alpha;\mathcal{U}}$ in terms of $\chuan{v_n}_{C^0(\mathcal{U})}$, hence, we get
	\eqref{equ3.3.25}.

We now denote by $g^{v_{\infty }}_{ij}$, $A^{v_\infty}_{ij}$ and $N_{v_{\infty}}$ respectively the first, second fundamental form and the unit normal vector field of $\mathcal{U}\subset\Sigmacheckvc$.
	Since $v_n\to 0$ in $C^{2,\alpha}(\mathcal{U})$-topology, from the formulas \eqref{equ3.3.9}, \eqref{equ3.3.13} and \eqref{equ3.3.12} we have
	\begin{Asse}\label{C}
 $g^{v_n}_{ij}\to g^{v_\infty}_{ij}$ in $C^{1,\alpha}(\mathcal{U})$-topology, $A^{v_n}_{ij}\to A^{v_\infty}_{ij}$ in $C^{\alpha}(\mathcal{U})$-topology and $N_{v_n}\to N_{v_\infty}$ in $C^{1,\alpha}(\mathcal{U})$-topology.
	\end{Asse}	
	
	\vskip 0.25cm	
	\textbf{2.} We now construct the Jacobi field $ w_\infty $ from $ w_n=\frac{u_n}{\max_{\Sigmachecks_n} u_n} $ the renormalization of $ u_n $.
	\vskip 0.25cm
	The smooth function $w_n$ on $\Sigmachecks_n$ satisfies $0\le w_n\le 1$ on $\Sigmachecks_n$, $w_n(p_n)=1$, $w_n=0$ on $\dhr\Sigmachecks_n$.
	Since $\Sigmachecks_n$ is a normal graph on $\Sigmacheckvc$, we can see $u_n,w_n$ as functions on $\Sigmacheckvc$, i.e.,  $u_n(x_1,x_2):=u_n((x_1,x_2,v_n(x_1,x_2)))$, $w_n(x_1,x_2):=w_n((x_1,x_2,v_n(x_1,x_2)))$.
	The first, second fundamental forms of $\Sigmachecks_n,\Sigmacheckvc$ denoted by $g^n_{ij}, g^\vc_{ij}$ and $A^n_{ij},A^\vc_{ij}$.
	We can also consider these functions on $\Sigmacheckvc$.
	Hence, we prove that the sequence $w_n$ has a subsequence converging in $C^{2,\alpha}$-topology to a  Jacobi field $\vvvc$ on $\Sigmacheckvc$.
	We consider the following two cases
	\vskip 0.25cm	
	\textbf{2a.} The case $\Sigmas_n$ is stable-unstable for each $n$.
	In this case, $u_n$ is a Jacobi field on $\Sigmachecks_n$, then so is $w_n$, i.e., $\Lrm_{\Sigmachecks_n}(w_n)=0$.
	We write the stability operator of $\Sigmachecks_m$, $m\in\Nbb\cup\{\infty\}$ under form
	\begin{gather}
	\Lrm_{\Sigmachecks_m}=\sum_{i,j=1}^{2}\overline{a}_{v_m}^{ij}(x_1,x_2)\frac{\dhr^2}{\dhr x_i\dhr x_j}+\sum_{i=1}^{2}\overline{b}_{v_m}^i(x_1,x_2)\frac{\dhr}{\dhr x_i}+ \overline{c}_{v_m}(x_1,x_2).
	\end{gather}
	Since $\Lrm_{\Sigmachecks_m}=\Delta_{\Sigmachecks_m}+\td{A^{v_m}}^2+\Ric(N_{v_m},N_{v_m})$ we have
	\begin{equation}
	\overline{a}_{v_m}^{ij}=(g^{v_m})^{ij},\quad \overline{b}_{v_m}^i=\frac{1}{\sqrt{\det(g^{v_m})}}
	\frac{\dhr}{\dhr x_i}\left((g^{v_m})^{ij}\sqrt{\det(g^{v_m})}\right)
	\end{equation}
	\begin{equation}
	\overline{c}_{v_m}=\Ric(N_{v_m},N_{v_m})+\sum_{i,j,k,\ell=1}^{2}A^{v_m}_{ij}A^{v_m}_{k\ell}(g^{v_m})^{ik}(g^{v_m})^{j\ell} 
	\end{equation}
	It follows from the assertion \ref{C} that $\overline{a}^{ij}_{v_n}\to \overline{a}^{ij}_{v_\infty}$, $\overline{b}^i_{v_n}\to \overline{b}^i_{v_\infty}$ and $\overline{c}_{v_n}\to \overline{c}_{v_\infty}$ in $C^{\alpha}(\mathcal{U})$-topology.
	
	By Schauder interior estimates \cite[Theorem 6.2]{GT01}, the sequence $\td{w_n}_{C^{2,\alpha}(\mathcal{U}')}$ is uniformly bounded for all $\mathcal{U}'\subset\subset\mathcal{U}$.
	By Arzela-Ascoli theorem together with a diagonal process, we can conclude that there exists a subsequence of $w_n$  converging uniformly to $w_\vc$ in $C^{2,\alpha}$-topology, moreover, $\Lrm_{\Sigmacheckvc}(\vvvc)=0$.
	
	\vskip 0.25cm
	\textbf{2b.} 
	The case $\Sigmas_n$ is strictly stable for all  $n$. Thus, $\Sigmachecku_n$ is a normal graph of a function $u_n$ on $\Sigmachecks_n$.
	For each $n$, the restriction of the projection $(x_1,x_2,x_3)\mapsto 
	(x_1,x_2)$ to $\Sigmas_{n}$ gives us a local coordinates $(x^n_1,x^n_2)$ of $\Sigmachecks_{n}$.
	We can extend the coordinates to Fermi coordinates $\left(x^n_1,x^n_2,x^n_3\right)$ of $\Sigmachecks_{n}$ in $\PSLhR$.

	Since $\Gr(u_n)$ is minimal, by \eqref{equ4.5.4}, $u_n$ satisfies the minimal surface equation on $\Gr(v_n)$:
	\begin{align}\label{equ3.3.32}
	0
	=&\sum_{i,j=1}^{2}(g^{u_n}) ^{ij}\left(\frac{\dhr^2u_n}{\dhr x^n_i\dhr x^n_j}+\Gamma^3_{ij}+\frac{\dhr u_n}{\dhr x^n_i}\Gamma^3_{3j}+\frac{\dhr u_n}{\dhr x^n_j}\Gamma^3_{i3}+\frac{\dhr u_n}{\dhr x^n_i}\frac{\dhr u_n}{\dhr x^n_j}\Gamma^3_{33}\right)\notag\\
	&-\sum_{i,j,m=1}^{2}\frac{\dhr u_n}{\dhr x^n_m}(g^{u_n}) ^{ij}\left(\Gamma^m_{ij}+\frac{\dhr u_n}{\dhr x^n_i}\Gamma^m_{3j}+\frac{\dhr u_n}{\dhr x^n_j}\Gamma^m_{i3}+\frac{\dhr u_n}{\dhr x^n_i}\frac{\dhr u_n}{\dhr x^n_j}\Gamma^m_{33}\right)
	\end{align}
	where $u_n$ and its partial derivatives are calculated at $(x^n_1,x^n_2)\in Gr(v_n)$ while $(g^{u_n})^{ij},\Gamma^k_{ij}$ are calculated at $(x^n_1,x^n_2,u_n)$.
	We know that $ \Sigmacheck^s_n $ converges to $ \Sigmacheckvc $ in $ C^{2,\alpha} $ topology by the first part of the proof, hence $ \chuan{v_n}_{C^{2,\alpha}(U)}\to 0 $.
	Moreover, $\chuan{u_n}_{C^0(\mathcal{U})}$ converges to $0$ as $n\to+\infty$, so by a similar argument to \ref{equ3.3.25},  using Holder estimate and Schauder estimate, there is $\gamma\in (0,\alpha)$ such that for all $\mathcal{U}'\subset\subset\mathcal{U}$, we have the estimate
	\begin{equation}\label{equ3.3.36}	\chuan{u_n}_{C^{2,\gamma}(\mathcal{U}')}\le C_5 \chuan{u_n}_{C^0(\mathcal{U})}.
	\end{equation}
	This yields that $\chuan{w_n}_{C^{2,\gamma}(\mathcal{U}')}\le C_5$
	for all $\mathcal{U}'\subset\subset\mathcal{U}$.
	It follows from Arzelà-Ascoli theorem  and a diagonal process that $w_n$ converges to $\vvvc:\Sigmacheckvc\to\Rbb$ in $C^{2,\frac{\gamma}{2}}$-topology.
	For each $n$, by \eqref{equ3.3.32} and \eqref{equ3.3.16}, $w_n$ satisfies
	\begin{align*}
	0=&\left(\sum_{i,j=1}^{2}(g^{u_n}) ^{ij}\frac{\dhr^2 w_n}{\dhr x^n_i\dhr x^n_j}-\sum_{i,j,m=1}^{2}(g^{u_n}) ^{ij}\Gamma^m_{ij}\dhrou{w_n}{x^n_m}\right)+\frac{1}{\max_{\Sigmachecks_n}(u_n)}\left(\sum_{i,j=1}^{2}(g^{u_n}) ^{ij}\Gamma^3_{ij} \right)\\
	&-\left(\sum_{i,j,m=1}^{2}\dhrou{w_n}{x^n_m}(g^{u_n}) ^{ij}\left(\dhrou{u_n}{x^n_i}\Gamma^m_{3j}+\dhrou{u_n}{x^n_j}\Gamma^m_{i3}+\dhrou{u_n}{x^n_i}\dhrou{u_n}{x^n_j}\Gamma^m_{33}\right)\right)\\
	=:& \mathrm{I}_n+\mathrm{II}_n+\mathrm{III}_n	\end{align*}
	where $u_n, w_n$ and its partial derivatives are calculated at $(x^n_1,x^n_2)\in Gr(v_n)$ while $(g^{u_n})^{ij},\Gamma^k_{ij}$ are calculated at $(x^n_1,x^n_2,u_n)$.
	Since $u_n\to 0$ in $C^{2,\gamma}(\mathcal{U})$-topology, from the formulas \eqref{equ3.3.9}, \eqref{equ3.3.13} and \eqref{equ3.3.12} we have
	\begin{Asse}
		$g^{u_n}_{ij}\to g^{\infty}_{ij}$ in $C^{1,\gamma}(\mathcal{U})$-topology, $A^{u_n}_{ij}\to A^{_\infty}_{ij}$ in $C^{\gamma}(\mathcal{U})$-topology and $N_{u_n}\to N_{_\infty}$ in $C^{1,\gamma}(\mathcal{U})$-topology.
	\end{Asse}
	Since $u_n\to 0$ as $n\to +\infty$ in $C^{2,\gamma/2}$-topology and $w_n$ is uniformly bounded in $C^{2,\gamma/2}$-topology,  so $\mathrm{III}_n\to 0$ as $n\to+\infty$. 
	We have
	\begin{equation*}
	\lim\limits_{n\to+\infty}\mathrm{I}_n=\sum_{i,j=1}^{2}(g^{\infty})^{ij}\left(\frac{\dhr^2\vvvc}{\dhr x_i\dhr x_j}-\sum_{k=1}^{2}\Gamma^k_{ij}\dhrou{\vvvc}{x_k}\right)=\sum_{i,j=1}^{2}(g^{\infty})^{ij}\vh{\nabla_{\dhr/\dhr x_i}(\nabla_{\Sigmacheckvc} \vvvc) ,\frac{\dhr}{\dhr x_j}}=\Delta_{\Sigmacheckvc} \vvvc.
	\end{equation*}
	Using \eqref{equ3.3.10}, we obtain
	\begin{equation*}
	\mathrm{II}_n=\frac{u_n}{\max_{\Sigmachecks_n}(u_n)}\left(\frac{1}{u_n}\sum_{i,j=1}^{2}(g^{v_n})^{ij}\Gamma^3_{ij} \right)-\sum_{i,j,k,\ell=1}^{2}\frac{(g^{v_n})^{ik}(g^{v_n})^{j\ell}}{W^2}\dhrou{w_n}{x^n_k}\dhrou{u_n}{x^n_\ell}\Gamma^3_{ij}. 	\end{equation*}
	Since $u_n\to 0$ as $n\to +\infty$ in $C^{2,\gamma/2}$-topology and $w_n$ is uniformly bounded in $C^{2,\gamma/2}$-topology, the second term of $\mathrm{II}_n$ converges to $0$ as $n\to+\infty$. 
	By \eqref{equ3.3.17}, 
	\begin{equation}
	\sum_{i,j=1}^{2}(g^{v_n})^{ij}\Gamma^3_{ij}=2H_n(x^n_1,x^n_2,u_n)
	\end{equation}
with $H_n(x^n_1,x^n_2,u_n)$ the mean curvature at $(x^n_1,x^n_2,u_n)$ of level set $\{x_3^n=u_n \}$. We have
$H_n(x^n_1,x^n_2,u_n)= u_n\dhrou{ H}{x_3}(x_1^n,x^n_2,\theta_n u_n)$	for $0<\theta_n<1$. Then $\frac{2}{u_n}H_n(x^n_1,x^n_2,u_n)=2\dhrou{ H}{x_3}(x_1^n,x^n_2,\theta_n u_n)=(\chuan{A}^2+\Ric(N,N))$ for the level set evaluated at $(x^n_1,x^n_2,\theta_n u_n)$ by equality \eqref{equ4.5.7a}. Since $u_n$ and $v_n$ converge to $0$ in $C^{2,\gamma}$-topology

	\begin{equation*}
	\lim\limits_{n\to+\infty}\mathrm{II}_n=
	\left\{\chuan{A}^2+\Ric(N,N)\right\}\vvvc.
	\end{equation*}
	So $w_\vc $ satisfies 
	\begin{equation}\label{equ4.3.24}
	\Lrm_{\Sigmacheckvc} w_\vc =\Delta_{\Sigmacheckvc} w_\vc +\left(\chuan{A}^2+\Ric(N,N)\right)w_\vc =0.
	\end{equation}	
	Hence, $w_\vc$ is a Jacobi field on $\Sigmacheckvc$.
	Moreover, $0\le w_\vc\le 1$ on $\Sigmacheckvc$ and $w_\vc=1$ at $\pcheckvc$.

	\textbf{3.} We have $w_\vc=0$ on $\dhr\Sigmacheckvc$.
	Moreover, the sequence  $\pcheck_n$ does not converge to $\dhr\Sigmacheckvc$ as $n\to+\vc$.

	Fix $\xi\in \dhr\Sigmacheckvc$. We will prove that  $w_\vc(\xi)=0$.
	Let $\xi_n\in\dhr\Sigmachecks_n$ satisfy $\xi_n\to\xi$.
	We have $u_n(\xi_n)=0$.
	We recall that 
	$u_n$ is here the smooth function on $\Sigmachecks_n$ satisfying the equation $\Lrm_{\Sigmachecks_n}(u_n)=0$ if $\Sigmachecks_n$ is stable-unstable and $u_n$ determines the normal minimal graph $\Sigmachecku_n$ on  $\Sigmachecks_n$ in the case that  $\Sigmachecks_n$ is strictly stable.
	
	Since the sequence $\Sigmacheck_n$ converges to $\Sigmacheckvc$, $\xi_n\to\xi$ as $n\to+\vc$, we choose $s>0$ small enough such that the sequence of the disks $D_s(\xi_n)$ in $\Sigmacheck_n$ converges to the disk $D_s(\xi)$ in $\Sigmacheckvc$. Moreover,  $\dhr D_s(\xi_n)\cap\dhr\Sigmachecks_n$ consists of only one component which is a smooth arc $T_n$ for all $n$. Since $w_n$ is a function in class $C^{2,\gamma}(D_s(\xi_n)\cup T_n)$ with $\gamma\in (0,1)$, it follows from Schauder boundary estimate, see \cite[Corollary 6.7]{GT01} that there is a constant $\bar{C}$ such that $\chuan{w_n}_{C^{2,\gamma}(D_s(\xi_n))}\le \bar{C}$.
    By theorem  Arzela-Ascoli, see \cite[Theorem 1.34]{AF03},
	the sequence $w_n$ has a subsequence converging  in $C^{2,\gamma/2}$ topology to the function $w_\vc$.
	
	Since $w_n(\xi_n)=0$, so $w_\vc(\xi)=0$. Moreover, the sequence $\pcheck_n$ does not converge to $\dhr\Sigmacheckvc$ as $n\to+\vc$.
	
\end{proof}

\begin{proof}[{\rm \bfseries Proof of Proposition \ref{pro3.3.14}}]
	By the previous lemmas, there exists a minimal surface $\Sigmacheckvc$ of $\PSLhR$ together with a Jacobi field $\vvvc$ on $\Sigmacheckvc$ where
	\begin{itemize}
	\item $\Sigmacheckvc$ is $(\eta_1\times\Rbb)\cup(\eta_2\times\Rbb)$ or a minimal graph of type Scherk of a function defined on $\Omega$ which assumes continuous data on $\gamma_1\cup\gamma_2$ and $\pm\vc$ on $\eta_1\cup\eta_2$.
	\item $0\le \vvvc\le 1$, $\vvvc=0$ on $\dhr\Sigmacheckvc$ and $\vvvc\ne 0$.
	\end{itemize}
	To conclude, we will show that there does \emph{not} exist such a Jacobi field $\vvvc$ on $\Sigmacheckvc$.

\vskip 0.25cm	
	\textbf{1.} We consider the case $\Sigmacheckvc=(\eta_1\times\Rbb)\cup(\eta_2\times\Rbb)$. Since $\vvvc$ is a Jacobi field on $\Sigmacheckvc$, it satisfies the equation:
	\begin{equation*}
	\Delta\vvvc+q\vvvc=0\quad \text{where } q=\td{A}^2+\Ric(N,N).
	\end{equation*}
We denote here by $\Delta,A,N$ resp. for the  Laplace operator, the second fundamental form and a unit normal vector field of the surface $\Sigmacheckvc$.
	We calculate the function $q$. Let $(T,E_3,N)$ be an orthonormal basis of $\PSLhR$. The Ricci curvature $\Ric(N,N)$ is
	\begin{equation*}
	\Ric(N,N)=K(N,T)+K(N,E_3).
	\end{equation*}
By Section 2.1,  the sectional curvature of a $2$-dimensional vector space $P$ of $T\PSLhR$ is $K(P)=(1+4\tau^2)a-(1+3\tau^2)$ where $a$ is square of the length of the orthogonal projection of the vector $E_3$ on $P$. Hence,
	\begin{equation*}
	\Ric(N,N)=-(1+3\tau^2)+\tau^2=-1-2\tau^2.
	\end{equation*}
In order to compute$\td{A}^2$, without loss of generality, we consider 
$\gammahat_1\times\Rbb=\{x=0\}$ in the cylinder model with
$\{E_2,E_3\}$ is the orthonormal basis of its tangent plane  $P \subset T\PSLhR$. Using formula (\ref{equ1.1.2as}), the second fundamental form in this basis is given by
$\begin{pmatrix}
	0& \tau \\ 
	\tau & 0
	\end{pmatrix}$. 
Hence, $\td{A}^2=2\tau^2$ and it follows that $q=2\tau^2-\left(1+2\tau^2\right)=-1$. Since $\vvvc\ge 0$ is a positive Jacobi field we have,  
$\Delta \vvvc\ge \Delta\vvvc +q\vvvc=0$.
The function $\vvvc$ attains its maximum $1$ in the interior of $\Sigmacheckvc$ and by the maximum principle
$\vvvc$ is constant $1$ in the interior of $\Sigmacheckvc$.
This contradicts the fact that $\vvvc=0$ on the boundary of $\Sigmacheckvc$.
	
	\vskip 0.25cm
	\textbf{2.} 
	We consider the case where $\Sigmacheckvc$ is a graph of type Scherk of a function $u$ defined on $\Omega$  with boundary data $u=-\vc$ on $\eta_1\cup\eta_2$ and $u=f_i$ on $\gamma_i$ with $f_i$ continuous (the case  where $u=+\vc$ follows from the same arguments).
Since $E_3=\partial_z$ is a Killing vector field, $N_3:=\vh{E_3,N}$ is a Jacobi field on $\Sigmacheckvc$. Moreover, $\vvvc$ is also a Jacobi field on $\Sigmacheckvc$, thus
$X:=\vvvc\nabla N_3-N_3\nabla \vvvc$ is a tangent vector field of $\Sigmacheckvc$ whose divergence vanishes.
	
	For each $t> 0$,  $\Sigmacheckvc^+(-t)=\Sigmacheckvc\cap \left\{z\ge -t \right\}$ is the graph of a function $u$ on a domain $\Omega_t \subset\Omega$ whose boundary is made up of the two geodesics $\gamma_1,\gamma_2$ and two arcs $\eta_1(-t),\eta_2(-t)$ together with their endpoints. Moreover $u$ assumes the value $-t$ on $\eta_1(-t)\cup\eta_2(-t)$.
Using that $\vvvc$ and $N_3$ vanish on the vertical part of the boundary of $\Sigmacheckvc$ and $w_\infty=0$ on $\partial \Sigmacheckvc$ , the divergence theorem gives us
	\begin{equation}\label{equ3.3.22}
	0=\int_{\Sigmacheckvc^+ (-t)}\Div X=\int_{I_1(-t)\cup I_2(-t)}\Big(\vvvc\vh{\nabla N_3,\nu}-N_3\vh{\nabla \vvvc,\nu}\Big)- \int_{J_1\cup J_2}
	N_3\vh{\nabla\vvvc,\nu}
	\end{equation}
where $\nu$ is the outer conormal vector of $\dhr\Sigmacheckvc^+(-t)$, $J_i$ is the graph of the function $u$ over $\gamma_i$ and $I_i(-t)$ is the graph of $u$ on $\eta_i(-t)$ for $i=1,2$.
Since on $J_1\cup J_2$, we have $N_3>0$ and $\vh{\nabla\vvvc,\nu}<0$ we have
that $\displaystyle\int_{J_1\cup J_2}  N_3\vh{\nabla\vvvc,\nu}$ is a negative constant. Since $\vvvc,\td{\nabla\vvvc}$ are bounded on $\Sigmacheckvc$ while the values  $N_3$ and $\vh{\nabla N_3,\nu}$  converge to $0$  on $I_i(-t)$ as $t\to +\vc$ (since $\Sigmacheckvc$ is asymptotically flat at infinity by Lemma \ref{lem2.1.3}-b)), we have that $\displaystyle\int_{I_1(-t) \cup I_2 (-t)}\Big(\vvvc\vh{\nabla N_3,\nu}-N_3\vh{\nabla \vvvc,\nu}\Big)\to 0$ as $t\to+\vc$. This contradicts  \eqref{equ3.3.22}.	
	
	\end{proof}

It follows from proposition \ref{pro3.3.14} that if $\Sigmas_n$ is a sequence of stable-unstable minimal surface then $h^+(\Sigmas_{n})-h^-(\Sigmas_{n})$ is bounded.
In the case where $\Sigmas_n$ is strictly stable and $h^+(\Sigmas_{n})-h^-(\Sigmas_{n})$
diverges, then $\Sigmau_n$ is a sequence with $h^+(\Sigmau_{n})-h^-(\Sigmau_{n})$ bounded. We describe now the limit of a subsequence of $\Sigma_n$ with $h^+(\Sigma_{n})-h^-(\Sigma_{n})$ bounded.

\begin{prop}\label{pro3.3.17}
If $h^+(\Sigma_{n})-h^-(\Sigma_{n})$ is bounded, the two sequences 
$n-h^+(\Sigma_{n})$ and $n+h^-(\Sigma_{n})$ converge to $+\infty$ as $n\to+\infty$.
\end{prop}

\begin{proof}
	Without loss of generality, we assume that  $n-h^+(\Sigma_n)$ is bounded.
	For each $n$, denote  $\Sigmacheck_n=\Tsf_{-n}(\Sigma_n)$.
	Since the sequence of minimal surfaces $\Sigmacheck_n$ has uniform bound for the curvature and an accumulation point, by Lemmas \ref{lem3.3.14}, \ref{lem3.3.15}, there exists a subsequence (we still denote by $\Sigmacheck_n$) converging to a minimal surface $\Sigmacheckvc$, with multiplicity $ 1 $, as in the proof of Proposition \ref{pro3.3.15b}.
	Moreover, assume that the sequence  $h^+(\Sigmacheck_n)$ converges to $h^+(\Sigmacheckvc)\in\Rbb$ and $h^-(\Sigmacheck_n)$ converges to $h^-(\Sigmacheckvc)\in\Rbb$ as $n\to+\vc$.
	The boundary of $\Sigmacheckvc$ consists of the two components contained in the two vertical planes $\gamma_1\times\Rbb$ and $\gamma_2\times\Rbb$.
	The component $\Gammacheckvc^i\subset\gamma_i\times\Rbb$ is a smooth curve consisting of an arc $J_i$ which is a graph on $\gamma_i$ together with two straight half-lines
	
For each $t> 0$, the boundary of $\Sigmacheckvc^+ (-t)=\Sigmacheckvc \cap \{ z \geq -t \}$ consists of
$\left(\dhr\Sigmacheckvc  \cap \{ z\geq -t \}\right)$ 
and the two arcs $L_1(-t), L_2(-t)$ contained in $\{ z=-t\}$ where $\pi (L_i (-t) )\subset\Hbbh$ has the same endpoints as $\gamma_i$ by Lemma  \ref{pro4.4.1}-(2). Using Lemma \ref{lem2.1.3}-b), the part $\Sigmacheckvc \cap \{ z \leq -t\}$ has two connected components, bounded by four vertical halflines and $L_1(-t),L_2(-t)$. These components are uniformly asymptotic to $\gamma_1 \times \Rbb$ and $\gamma_2 \times \Rbb$ at infinity. This prove that $\pi(L_i(-t))$ is converging to $\gamma_i$ when $t \to +\vc$.
Since $E_3$ is a Killing vector field, the divergence theorem gives us
\begin{equation}\label{divequation}
0=\int_{\Sigmacheckvc^+ (-t)}\Div E_3=\int_{\dhr\Sigmacheckvc ^+(-t)} \vh{E_3,\nu}
=\sum_{i=1}^{2}\left(\int_{J_i}\vh{E_3,\nu}+\int_{L_i  (-t)}\vh{E_3,\nu}\right)
\end{equation}
where $\nu$ is the outer conormal vector of $\dhr\Sigmacheckvc^+ (-t)$, remark that $\vh{E_3,\nu}=0$ on the vertical part of the boundary of $\Sigmacheckvc^+ (t)$.
By the boundary maximum principle, $\Sigmacheckvc$ is not tangent to $\gamma_i\times\Rbb$ on $J_i$, hence, {for }$ i=1,2 $, $ \int_{J_i}\vh{E_3,\nu}<\ellH{\gamma_i} $. In an other respect, $\vh{E_3,\nu} \to -1$ on $L_1(-t) \cup L_2(-t)$ (the surface is asymptotic to $\gamma_i \times \Rbb$) when $t \to +\vc$ and $\int_{L_i  (-t)}\vh{E_3,\nu} \to -\ellH{\gamma_i}$ for $t \to +\vc$ giving a contradiction with (\ref{divequation}). A similar proof for $n+h^-(\Sigma_n)$ yields the conclusion of the proposition.

\end{proof}

\begin{proof}[\rm \bfseries Proof of Theorem \ref{thm3.3.1}]
From Proposition \ref{pro3.3.14}, there exists a sequence
of surfaces $\Sigma_n=\Sigmas_n$ or $\Sigma_n=\Sigmau_n$ with $h^+(\Sigma_{n})-h^-(\Sigma_{n})$ bounded. 
We recall that $\Sigma_n$ is tangent to $\left\{z=h^+(\Sigma_n)\right\}$ at the point $p^+(\Sigma_n)$. Denote by $\Sigmacheck_n$ and $\pcheck_n$  the image of $\Sigma_n$ and $p^+(\Sigma_n)$ under the vertical translation such that $h^+(\Sigmacheck_n)=0$. By extracting a subsequence, we can assume that the sequence  $\Sigmacheck_n$ passing through $\pcheck_n$ converges to  $\Sigmacheckvc$ passing through $\pcheckvc \in \Omega \times \{0\}$. 
By Proposition \ref{pro3.3.17}, $n-h^+(\Sigma_{n})$, $n+h^-(\Sigma_{n})$ is converging to $+\infty$, the surface $\Sigmacheckvc$ is bounded by the four vertical line passing through the vertex of $\Omega$.
Using arguments of Proposition \ref{pro3.3.15b}, the sequence $\Sigmacheck_n$ converges to  $\Sigmacheckvc$ with multiplicity $1$.

By Lemma  \ref{lem3.3.14}, for $M>0$ large enough
$\Sigmacheck_n\setminus \left(\Omega \times[-M,M]\right)$ is between $\left(\gamma_1\times\Rbb\right)\cup \left(\gamma_2\times\Rbb\right)$  and the graph of the function $u^{\pm}$ which take value $\pm \vc$ on $\gamma_i$. This prove with Lemma \ref{lem2.1.3}-b) that outside a compact set $\Sigmacheckvc$ is a normal graph on $\gamma_i \times \Rbb$ of a function converging uniformly to zero at infinity.

To prove that $\Sigmacheckvc$ is  not simply connected,   we exhibit a curve $ \alpha_\infty \subset\Sigmacheckvc$ homotopically non trivial.
For $n_0>M$ large enough, we consider the rectangle $R=\Gamma^1_
{n_0} \subset \gamma_1 \times \Rbb$. The rectangle $R$ is homotopic to $\Gamma^1_n$ for any $n >n_0$. We denote by $ \alpha_n, \alpha_\infty $ the horizontal graph on $R$ in $ \Sigmacheck_n, \Sigmacheckvc$ respectively.

If $\alpha_\infty$ bounds a disk $D_\infty$ in $\Sigmacheckvc$, 
then there is a sequence of disks $D_n\subset\Sigmacheck_n$ whose boundaries are $\alpha_n$ and converging to $D_\infty$. For $n$ large enough, the homotopy from $R$ to $ \Gamma^1_n$ lifts to a homotopy in $ \Sigmacheck_n\setminus(\Omega\times[-M,M]) $ from $ \alpha_n $ to $ \Gamma^1_n $, contradicts the existence of the disk $ D_n$.


	In order to prove $\Sigmacheckvc$ is an annulus, it suffices to verify that any pair of smooth Jordan curves $\alpha_\infty,\beta_\infty$ non intersecting and homotopically nontrivial on $\Sigmacheckvc$ will be the boundary of an annulus in $\Sigmacheckvc$.
	There is a convex compact set $ K $ and two sequences of smooth Jordan curves $\alpha_n,\beta_n\subset \Sigmacheck_n\cap K$ with $\alpha_n\to \alpha_\infty$ and $\beta_n\to\beta_\infty$, $\alpha_n\cap \beta_n=\emptyset$ for all $n\gg 0$.
	Moreover, for $n\gg 0$, $\alpha_n,\beta_n$ are homotopically nontrivial. If not there exists a sequence of disks $D_{n}\subset\Sigmacheck_n$ whose boundary is $\alpha_n$ and $D_n$ is contained in the convex hull of 
	$\alpha_n$, so $D_n\subset K$ for $n\gg 0$.
	 Then the sequence of disks $D_{k_n}$ has a subsequence converging to a disk in $\Sigmacheckvc$ whose boundary is $\alpha_\infty$, a contradiction.
Since $\Sigmacheck_n$ is a minimal annulus, the pair of non-intersecting, homotopically nontrivial Jordan curves $\alpha_n,\beta_n$ (for $n\gg 0$) delimits a compact sub-annulus $\Acal_n$ of $\Sigmacheck_n\cap K$. Since $\alpha_n ,\beta_n$ are properly contained in a convex compact set of $\PSLhR$, there exists a subsequence of $\Acal_n$ converging to an annulus on $\Sigmacheckvc$ with boundary $\alpha_\infty \cup\beta_\infty$.
This completes the proof of Theorem.
\end{proof}


\section{Minimal annulus on unbounded domain and Riemann's type examples}
\label{sect3.4}
In this section we prove Theorem \ref{thm1}.
Let $\gammahat_1$, $\gammahat_2$ be two ultraparallel complete geodesics of $\Hbbh$.
Recall that $\Omegahat \subset \Hbbh$ is an ideal quadrilateral domain whose ideal boundary 
$\partial \Omegahat$ is composed of four complete geodesics $\gammahat_1,\etahat_1, \gammahat_2,\etahat_2$ in this order together with distincts endpoints $\widehat{p}_1,\widehat{q}_1,\widehat{p}_2,\widehat{q}_2$  at infinity. 
The condition $\dist_{\Hbbh}(\gammahat_1, \gammahat_2) < 2 \ln (\sqrt{2} +1)$
is equivalent to $\dist_{\Hbbh}(\gammahat_1, \gammahat_2) < \dist_{\Hbbh}(\etahat_1, \etahat_2)$. 
We consider a sequence of bounded quadrilateral subdomains $\Omega_n \subset \Omegahat$ whose boundary is composed  of four open geodesic arcs $\gamma^n_1, \eta^n_1, \gamma^n_2 ,\eta^n_2$ in this order together with their endpoints $p^n_1,q^n_1,p^n_2,q^n_2$.

Let $\gammahat_{12}$ (resp. $\etahat_{12}$) the unique complete geodesic of $\Hbbh$  perpendicular to $\gammahat_1$ and $\gammahat_2$ (resp. $\etahat_1$ and $\etahat_2$).
The two geodesics $\gammahat_{12}$ and $\etahat_{12}$ are perpendicular at $O\in\Hbbh$, see Figure \ref{fig5.1}.
Then $2\Distaa{\Hbbh}{O,\gammahat_1}=\Distaa{\Hbbh}{\gammahat_1,\gammahat_2}$ and $2\Distaa{\Hbbh}{O,\etahat_1}=\Distaa{\Hbbh}{\etahat_1,\etahat_2}$.
Since $\Distaa{\Hbbh}{\gammahat_1,\gammahat_2}<\Distaa{\Hbbh}{\etahat_1,\etahat_2}$,  $\Distaa{\Hbbh}{O,\gammahat_1}<\Distaa{\Hbbh}{O,\etahat_1}$.
Let $\eta'_1,\eta'_2$ be two distinct complete geodesics of $\Hbbh$ perpendicular to $\eta_{12}$ and having the same distance $\Distaa{\Hbbh}{O,\gammahat_1}$  to $O$.
Let $\Omega_0$ be the compact quadrilateral sub-domain of $\Omega$ delimited by $\gamma'_1,\gamma'_2$ and $\eta'_1,\eta'_2$ and vertices $p_1',q_1',p_2',q_2'$. By construction, $\Omega_0$ is a regular hyperbolic quadrilateral  which satisfy

$$\ellH{\gamma'_1}+\ellH{\gamma'_2}=\ellH{\eta'_1}+\ellH{\eta'_2}.$$
\noindent
\begin{lem}
\label{lem5.1}
Any quadrilateral domain $\Omega$ bounded by compact geodesic $\gamma_i \subset  \gammahat_i$ and compact  geodesic $\eta_1,  \eta_2$, such that $\Omega_0$ is strictly contained in $\Omega$ satisfies
$$\ellH{\gamma_1}+\ellH{\gamma_2}>\ellH{\eta_1}+\ellH{\eta_2}.$$
\end{lem}
\begin{proof} 
Consider a quadrilateral domain
$\Omega_1$ with vertices $p_1,\bar{q}_1,\bar{p}_2,\bar{q}_2$
which strictly contains a quadrilateral domain $\bar{\Omega}_1$ with vertices $\bar{p}_1,\bar{q}_1,\bar{p}_2,\bar{q}_2$ satisfying
$$\dist(\bar{p}_1,\bar{q}_2) +\dist(\bar{q}_1,\bar{p}_2) \leq \dist(\bar{p}_1, \bar{q}_1) +\dist(\bar{p}_2, \bar{q}_2) .$$
By triangle inequality and this inequality we have
$$\dist({p}_1,\bar{q}_2)+\dist(\bar{q}_1,\bar{p}_2) < \dist({p}_1, \bar{p}_1) + \dist( \bar{p}_1,\bar{q}_2)+  \dist(\bar{q}_1,\bar{p}_2) \leq \dist(p_1, \bar{q}_1)+    \dist(\bar{p}_2, \bar{q}_2).$$
This imply that $\Omega_1$  satisfy the conclusion of the lemma. One can apply this argument recursively to each vertex of $\Omega$ to prove the lemma.

\end{proof}
We  consider in the following  a sequence of quadrilateral domain $\{\Omega_n\}_{n \in \Nbb}$ with


\begin{enumerate}
\item $\Omega_0 \subset \Omega_n \subset \Omega_{n+1} \subset \Omegahat$ ,
\item for $i=1,2$, $\gamma^n_i\subset\gamma^{n+1}_i\subset\gamma_i$ for all $n$, 
\item $\displaystyle{\cup_{n \in \Nbb} \Omega_n=\Omega_{\vc}\subset  \Omegahat}$.
\end{enumerate}
Hence by Lemma \ref{lem5.1}, $\Omega_n$ satisfy
$\ellH{\gamma_1^n}+\ellH{\gamma_2^n}>\ellH{\eta_1^n}+\ellH{\eta_2^n}.$

\begin{figure}[h!]
	\centering
	\includegraphics[width=0.25\linewidth]{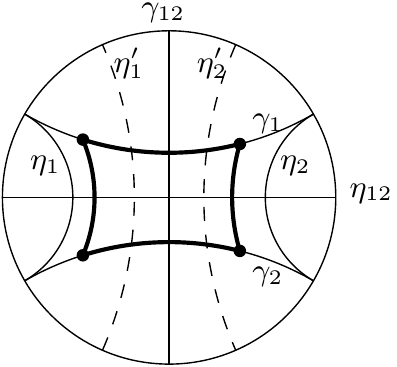}
	\caption{The quadrilateral domain $\Omega_n$}
	\label{fig5.1}
\end{figure}

\noindent

For $i=1,2$, by  Theorem \ref{thm4.2.6}  there are  minimal solutions $u^\pm_{n,i}$ on $\Omega_n$ such that $u^\pm_{n,i}=\pm\infty$ on $\gamma^n_i$ and $u^\pm_{n,i}=0$ on  $\eta^n_1\cup \eta^n_2\cup \gamma^n_j$ with $\{i,j\}=\{1,2\}$.
Define $u^+_n=\sup\{u^+_{n,1},u^+_{n,2}\}$ and $u^-_n=\inf\{u^-_{n,1},u^-_{n,2} \}$.

\begin{prop}\label{cor3.4.3}
	For each $n$ there is a proper minimal annulus $\Acal_n$ in $\Omega_n \times \Rbb$ whose boundary is the four vertical geodesics passing through the endpoints $p^n_1,q^n_1,p^n_2,q^n_2$ of  $\gamma_1^n$ and $\gamma_2^n$ such that for each complete geodesic  $\gammahat$ that meets $\eta_1^n$ and $\eta_2^n$, the intersection of this annulus $\Acal_n$ with $\gammahat\times\Rbb$ is compact. Moreover, $\Acal_n$ is below {{{}}} the graph of  the function $u^+_n+h^+(\Acal_n)$ and above the graph of the function $u^-_n+h^-(\Acal_n)$ where $h^+(\Acal_n)$ and $h^-(\Acal_n)$ denote the height where the tangent
planes are horizontal at $p_n^+$ and $p_n^-$.
\end{prop}

\begin{proof} 
The existence of  $\Acal_n$ is given by Theorem \ref{thm3.3.1}, for the bounded quadrilateral domains $\Omega_n$. By construction each $\Acal_n$ is the limit of a sequence of $\{\Sigma_{n,m}\}
_{m \in \Nbb}$ that satisfy Lemma 3.8 with $h^+(\Sigma_{n,m})$ and $h^-(\Sigma_{n,m})$ converging to $h^+(\Acal_{n})$ and $h^-(\Acal_n)$. Hence
$u^+_n+h^+(\Acal_{n})$ (and $u^-_n+h^-(\Acal_{n})$) is a barrier for the limit $\Acal_n$.

\end{proof}

 Each $\Acal_n$ is the limit with multiplicity $1$ of the sequence of compact minimal annuli $\left\{\Sigma_{n,m}\right\}_{m\in \Nbb}$ constructed on the bounded convex quadrilateral domain $\Omega_n$ whose boundary is $\Gamma^1_{n,m}\cup\Gamma^2_{n,m}$, where $\Gamma^i_{n,m}$ is the boundary of a smooth convex domain obtained after smoothing the four corners of the convex domain $ \gamma^n_i \times[-m,m]$.

We recall that $\Fcal^h$ is the foliation of $\Omegahat \times \Rbb$ by horizontal leaves, 
$\Fcal^{\gammahat_{12}}$ (resp. $\Fcal^{\etahat_{12}}$) is the foliation by vertical planes with each leaf orthogonal to the geodesic $\gammahat_{12}$ (resp. $\etahat_{12}$) and $\Fcal^{\etahat_1}$ (resp. $\Fcal^{\etahat_2}$) is the foliation by vertical planes with each leaf orthogonal to the geodesic $\etahat_1$ (resp. $\etahat_2$).



\begin{prop}
\label{foliationcomplete}
Each annulus $\Acal_n$  is tangent to the foliations  $\Fcal^h$, $\Fcal^{\etahat_1}$, $\Fcal^{\etahat_2}$, $\Fcal^{\etahat_{12}}$ and $\Fcal^{\gammahat_{12}}$  of $\PSLhR$ at most at two points.
\end{prop}

\begin{proof}
The foliation $\Fcal^h$, $\Fcal^{\etahat_{12}}$ satisfy the hypothesis of Proposition \ref{pro3.1}, hence they meet tangentially  each $\Sigma_{n,m}$ at most at two points.
Following notations of Proposition \ref{pro3.1}, leaves of $\Fcal^{\etahat_1}$, $\Fcal^{\etahat_2}$ satisfy that for $0<t<a$ and $b<t<1$ intersects only one of the boundary curves $\Gamma^1_n$ and $\Gamma^2_n$.
By Lemma \ref{lem1.haus}-c), the foliation can have only one tangent point for $0<t\leq a$ and one other tangent point for $b\leq t<1$.  The leaves for $a<t<b$ don't intersect the boundary curves. By Lemma  \ref{lem1.haus}-b) there is no tangency.  In summary $\Fcal^{\etahat_1}$ and $\Fcal^{\etahat_2}$ has at most two tangent points with $\Sigma_{n,m}$.

By Lemma \ref{lem2.3.18}, since $\Acal_n$ is the limit of a subsequence of $\Sigma_{n,m}$, each tangent point $q$ of $\Acal_n$ with a leaf of a foliation $\Fcal$, will produce in the neighborhood of $q$ a tangency between $\Sigma_{n,m}$ and $\Fcal$ for $m>0$ large enough. This prove the proposition.
\end{proof}

By arguments of Lemma \ref{pro4.4.1}  we have the geometric structure of $\Acal_n$:

\begin{prop}\label{pro3.4.4}
	We have the following properties 
	\begin{enumerate}
	\item $\Acal_n$ has one or two horizontal points.
	\item If $t>h^+(\Acal_n)$ (resp. $t<h^-(\Acal_n)$) then $\Acal_n^+(t)=\Acal_n \cap \{ z \geq t \}$ (resp. $\Acal_n^-(t)$) consists of two simply connected components.
	Then, $\Gamma (t)$ consists of two components diffeomorphic to $[0,1]$ and joining two straight lines that pass through the endpoints of $\gamma^n_i$ for $i=1,2$.
	\item If $h^+(\Acal_n)-h^-(\Acal_n)>\sqrt{1+4\tau^2}\pi$, then
	\begin{enumerate}
	\item[\rm (a)] For each $t\in (h^-(\Acal_n),h^+(\Acal_n))$, $\Acal^+_n(t)$ and $\Acal^-_n(t)$ are simply connected. Moreover, $\Gamma(t)$ consists of two components diffeomorphic to $[0,1]$, joining two straight lines that pass through the endpoints of $\eta^n_i$ for $i=1,2$.
	
	\item[\rm (b)] The set $\Acal_n\cap \left\{h^-(\Acal_n)<z<h^+(\Acal_n) \right\}$ consists of two simply connected components.
	\end{enumerate}
	\end{enumerate}
\end{prop}

\begin{proof}
The assertion (1) is a direct application of Proposition \ref{foliationcomplete}  with the foliation $\Fcalh$. By passing at the limit with $\Sigma_{n,m}$, the limit surfaces $\Acal_{n}$ satisfy with the use of Proposition \ref{cor3.4.3} all the conclusions of Lemma \ref{pro4.4.1}.

\end{proof}
Now we prove uniform curvature bound for the sequence $\Acal_n$:
\begin{prop}\label{pro3.4.6}
If the geodesics $\gammahat_1,\gammahat_2$ are ultraparallel, the sequence of minimal annuli $\Acal_n$ has a curvature uniformly bounded, i.e., we have
	\begin{equation}
	\sup_n\sup_{\Acal_n}\td{A_{\Acal_n}}<+\infty.
	\end{equation}
\end{prop}

\begin{proof}
The proof is similar to that of Theorem \ref{pro4.4.4}. Since $\Acal_n$ is the limit of a sequence of compact annuli $\Sigma_{n,m}$ with uniform bounded curvature independent of $m$ by Theorem \ref{pro4.4.4}, the curvature is uniformly bounded on each $\Acal_n$.
Since $\Acal_n$ is asymptotically flat, the maximum value of the curvature is attained at some point $p_n$.

Assume now, that the curvature at $p_n$ diverges. After an isometry 
$\Tsf_{p_n}$ which send the point $p_n$ at $O$, and using a dilatation, we construct a complete embedded minimal surface $\widetilde{\Acal}_\infty$  similar to the one of the Proposition \ref{lem3.3.12} and satisfying conclusions (1) to (4). There remains to prove that $\widetilde{\Acal}_\infty$ is a finite total curvature surface without boundary to satisfy conclusion (5).

We consider foliations $\Fcal^h$, $\Fcal^{\etahat_{12}}$ and $\Fcal^{\gammahat_{12}}$. If the horizontal projection $q_n=\pi(p_n)$ remains in a compact set of $\Omegahat$, after the blow up, we recover $\widetilde{\Acal}_\infty$ and three family of parallel planes induced by 
the blow up of the three foliations. These parallel planes are transverse
because  $q_n$ is away from the ideal vertices of $\Omegahat$. 
As in the proof of Proposition \ref{lem3.3.12}, we use
 the Propostion \ref{foliationcomplete} to prove that the Gauss map take five values a finite number of times and
with the Mo-Osserman theorem we prove that
$\widetilde{\Acal}_\infty$ has finite total curvature.
In the case where $q_n$ converge to an ideal point of $\Omegahat$ at infinity, the foliations $\Fcal^{\etahat_{12}}$ and $\Fcal^{\gammahat_{12}}$ converge to the same foliation of $\Rbb^3$ (the  leaves are almost tangential at $p_n$). Hence we consider instead of $\Fcal^{\gammahat_{12}}$, the foliation $\Fcal^{\etahat_1}$ if $q_n$ converge to an endpoint of $\etahat_1$ or  $\Fcal^{\etahat_2}$ if  $q_n$ converge to an endpoint of $\etahat_2$. As in the preceding case, we recover with $\widetilde{\Acal}_\infty$, three family of parallel planes and we conclude in the same way the finite total curvature of the limit. 

There remains to prove that $\widetilde{\Acal}_\infty$ has no boundary. Assume the contrary, $\partial \widetilde{\Acal}_\infty$ is a vertical straight line $L$.
There is a subsequence of $q_n$ which converges to $q_{\vc}$ and stay at a finite distance of a vertex of $\Omega_n$, suppose $p_1^n$ without loss of generality.
We observe that $L$ is the limit of  
$\Tsf_{p_n}( \{p_1^n\} \times \Rbb)$.
For any $n$, there is a leaf $\Lambda_n$ in $\Fcal^{\etahat_2}$  such that $\{  p^n_1 \} \times \Rbb=\Lambda_n \cap \{ \gammahat_1 \times \Rbb \}$ and $\Acal_n$ is contained in a wedge bounded by $\Lambda_n \cup \{ \gammahat_1 \times \Rbb \}$. Hence, the surface $\widetilde{\Acal}_\infty$ is contained in a wedge bounded by the  limit of $\Tsf_{p_n} (\Lambda_n)$ and  $\Tsf_{p_n}(\{ \gammahat_1 \times \Rbb \})$.  Since angle between $\Lambda_n$ and $\gammahat_1 \times \Rbb$ is less than $\pi$, there is no finite total curvature end contained in the wedge and this give a contradiction exactly as at the end of the proof of Proposition \ref{lem3.3.12}. This prove that $\widetilde{\Acal}_\infty$ satisfy all properties of Proposition \ref{lem3.3.12}.

To finish the proof of the estimate, we study the geometry of $\widetilde{\Acal}_\infty$. Finite total curvature imply that there is a non trivial curve $\mutilde$ in $\widetilde{\Acal}_\infty$, which is the limit of $\mutilde_n$ a sequence of non trivial curve in $\widetilde{\Acal}_n$ contains in a neighborhood of $p_n$ (see the end of the proof of  Theorem \ref{pro4.4.4}). When $q_n$ remains in a compact part of $\Omegahat$, we prove as in Theorem \ref{pro4.4.4}, that $\mutilde$ bound a finite total curvature end contains in an acute wedge of $\Rbb^3$. The main point is that $q_n$ is at a distance strictly positive of
$\gamma_1$ or $\gamma_2$.

In the case, where $q_n$ converges to an ideal vertex of $\Omegahat$ say $\widehat{p}_1$ without loss of generality.
The curve $\mutilde_n$ with two straight lines in  $\gammahat_2 \times \Rbb$ bound an annulus.
This annulus is contained into a wedge bounded by vertical planes joining $\mutilde_n$ and the end points of 
$\gammahat_2$.  At the limit when $n \to +\vc$, the annulus converges to a non flat ends in $\Rbb^3$ which is contained in a slab. This slab comes from the limit of  ${\Tsf_{p_n}}(\gammahat_3 \times \Rbb)$ where $\gammahat_3$ has a contact point with $\mutilde_n$ and is close to a diagonal of $\Omegahat$  and ${\Tsf_{p_n}}(\gammahat_4 \times \Rbb)$, where $\gammahat_4$ has a contact point with $\mutilde_n$ and is close to $\etahat_2$. This finish the proof of estimate of curvature.
\end{proof}

%

Now using Proposition \ref{cor3.4.3}, \ref{foliationcomplete},  \ref{pro3.4.6}  we have the main theorem

\begin{proof}[\rm \bfseries Proof of Theorem \ref{thm1}] The proof is similar to the one of Theorem \ref{thm3.3.1}. 
The  minimal surfaces $\Sigma_1$, $\Sigma_2$ and $\Sigma_3$ in Theorem \ref{thm1} will be constructed as a limit of a subsequence of the vertical translations of annulus $\Acal_n$ contains in $\Omega_n \times \Rbb$. We need to prove that the topology does not disappear at the limit. The idea is to prove that if its disappear, the limit is a simply connected graph on  $\Omega_\vc=\cup_{n \in \Nbb} \Omega_n$ which contains $\Omega_0$ and contradict the conclusion of Lemma
\ref{lem5.1}.

We construct the surface $\Sigma_1$ as the limit  of annuli contains in $\Omega_n \times \Rbb$ with $\cup_{n \in \Nbb}\gamma_i^n=\gammahat_i$ and $\Omega_\infty=\Omegahat$ (see Proposition \ref{cor3.4.3}). 
For the surface $\Sigma_2$, we fix a point $p_1 \in \gammahat_1$ and consider the sequence of geodesics $\gamma_1^n$ with endpoints $p_1$ and $q_1^n$ such that $\gamma_1' \subset \gamma_1^n \subset \gamma_1^{n+1}$ and $\cup _{n \in \Nbb}\gamma_1^n$ is an half geodesic with end point $p_1$ and ideal point $\hat{q}_1$. We consider a sequence of geodesics $\gamma_2^n$, such that $\cup_{n \in \Nbb}\gamma_2^n=\gammahat_2$ and get a sequence $ \Omega_n$  with $\Omega_\vc\subset \Omegahat$ a quadrilateral domain with 3 vertices $\hat{q}_1, \hat{p}_2, \hat{q}_2$ at infinity and $p_1 \in \gammahat_1$.

For the surface $\Sigma_3$, we fix a point $p_2  \in  \gammahat_2$  and consider the sequence of geodesics $\gamma_2^n$ with endpoints $p_2$ and $q_2^n$ such that $\gamma_2' \subset \gamma_2^n \subset \gamma_2^{n+1}$. We consider the same sequence of geodesics $\gamma_1^n$ as for the surface $\Sigma_2$, with end points $p_1$ and $q_1^n$. The sequence 
$\Omega_n$ bounded by $p_1, q_1^n, p_2, q_2^n$ contain $\Omega_0$, and $\Omega_\vc\subset \Omegahat$ in this case is a quadrilateral  domain with 2 vertices $\hat{q}_1,   \hat{q}_2$ at infinity and $p_1,  p_2$.

Let $\Acalcheck_n$ be the vertical translation of $\Acal_n$ such that $h^+(\Acalcheck_n)=0$ and $\Acalcheck_n$ is tangent to $\{z=0\}$ at the  point $p^+_n$. 
Uniform bound of the curvature imply that $\Acalcheck_n$ is locally a bounded graph with bounded slope on a disc of radius $r>0$, centered at $p^+_n$.  When $\Omega_n$ converge to $\Omegahat$, the sequence $p^+_n$ stays at distance $r$ of $\partial \Omegahat$. 
Then, the sequence of the minimal surfaces $\Acalcheck_n$ has an accumulation point at $p_\infty$ and $\Acalcheck_n$ has a subsequence converging to a minimal surface $\Acalcheck_\infty$ tangent to $\{z=0\}$ at $p_\infty$ by Proposition  \ref{pro3.4.6}, with multiplicity less than 2, by Proposition \ref{foliationcomplete}.

Assume  that $h^-(\Acalcheck_n) \to -\vc$ and consider $\Acalcheck^+_n(t)=\Acalcheck_{n} \cap \{ z \geq t\}$. By Proposition  \ref{cor3.4.3}, $\Acalcheck^+_{n}(0)$ has two connected components which are asymptotic to $\gammahat_i \times \Rbb$. By Lemma \ref{lem2.1.3}, the limit $\Acalcheck^+_{\vc}(t)$ has two connected components which stays at geodesic distance 
$\epsilon$  from $\gammahat_i \times \Rbb$ for  $t \geq d(\epsilon)>0$. By Proposition \ref{pro3.4.4}-3a) and Lemma \ref{lem2.1.3}, 
$\Acalcheck^-_{\vc}(0)= \Acalcheck_{\vc} \cap \{ z \leq 0\}$  has two connected components whose boundary is in $\{z=0\}\cup (\etahat_i \times \Rbb)$ and its boundary at infinity is in $\partial _{\vc}( \etahat_i \times \Rbb)$. Each of these connected components stays at distance $\epsilon$ from $\etahat_i  \times \Rbb$ for $t  \leq d(\epsilon)<0$ by Lemma \ref{lem2.1.3}.

Assume that $\Acalcheck_{\vc}$ is not a graph, then there exists a vertical plane $\gammahat \times \Rbb$ tangent at $p$ to $\Acalcheck_{\vc}$ and $\Gamma= (\gammahat \times \Rbb) \cap \Acalcheck_{\vc}$ separate the surface in at least four non compact connected components. We prove that there is one connected component which has its boundary in $\gammahat \times \Rbb$. Assume the contrary, each one of the connected component has one vertical line $\{\widehat{p}_i\} \times \Rbb$, $\{\widehat{q}_i\} \times \Rbb$ in its boundary at infinity or $\{p_i\} \times \Rbb$ for $\Sigma_2, \Sigma_3$. But, by the multiplicity one property and the asymptotic geometry of $\Acalcheck_{\vc}$, due mainly to Lemma \ref{lem2.1.3}, two vertical lines at the vertices of $\Omega_\vc$ located in one side of $\gammahat \times \Rbb$ are connected in the same component. We conclude that there is at most two connected components which contains the four vertical line located at the vertices, hence the two others non compact connected components have its boundary in $\gammahat \times \Rbb$ and its boundary at infinity  in $\partial _{\vc} \gammahat \times \Rbb$. By Lemma  \ref{lem2.1.3}-c), these connected components would be entirely contains in $\gammahat \times \Rbb$.

This prove with Proposition \ref{pro3.4.4} and arguments of \ref{pro3.3.15a} that $\Acalcheck_{\vc}$  is the vertical graph of a function $u:\Omega_{\vc} \to\Rbb$. 
where $u$ assumes the value $\pm\infty$ on $\partial \Omega_{\vc}$.
By Theorem \ref{thm4.2.6}  for vertical graph, $\Distaa{\Hbbh}{\gammahat_1,\gammahat_2}=\Distaa{\Hbbh}{\eta_1,\eta_2}$, contradicts Lemma \ref{lem5.1}. This prove that $h^-(\Acalcheck_n)$ converges to
a finite value and $h^-(\Acalcheck_\vc)$ is finite for the three sequences of domains $\Omega_n$ we consider above.

Now we describe the asymptotic geometry. 
Solutions $\tilde{u}^\pm_i$  on $\Omega_\vc$ such that $\tilde{u}^\pm_i=\pm\infty$ on $\partial \Omega_\vc \cap \gammahat_i$ and  $\tilde{u}^\pm_i=0$ on the rest of $\partial \Omega_\vc$
with $\{i,j\}=\{1,2\}$ define functions $u^+_\vc=\sup\{\tilde{u}^+_1,\tilde{u}^+_2\}$ and  $u^-_\vc=\inf\{\tilde{u}^-_1,\tilde{u}^-_2\}$.
Since $u_\vc^{\pm}$ are limit of the function $u_{n}^{\pm}$, they are  barriers for $\Acalcheck^+_{\vc}(t_1)$ and $\Acalcheck^-_{\vc}(t_2)$ for $t_1>0$ and $t_2 < h^-(\Acalcheck_{\vc})$. This barrier with Lemma  \ref{lem2.1.3}-b) imply that
for $t_1$ and $t_2$ large enough, $\Acalcheck^+_{\vc}(t_1)$ and $\Acalcheck^-_{\vc}(t_2)$ are horizontal graphs on $\gammahat_i \times \Rbb$ eventually bounded by vertical lines $\{p_i\} \times \Rbb$.

For a ball $B_R(0)$ large enough, the domain $\Omegahat \setminus B_R(0)$ has four (resp. three and two) connected component for $\Sigma_1$ (resp. for $\Sigma_2$ and $\Sigma _3$) bounded by $\gammahat_i $ and $\etahat_i$, such that any points of $\Omegahat  \setminus B_R(0)$ are at distance less than $\epsilon >0$ of the boundary $\partial \Omegahat$. 

This imply by uniform bounded curvature estimate that around each point in $(\Omegahat \setminus B_R(0)) \times \Rbb$, the surface 
$\Acalcheck_n$ is a local horizontal graph over a disc of radius $c>0$ in $\gammahat_i \times \Rbb$ , with unit normal vector converging uniformly to the unit normal vector to the vertical plane. In particular limit surfaces $\Sigma_1, \Sigma_2,\Sigma_3$ are uniform graph on  $\gammahat_i \times \Rbb$ in the domain $(\Omegahat \setminus B_R(0)) \times [t_2,t_1]$. This prove that the constructed surfaces $\Sigma_1, \Sigma_2,\Sigma_3$ are uniformly asymptotic to $\gamma_i \times \Rbb$ outside a compact set with boundaries at infinity four vertical lines located at the vertices of $\Omegahat$ for $\Sigma_1$,
three ideal vertical lines at infinity and $\{p_1\} \times \Rbb$ for $\Sigma_2$,  two ideal vertical lines at infinity and $\{p_1\} \times \Rbb$, $\{p_2\} \times \Rbb$  for $\Sigma_3$.

Finally, by a similar argument to the one in the proof of the Theorem \ref{thm3.3.1}, we obtain that 
$\Acalcheck_{\vc}$ is not simply connected and is topologically an annulus. This prove the Theorem.

\end{proof}

\section{Local calculus in $3$-Riemannian manifold}
\label{sectB}

In order to compute the mean curvature $ H $ of minimal graph in a local coordinates,  we need some local calculus in any $3$- Riemannian manifold, see for example \cite[Chapter 7 \S 1]{CM11}.
Equivalently, we consider $\Rbb^3$ endowed with a Riemannian metric $g$.
A frame of $\Rbb^3$ induced from the coordinates $(x_1,x_2,x_3)$ is $\dhro{x_1},\dhro{x_2},\dhro{x_3}$.
We denote by $g_{ij}=\vh{\dhro{x_i},\dhro{x_j}}$ the coefficients of the metric and
$g^{ij}$ the coefficients of its inverse matrix $\left(g_{ij}\right)_{i,j=1}^3$.
The  Christoffel symbols of the metric $g$ is denoted by $\Gamma^k_{ij}$ where $\nabla_{\dhr/\dhr x_i}\dhro{x_j}=\sum_{k=1}^{3}\Gamma^k_{ij}\dhro{x_k}$.
By definition, $g_{ij}$ and $\Gamma^k_{ij}$ are smooth functions on $\Rbb^3$.
We note that 
$\dhro{x_3}$ is a Killing vector field if and only if $g_{ij}$ does not depend on $x_3$ for all $i,j$,
so do the Christoffel symbols $\Gamma^k_{ij}$.
Let $\mathcal{U}\subset\Rbb^2$ be an open set and $v:\mathcal{U}\to\Rbb$ be a map of class $C^2$.
The graph of a function  $v$ (w.r.t.  the coordinates $(x_1,x_2,x_3)$) is the surface $\Gr(v)=\left\{(x_1,x_2,v(x_1,x_2)) ; (x_1,x_2)\in\mathcal{U} \right\}$ in $\Rbb^3$.
We need to determine the first, the second fundamental forms of $\Gr(v)$ and the condition for $v$ such that $\Gr(v)$ is minimal surface w.r.t.  the Riemannian metric $g$ of $\Rbb^3$.
The canonical frame $ \{E_1,E_2\} $ of $\Gr(v)$ is  
\begin{equation}\label{equ3.3.8}
E_i(x_1,x_2,v)=\Rest{\dhro{x_i}}{(x_1,x_2,v)}+\dhrou{v}{x_i}(x_1,x_2)\Rest{\dhro{x_3}}{(x_1,x_2,v)}=\sum_{j=1}^{3}T^j_i(x_1,x_2)\Rest{\dhro{x_j}}{(x_1,x_2,v)}
\end{equation}
where $T^j_i:\Gr(v)\to\Rbb$ is given by $T^j_i(x_1,x_2,v)=\delta^j_i+\delta^j_3\dhrou{v}{x_i}(x_1,x_2)$ and $\delta^j_i$ are Kronecker symbols.
Thus the first fundamental form $g^v$ of  $\Gr(v)$ in frame $(E_1,E_2)$ is 
\begin{align}
g^v_{ij}(x_1,x_2,v):=&\vh{E_i(x_1,x_2,v),E_j(x_1,x_2,v)}\notag\\
=&g_{ij}(x_1,x_2,v)+\dhrou{v}{x_i}(x_1,x_2)g_{j3}(x_1,x_2,v)+\dhrou{v}{x_j}(x_1,x_2)g_{i3}(x_1,x_2,v)\notag\\
&+\dhrou{v}{x_i}(x_1,x_2)\dhrou{v}{x_j}(x_1,x_2)g_{33}(x_1,x_2,v).\label{equ3.3.9}
\end{align}

We remark that the function $v$ and its partial derivatives $\dhrou{v}{x_i},\frac{\dhr^2v}{\dhr x_i\dhr x_j}$ are calculated at $(x_1,x_2)$ while $g_{ij},g^{ij},
T^j_i,\Gamma^m_{ij}$ and the vector fields $\dhro{x_i}$ are calculated at $(x_1,x_2,v)$. This remark  hold for the rest
of the section.

In order to determine the second fundamental form of $\Gr(v)$, we need to know its unit normal vector field.
Denote by $N=N_{\Gr(v)}$ the upward unit normal vector field of $\Gr(v)$.
We know that $N$ is the restriction to $\Gr(v)$ of the vector field $\frac{\nabla \Phi}{\chuan{\nabla\Phi}}$ where $\Phi:\mathcal{U}\times\Rbb\to\Rbb, (x_1,x_2,x_3)\mapsto x_3-v(x_1,x_2)$.
Since 
\begin{equation}
\vh{\nabla\Phi,\dhro{x_i}}(x_1,x_2,x_3)=\dhrou{\Phi}{x_i}(x_1,x_2,x_3)
=\begin{cases}
-\dhrou{v}{x_i}(x_1,x_2) & i=1,2,\\
1& i=3,
\end{cases}
\end{equation}
then, 
\begin{equation}\label{equ3.3.13}
\vh{N,\dhro{x_i}}(x_1,x_2,v)
=\begin{cases}
-\frac{1}{W}\dhrou{v}{x_i}(x_1,x_2)& i=1,2,\\
\frac{1}{W}& i=3
\end{cases}
\end{equation}
with $W(x_1,x_2)=\chuan{\nabla\Phi}(x_1,x_2,v)$.
Since  
$\chuan{\nabla\Phi}^2=\sum_{i,j=1}^{3}g^{ij}\vh{\nabla\Phi,\dhro{x_i}}\vh{\nabla\Phi,\dhro{x_j}}$ on $\mathcal{U}\times\Rbb$ so
\begin{equation}\label{equ3.3.12a}
W^2= g^{33}-2\sum_{i=1}^{2}g^{i3}\dhrou{v}{x_i}+\sum_{i,j=1}^{2}g^{ij}\dhrou{v}{x_i}\dhrou{v}{x_j}
\end{equation}
The second fundamental form of $\Gr(v)$ w.r.t. the frame $(E_1,E_2)$ is given by
\begin{equation}\label{equ3.3.14}
A_{ij}=A^{\Gr(v)}_{ij}=\vh{N,\nabla_{E_i}E_j}.
\end{equation}
It follows from the formula \eqref{equ3.3.8} of $E_i$ and from the definition of $T^j_i$ that
\begin{align*}
\left(\nabla_{E_i}E_j\right)(x_1,x_2,v)&
=\sum_{\ell=1}^{3}\nabla_{E_i}\left(T^\ell_j\dhro{x_\ell}\right)=\sum_{\ell=1}^{3}E_i\left(T^\ell_j\right) \dhro{x_\ell}+\sum_{\ell=1}^{3}T^\ell_j\nabla_{E_i}\dhro{x_\ell}\\
&=\sum_{\ell=1}^{3}\dhro{x_i}\left(\delta^\ell_j+\delta^\ell_3\dhrou{v}{x_j}\right)\dhro{x_\ell}+\sum_{k,\ell=1}^{3}T^k_iT^\ell_j\nabla_{\dhr/\dhr x_k}\dhro{x_\ell} \\
&=\frac{\dhr^2v}{\dhr x_i\dhr x_j}\dhro{x_3}+\sum_{k,\ell,m=1}^{3}T^k_iT^\ell_j\Gamma^m_{k\ell}\dhro{x_m}
\end{align*}
From this and the two formulas \eqref{equ3.3.13}, \eqref{equ3.3.14} we obtain
\begin{align}
A_{ij}(x_1,x_2,v)
&= \frac{\dhr^2v}{\dhr x_i\dhr x_j}\vh{N,\dhro{x_3}}+\sum_{k,\ell,m=1}^{3}T^k_iT^\ell_j\Gamma^m_{k\ell}\vh{N,\dhro{x_m}}\notag\\
&=\frac{1}{W}\left(\frac{\dhr^2v}{\dhr x_i\dhr x_j}+\sum_{k,\ell=1}^{3}T^k_iT^\ell_j\Gamma^3_{k\ell}\right)-\frac{1}{W}\sum_{m=1}^{2}\sum_{k,\ell=1}^{3} T^k_i T^\ell_j\Gamma^m_{k\ell}\dhrou{v}{x_m}\label{equ3.3.12}
\end{align}
From the formula of $T^j_i$, for each $m=1,2,3$, we have
\begin{align}
\sum_{k,\ell=1}^{3} T^k_i T^\ell_j\Gamma^m_{k\ell}
&=\sum_{k,\ell=1}^{3}\left(\delta^k_i+\delta^k_3\dhrou{v}{x_i}\right)\left(\delta^\ell_j+\delta^\ell_3\dhrou{v}{x_j}\right)\Gamma^m_{k\ell}\notag\\
&=\sum_{k,\ell=1}^{3}\left(\delta^k_i\delta^\ell_j+\delta^k_i\delta^\ell_3\dhrou{v}{x_j}+\dhr^k_3\delta^\ell_j\dhrou{v}{x_i}+\delta^k_3\delta^\ell_3\dhrou{v}{x_i}\dhrou{v}{x_j} \right)\Gamma^m_{k\ell}\notag\\
&= \Gamma^m_{ij}+\Gamma^m_{i3}\dhrou{v}{x_j}+\Gamma^m_{3j}\dhrou{v}{x_i}+\Gamma^m_{33}\dhrou{v}{x_i}\dhrou{v}{x_j}\label{equ3.3.16a}
\end{align}

We deduce the expression of the mean curvature $ H $ in local coordinates.
Since the mean curvature  $H$ of $\Gr(v)$ is given by
$2H=\sum_{i,j=1}^{2}(g^v)^{ij}A_{ij}$
where
$(g^v)^{ij}$ are the coefficients of the inverse of the matrix $\left(g^v_{ij}\right)_{i,j=1}^2$, it follows from the formulas of $A_{ij}$ and \eqref{equ3.3.12}, \eqref{equ3.3.16a} that

\begin{align}\label{equ4.5.4}
2HW
=&\sum_{i,j=1}^{2}(g^v) ^{ij}\left(\frac{\dhr^2v}{\dhr x_i\dhr x_j}+\Gamma^3_{ij}+\frac{\dhr v}{\dhr x_i}\Gamma^3_{3j}+\frac{\dhr v}{\dhr x_j}\Gamma^3_{i3}+\frac{\dhr v}{\dhr x_i}\frac{\dhr v}{\dhr x_j}\Gamma^3_{33}\right)\notag\\
&-\sum_{i,j,m=1}^{2}\frac{\dhr v}{\dhr x_m}(g^v) ^{ij}\left(\Gamma^m_{ij}+\frac{\dhr v}{\dhr x_i}\Gamma^m_{3j}+\frac{\dhr v}{\dhr x_j}\Gamma^m_{i3}+\frac{\dhr v}{\dhr x_i}\frac{\dhr v}{\dhr x_j}\Gamma^m_{33}\right).
\end{align}


Let $M^3$ be a Riemannian manifold and  $\Sigma^2$ be a submanifold of $M$.
Lets introduce Fermi coordinates $(x_1,x_2,x_3)$.
Let $N\Sigma\to \Sigma$ be the normal vector bundle of $\Sigma$ in $M$.
By identifying $\Sigma$ with the image of zero section of $N\Sigma$, we can see $\Sigma$ as a submanifold of $N\Sigma$.
The exponential map  $\exp_\Sigma^\bot:N\Sigma\to M$ fixes all points of $\Sigma$, moreover, it is local diffeomorphism on $\Sigma$.
Hence, each point of $\Sigma$ belongs to a local chart $(U,x)$ of $M$ where $(U\cap\Sigma, (x_1,x_2))$ is a local chart of $\Sigma$ and $x_3\mapsto (x_1,x_2,x_3)$ is a unit-speed geodesic of $M$  perpendicular to $\Sigma$.
Then in this  Fermi coordinates we have $g_{3i}=g_{i3}=\delta_{3i}$
for all $i=1,2,3$.
It follows that  for $i=1,2,3$
\begin{equation}\label{equ3.3.16}
	\Gamma^3_{3i}=\Gamma^3_{i3}=\vh{\nabla_{\dhr/\dhr x_i}\dhro{x_3},\dhro{x_3}}=\frac{1}{2}\dhro{x_i}g_{33}=0.
\end{equation}

We now show in the next paragraph that the coefficients of the inverse of the matrix $\left(g^v_{ij}\right)^2_{ i,j=1}$ satisfy the following equality
\begin{equation}\label{equ3.3.10}
\left(g^v\right)^{ij}=g^{ij}-\sum_{k,\ell=1}^{2}\frac{g^{ik}g^{j\ell}}{W^2}\dhrou{v}{x_k}\dhrou{v}{x_\ell}.
\end{equation}
For simplicity, we denote the matrix 
$\left(g^v_{ij}\right)^2_{ i,j=1},\left(g_{ij}\right)^2_{ i,j=1}$ and $\left(\dhrou{v}{x_i}\dhrou{v}{x_j}\right)^2_{ i,j=1}$ by respectively $G^v,\widehat{G}$ and $B$ and
the inverse $\widehat{G}^{-1} = \left(g^{ij}\right)^2_{ i,j=1}$.
From \eqref{equ3.3.9},
  we obtain the following equality of the matrix
$G^v=\widehat{G}+B$.
This yields
\begin{equation}\label{equ3.3.19a}
\left(G^v\right)^{-1}=\left(I+ \widehat{G}^{-1}B\right)^{-1}\widehat{G}^{-1}
\end{equation}
where $I$ is the $2$-by-$2$ identity matrix.
The equality \eqref{equ3.3.10}  will be attained from \eqref{equ3.3.19a} and a calculation of the inverse of $I+\widehat{G}^{-1}B$.
Denote by $C$ the $2$-by-$1$ matrix of the partial derivatives $\dhrou{v}{x_i}$ of $v$.
Hence, $B=CC^\mathsf{T}$ where $C^\mathsf{T}$ is the transpose of $C$.
It follows from \eqref{equ3.3.12a} that
$C^\mathsf{T} \widehat{G}^{-1}C=\sum_{i=1}^{2}g^{ij}\dhrou{v}{x_i}\dhrou{v}{x_j}=W^2-1$.
Hence
\begin{align*}
\left(\widehat{G}^{-1}B\right)^2&
=\left(\widehat{G}^{-1}CC^\mathsf{T}\right)\left(\widehat{G}^{-1}CC^\mathsf{T}\right)=\widehat{G}^{-1}C\left(C^\mathsf{T}\widehat{G}^{-1}C\right)C^\mathsf{T}\\
&=\left(W^2-1\right)\widehat{G}^{-1}CC^\mathsf{T}=\left(W^2-1\right)\widehat{G}^{-1}B
\end{align*}
So the inverse of $I+\widehat{G}^{-1}B$  is $I-\frac{1}{W^2}\widehat{G}^{-1}B$.
Replacing it to \eqref{equ3.3.19a} and comparing the coefficients of the matrix, we obtains \eqref{equ3.3.10}.

\medspace

Each level set $\left\{p\in U; x_3(p)=x_3 \right\}$ of the coordinate function $x_3:U\to\Rbb$ is an embedded surface whose unit normal vector, denoted by $N$, is the restriction to this level set of  $\dhro{x_3}$.
Since $\dhro{x_1},\dhro{x_2}$ is a frame of the level sets, the first fundamental form of each level set is $\left(g_{ij}\right)^2_{ i,j=1}$.
Define $H$ and $A$ respectively the mean curvature and the second fundamental form of level sets
of the coordinate function $x_3$.
Then $H$ and $A$ are calculated as follows
\begin{equation}\label{equ3.3.17}
A_{ij}=\vh{\nabla_{\dhr/\dhr x_i}\dhro{x_j},\dhro{x_3}}=\Gamma^3_{ij},\quad 2H=\sum_{i,j=1}^{2}g^{ij}A_{ij}=\sum_{i,j=1}^{2}g^{ij}\Gamma^3_{ij}
\end{equation}
The partial derivative of the mean curvature function $H:U\to\Rbb$ w.r.t.   $x_3$ is given by:
\begin{equation}\label{equ4.5.7a}
2\dhrou{H}{x_3}=\chuan{A}^2+\Ric(N,N).
\end{equation}
Indeed, let $S=\nabla N:V\mapsto \nabla_VN$ be the shape operator of level sets of the coordinate function $x_3$.
For all tangent vector fields $V,W$ of each level sets, we have $A(V,W)=-\vh{SV,W}$.
So
\begin{equation}\label{equ3.3.22a}
\tr (S)=-2H,\quad \tr \left(S^2\right)=\chuan{A}^2.
\end{equation}
By the radial curvature equation, see \cite[Theorem 2 page 44]{Pet06}, we have the equality
\begin{equation*}
\nabla_N S+ S^2=- R(\cdot,N)N. 
\end{equation*}
Taking trace of both sides of this equality, by using \eqref{equ3.3.22a}, we obtain the desired formula \eqref{equ4.5.7a}.

\appendix

\end{document}